\documentclass[11pt]{article}

\usepackage{amsmath,amsthm,amsfonts,latexsym}
\usepackage{amssymb}
\usepackage{mathabx}
\usepackage[latin1]{inputenc}

\begin{document}

\newtheorem*{theo}{Theorem}
\newtheorem*{pro} {Proposition}
\newtheorem*{cor} {Corollary}
\newtheorem*{lem} {Lemma}
\newtheorem{theorem}{Theorem}[section]
\newtheorem{corollary}[theorem]{Corollary}
\newtheorem{lemma}[theorem]{Lemma}
\newtheorem{proposition}[theorem]{Proposition}
\newtheorem{conjecture}[theorem]{Conjecture}

\theoremstyle{definition}
 \newtheorem{definition}[theorem]{Definition} 
  \newtheorem{example}[theorem]{Example}
   \newtheorem{remark}[theorem]{Remark}
   
\newcommand{\Naturali}{{\mathbb{N}}}
\newcommand{\Reali}{{\mathbb{R}}}
\newcommand{\Complessi}{{\mathbb{C}}}
\newcommand{\Toro}{{\mathbb{T}}}
\newcommand{\Relativi}{{\mathbb{Z}}}
\newcommand{\HH}{\mathfrak H}
\newcommand{\KK}{\mathfrak K}
\newcommand{\LL}{\mathfrak L}
\newcommand{\as}{\star_{\sigma}}
\newcommand{\tn}{\vert\hspace{-.3mm}\vert\hspace{-.3mm}\vert}
\def\A{{\mathcal A}}
\def\B{{\mathcal B}}
\def\E{{\mathcal E}}
\def\F{{\mathcal F}}
\def\H{{\mathcal H}}
\def\K{{\mathcal K}}
\def\L{{\mathcal L}}
\def\N{{\mathcal N}}
\def\M{{\cal M}}
\def\gM{{\frak M}}
\def\O{{\cal O}}
\def\P{{\cal P}}
\def\S{{\cal S}}
\def\T{{\cal T}}
\def\U{{\cal U}}
\def\C{{\mathcal C}}
\def\V{{\mathcal V}}
\def\X{{\mathcal X}}
\def\Y{{\mathcal Y}}
\def\qed{\hfill$\square$}
\def\veps{{\varepsilon}}

\title{Actions of Fell bundles, II: 
tensor products and Fourier-Stieltjes algebras}

\author{Erik B\'edos, Roberto Conti\\}
\date{\today}
\maketitle
\markboth{R. Conti, Erik B\'edos}{
}
\renewcommand{\sectionmark}[1]{}
\begin{abstract} 
 In the context of Fell bundles over discrete groups,
 we introduce and analyze 
suitable notions of direct sum and tensor product of actions of Fell bundles on  Hilbert bundles. 
For $C^*$-dynamical systems, where actions arise from equivariant representations, we discuss how the tensor product of two such actions is related to the action associated with the tensor product of the underlying equivariant representations. Likewise, since any action gives rise to a $C^*$-correspondence between the full cross-sectional $C^*$-algebras of the involved Fell bundles, we examine the connection between the $C^*$-correspondence associated with the tensor product  of two given actions and the tensor product of the associated $C^*$-correspondences.
We also discuss how the tensor product of actions naturally appears in Rieffel-like induction processes. 
 Finally we present related results on the Fourier-Stieltjes and Fourier algebras of a Fell bundle. 
 
\end{abstract}



\section{Introduction}
The theory of Fell bundles (originally called $C^*$-algebraic bundles in \cite{FD1, FD2}) over locally compact groups
may be thought of as providing a unifying setting for discussing various concepts and constructions about groups and related objects,
such as group $C^*$-algebras, group representations, $C^*$-dynamical systems and associated crossed products,  induction processes, etc. 
To avoid many technicalities, 
we will assume throughout this paper that all groups are discrete.   
Taking inspiration from the notions of unitary representation of a group on a Hilbert space and representation of a $C^*$-algebra on a Hilbert $C^*$-module, 
we introduced in a previous article \cite{BeCo8} the concept of  $(\A,\varphi)$-action of a Fell bundle $\A$ over a group $G$ on a Hilbert $\B$-bundle, where $\B$ is a Fell bundle over a group $H$ and $\varphi : G\to H$ is a homomorphism (when $H =G$ and $\varphi$ is the identity map, it is convenient to simply call it
an $\A$-action on a Hilbert $\B$-bundle).  A long list of examples was then exhibited in \cite{BeCo8}, providing evidence that such a concept fits well with many other 
structural aspects scattered in the literature about Fell bundles and their associated cross-sectional $C^*$-algebras.
Now, in order to broaden the scope of applications, it is natural to discuss operations 
like direct sum and tensor product of actions of Fell bundles. While the treatment of an arbitrary direct sum of $(\A,\varphi)$-actions does not pose any significant difficulty (see Section \ref{DirectSum}), the case of tensor product of two such actions looks definitely more problematic.
The main technical point of this paper is to show that it is possible to form a kind of tensor product action under suitable compatibility assumptions. 

To be more precise, let $G, H, K$ be groups, $\A=(A_g)_{g\in G}$, $\B=(B_h)_{h\in H}$, and $\C=(C_k)_{k\in K}$ be Fell bundles
 and $\varphi : G\to H$ and $\psi : H\to K$ be homomorphisms. Consider an  $(\A, \varphi)$-action  $\rho$ on some  Hilbert $\B$-bundle $\X=(X_h)_{h\in H}$ and  
 a $(\B,\psi)$-action $\sigma$ on some  Hilbert $\C$-bundle $\Y = (Y_k)_{k\in K}$. As we show in subsection \ref{tp-general}, one may then 
  create an $(\A,\psi\circ \varphi)$-action $\rho \dot\boxtimes \sigma$ on a certain Hilbert $\C$-bundle $\X\boxtimes_{\sigma} \Y$. 
 The construction of $\X\boxtimes_{\sigma} \Y$ (and its relatives discussed below)
is closely connected with the one presented by Abadie and Ferraro in \cite[Section 5.1]{AF} in the more restrictive context of equivalence of Fell bundles over groups and Morita-Rieffel equivalence of the associated cross-sectional $C^*$-algebras. 
 In the case of a Fell bundle $\B$ over $G$ and two $\B$-actions $\rho$ and $\sigma$ on some  Hilbert $\B$-bundles $\X$ and $\Y$, one can actually also define a $\B$-action $\rho\dot\otimes\sigma$ on a Hilbert $\B$-bundle $\X\otimes_\sigma\Y$, which appears to be unitarily equivalent to $\rho \dot\boxtimes \sigma$. Of course, in this situation, one may also form a $\B$-action $\sigma\dot\otimes \rho$ on the Hilbert $\B$-bundle $\Y\otimes_\rho \X$, 
having no obvious a priori relationship with  $\rho\dot\otimes\sigma$.  
 As there are a number of technical steps, it is quite instructive to deal with this simpler setting first, as we do in subsection \ref{tp-b-actions}.   
 When $\B$ is the Fell bundle arising from a unital $C^*$-dynamical system, the concept of $\B$-action is related to the notion of equivariant representation of the system, cf.~\cite{BeCo8}.  We show in Theorem \ref{dynsyst-emb} that the tensor product of two actions associated with some equivariant representations contains in a suitable sense 
 (many copies of)
 the action associated with the tensor product of these equivariant representations (as defined in \cite{BeCo6}), up to a perturbation by some kind of 1-cocycle (see subsection \ref{tp-equiv} for some details about this notion). 
 
 Besides the intrinsic interest of the tensor product as a way of producing new actions, we exemplify its occurrence in various situations. 
We show first in Theorem \ref{corr-morph} that given two actions $\rho$ and $\sigma$ as above,  there exists a morphism 
from the (internal) tensor product of the $C^*$-correspondences associated with $\rho$ and $\sigma$ in \cite[Section 6.1]{BeCo8} to 
the $C^*$-correspondence associated with $\rho\dot\boxtimes \sigma$. 
Next, we consider Fell bundles  $\A$ and $\B$ over $G$, an $\A$-action $\rho$ on a  Hilbert $\B$-bundle $\X$, and   a
$*$-representation $\pi$ of $\B $ on some Hilbert $C$-module $Y$.
Analogously to Rieffel's general construction of induced $*$-representations of $C^*$-algebras, as presented for example in \cite[Section 2.4]{RW}, $\rho$ and $\pi$ induce in a natural way a $*$-representation ${\rm Ind}_\rho(\pi)$ of $\A$ on a certain Hilbert $C$-module (see subsection 5.2).  
On the other hand, one may also associate to $\pi$ a $\B$-action $\rho_\pi$ on a certain Hilbert $\C$-bundle, cf.~\cite[Example 4.15]{BeCo8}, where $\C$ is the trivial Fell bundle over $G$ associated with $C$. We show in Theorem \ref{Ind-prop} that the $\A$-actions $\rho_{{\rm Ind}_\rho(\pi)}$ and $\rho\dot\boxtimes\rho_\pi$ are  unitarily equivalent.
 Further, when  $\A$ is a Fell bundle over $G$,  $\pi$ is a $*$-representation of $\A $ on a  Hilbert $B$-module $X$,
  and $\pi'$ is a $*$-representation of $B$ on a Hilbert $C$-module $Y$, there is a naturally induced $*$-representation
 $\pi''$ of $\A$ on the Hilbert $C$-module $X\otimes_{\pi'} Y$. Using Theorem \ref{Ind-prop} we deduce in  Proposition \ref{rhopitensrhopi'} that the $\A$-action associated to $\pi''$ is unitary equivalent to a  tensor product action of the form $\rho_\pi\dot\boxtimes\rho_{\pi'}$, where $\rho_{\pi'}$ is naturally defined.
 Finally, we describe in subsection \ref{RieffelInd} how the tensor product construction may be used to produce a Rieffel-like induction of actions of Fell bundles via conditional expectations, generalizing Rieffel's induction procedure \cite{Rie74} for $*$-representations of $C^*$-algebras on Hilbert $C^*$-modules via conditional expectations. Our setting is a Fell bundle $\A$ over $G$ having  a sub-Fell bundle $\B$ with a conditional expectation $\E : \A \to \B$ as in \cite{Exel2}, and a $\B$-action $\sigma$ on a Hilbert $\C$-bundle $\Y$, $\C$ being also a Fell bundle over $G$. We may then define  an $\A$-action ${\rm Ind}_\E(\sigma)$ induced by $\sigma$ via $\E$ on a certain Hilbert $\C$-bundle, of the form ${\rm Ind}_\E(\sigma)= \rho_\E \dot\boxtimes \sigma$ where $\rho_\E$ is the $\A$-action on a Hilbert $\B$-bundle defined in \cite{BeCo8}. 
 Given a subgroup $H$ of $G$, this procedure may be applied to induce $\A$-actions from certain actions of the restriction of $\A$ to $H$.
 
 Another aim of the present work has been to initiate a more thorough study of the Fourier-Stieltjes algebra $B(\B)$ of a unital Fell bundle $\B$ over a group $G$, that was succinctly introduced in \cite[Remark 3.17]{BeCo7} as the span of all so-called positive definite $\B$-bundle maps.
 In harmonic analysis, the Fourier-Stieltjes  algebra $B(G)$ of  $G$ is defined as the span of all positive definite complex functions on $G$, or, equivalently, as the space of all coefficient functions of unitary representations of $G$ on Hilbert spaces, cf.~\cite[D\'efinition 2.2]{Eym}, \cite[Definition 2.1.5]{KL}. We show in Theorem \ref{FS-coeff} that a similar characterization holds for $B(\B)$, by replacing unitary representations of $G$ with $\B$-actions on Hilbert $\B$-bundles.  This may also be seen as a generalization of the analogous result for 
twisted unital $C^*$-dynamical systems, cf.~\cite[Definition 3.1, Corollary 4.4]{BeCo6}. As in the classical case, one may define a norm on $B(\B)$, 
making it a unital Banach algebra, cf.~Theorem \ref{FS-Banach}. This partly relies on the fact that infinite direct sums of actions are well-behaved, as we explain in Section \ref{DirectSum}.
We also show in Proposition \ref{embFS(G)} that $B(G)$ can be continuously embedded in $B(\B)$. Finally, we 
define the Fourier algebra  $A(\B)$ as the closure in $B(\B)$ of the span of all finitely supported positive definite $\B$-bundle maps, and show that it coincides with the closure of the set of all finitely supported elements in $B(\B)$ (as in the case of a unital $C^*$-dynamical system, cf.~\cite{Ravn}).  

\section{Preliminaries} \label{prelim}

Let $A$ be a $C^*$-algebra. We will denote by $A^+=\{ a \in A: a\geq 0\}$ the cone of positive elements in $A$. By a Hilbert $A$-module $X$ we will always mean a \emph{right} Hilbert $A$-module, as defined in \cite{La1}, and $\L_A(X)$ will then denote the $C^*$-algebra consisting of all adjointable linear maps from $X$ to itself. When $B$ is a $C^*$-algebra and $A$ acts on the left of a Hilbert $B$-module $X$ by adjointable operators (i.e.,  there exists a $*$-representation of $A$ on $X$, meaning that there exists a $*$-homomorphism of $A$ into $\L_B(X)$), we will say that $X$ is a $C^*$-correspondence over $A$ and $B$, or an $A$-$B$ correspondence. Note that we do not assume that $A$ acts non-degenerately on $X$, as some authors do.  

Throughout this paper, $G$ and $H$ will always be discrete groups. By a slight abuse of notation, we will denote the unit of any group by  $e$. If $X$ is a complex vector space,  then $C_c(G, X)$ will denote the vector space of all finitely supported functions from $G$ to $X$ (with pointwise operations).

We will be considering a Fell bundle  $\A=(A_g)_{g \in G}$ over $G$. Our notation and terminology will be as in Exel's book \cite{Exel2}, hence as in \cite{BeCo7} and \cite{BeCo8}. 
 In particular, $\A$ will also denote the disjoint union of the $A_g$'s. We recall that the unit fibre $A_e$ is a $C^*$-algebra and that
$C_c(\A)$ denotes the $*$-algebra consisting of all finitely supported functions $f:G\to \A$ such that $ f(g) \in A_g \text{ for all } g \in G$, the product and involution being given by
\[ (f_1 \star f_2) (h) = \sum_{g\in G} f_1(g) f_2(g^{-1}h), \quad f^*(h) = f(h^{-1})^*\]
for all $h \in G$. 
A $*$-representation $\pi=(\pi_g)_{g\in G}$ of $\A$ in some $*$-algebra $C$ is a family of linear maps $\pi_g:A_g\to C$ satisfying that \[\pi_g(a)\pi_{g'}(a')= \pi_{gg'}(aa') \quad \text{ and }\quad \pi_g(a)^*= \pi_{g^{-1}}(a^*)\] for all $g, g'\in G$ and $a \in A_g, a'\in A_{g'}$. It induces a $*$-homomorphism $\phi_\pi:C_c(\A) \to C$, often called the integrated form of $\pi$,  given by 
\[\phi_\pi(f) = \sum_{g\in G} \pi(f(g)) \quad \text{for all } f \in C_c(\A).\]
For $g \in G$, we let $j^\A_g: A_g \to C_c(\A)$ denote the canonical injection. When $a\in A_g$ we will often write $a\odot g$ instead of $j^\A_g(a)$. 

The full cross-sectional $C^*$-algebra $C^*(\A)$ is defined as the $C^*$-completion of $C_c(\A)$ with respect to the universal $C^*$-norm  $\|\cdot\|_{\rm u}$ given by \[\|f\|_{\rm u} 
=  \sup\big\{p(f): p \text{ is a } C^*\text{-seminorm on } C_c(\A)\big\} \leq \sum_{g\in G} \|f(g)\|.\]
We let $\kappa^\A:C_c(\A)\to C^*(\A)$ 
denote the canonical injection and  let $\widehat j^\A= (\widehat j^\A_g)_{g\in G}$ be the $*$-representation of $\A$ in $C^*(\A)$ given by $\widehat j^\A_g := \kappa^\A\circ j^\A_g: A_g \to C^*(\A)$  for each $g\in G$. As shown in \cite[Proposition 17.9, (iv)]{Exel2}, each $\widehat j^{\A}_g$ is an isometric map. Also, $\kappa^\A$ is the integrated form of $\widehat j^\A$, i.e.,
\[\kappa^\A(f) =  \sum_{g\in G} \widehat j^\A_g(f(g))  \quad \text{for all } f\in C_c(\A).\] 

We will make use of the matrix $C^*$-algebras associated to a Fell bundle $\A$ by Abadie and Ferraro in \cite{AF}. If $g_1, \ldots, g_n \in G$ and $\mathbf{g}:= (g_1, \ldots, g_n) \in G^n$, 
then
\[ M_\mathbf{g}(\A) := \big\{R=[r_{ij}]\in M_n(\A) : r_{ij} \in A_{g_i^{-1}g_j} \text{ for all } i, j =1, \ldots, n\big\}\]
is a $*$-algebra with respect to the natural operations, 
which can be equipped with a norm turning it into a $C^*$-algebra, 
cf.~\cite[Lemma 2.8]{AF}. As in \cite{BeCo7}, we will represent  $M_\mathbf{g}(\A)$ by adjointable operators on a  Hilbert $A_e$-module 
in the following way.
Consider $A_\mathbf{g}:= A_{g_1^{-1}}\oplus \cdots \oplus A_{g_n^{-1}}$ as the  Hilbert $A_e$-module obtained by taking the direct sum of the $A_{g_i}$'s (considered as  Hilbert $A_e$-modules), whose inner product is given by
\[ \big\langle (a_1, \ldots, a_n), (a'_1, \ldots, a'_n)\big\rangle_{A_e}  = \sum_{i=1}^n a_i^*a'_i\]
for $a_i, a'_i \in A_{g_i^{-1}}$, $i=1, \ldots, n$.
Then each $R=[r_{ij}] \in M_\mathbf{g}(\A)$ acts on $A_\mathbf{g}$ by  
\[L_R(a_1, \ldots, a_n) = (a'_1, \ldots, a'_n), \ \text{ where } a'_i := \sum_{j=1}^n r_{ij}a_j \text { for each }i = 1, \ldots, n. \] 
It is easy to check that the operator $L_R : A_\mathbf{g}\to A_\mathbf{g}$ is adjointable with $(L_R)^*=L_{R^*}$, and that the map $R \mapsto L_R$ from $M_\mathbf{g}(\A)$ into  $\mathcal{L}_{A_e}(A_\mathbf{g})$ is an injective $*$-homomorphism. 
Thus, the norm on $M_\mathbf{g}(\A)$ satisfies that $\|R\|= \|L_R\|$. 

We recall some other relevant definitions from \cite{BeCo7}.
Let $\varphi: G \to H$ be a group-homomorphism, i.e., $\varphi\in {\rm Hom}(G,H)$ and $\B=(B_h)_{h\in H}$ be Fell bundle over $H$.  

An \emph{$\A$-$\varphi$-$\B$ bundle map} is a family of maps $T = (T_g)_{g\in G}$  such that $T_g: A_g \to B_{\varphi(g)}$ is linear and bounded for every $g\in G$. 
When $G=H$ and $\varphi = {\rm id}_G$, so $\A$ and $\B$ are both Fell bundles over $G$, we call $T$ an $\A$-$\B$ \emph{bundle map}
  Moreover, if $\A=\B$, we call $T$ a \emph{$\B$-bundle map}.

When $T = (T_g)_{g\in G}$ is an $\A$-$\varphi$-$\B$ bundle map, the map $\phi_T: C_c(\A) \to C_c(\B)$ given by
\[ \phi_T\big(f)= \sum_{g\in G} j^\B_{\varphi(g)}(T_g(f(g)))\quad \text{for all } f\in C_c(\A)\]
is linear. We say that $T$ is \emph{full} if $\phi_T$ is bounded w.r.t.~the universal $C^*$-norms on $C_c(\A)$ and $C_c(\B)$, and denote the extension of $\phi_T$ to a linear bounded map from  $C^*(\A)$ into $C^*(\B)$ by $\Phi_T$.
 
An $\A$-$\varphi$-$\B$ bundle map $T = (T_g)_{g\in G}$ is said to be \emph{positive definite} if
\begin{equation*}
%\label{Bposdef}
\sum_{i, j=1}^n b_i \, T_{g_i^{-1}g_j} (a_i^* a_j) \, b_j^* \, \in  (B_e)^+
\end{equation*}
for all $n\in \Naturali$, $g_1, \ldots, g_n \in G$ and all  $a_i \in A_{g_i}$, $b_i\in B_{\varphi(g_i)}$, $i=1, \ldots, n$.
Equivalently, cf.~\cite[Remark 3.4]{BeCo7}, for every  $n\in \Naturali$, $\mathbf{g}= (g_1, \ldots, g_n) \in G^n$ and $a_i \in A_{g_i}$, $i=1, \ldots, n$, 
  the matrix $\big[ T_{g_i^{-1}g_j} (a_i^* a_j)\big]$ is positive in the $C^*$-algebra $M_{\varphi(\mathbf{g})}(\B)$, where $\varphi(\mathbf{g}):= (\varphi(g_1), \ldots, \varphi(g_n)) \in H^n$. 

The following result is \cite[Theorem 3.11]{BeCo7}: 

\begin{theorem} \label{PDCP}
Let $\A=(A_g)_{g \in G}$ and $\B=(B_h)_{h \in H}$ be Fell bundles and let $T=(T_g)_{g\in G}$ be an $\A$-$\varphi$-$\B$ bundle map, where $\varphi \in {\rm Hom}(G,H)$.
Then  $T$ is positive definite if and only if $T$ is full and  $\Phi_T: C^*(\A) \to C^*(\B)$ is completely positive.
\end{theorem}

Next, we recall the definition of a Hilbert $\B$-bundle.
\begin{definition}[\cite{AF},\cite{ABF1}]\label{rHilbertBbundle} Let $\B=(B_g)_{g\in G}$ be a Fell bundle over $G$.
 A (\emph{right}) \emph{Hilbert $\B$-bundle}\footnote{We will not consider \emph{left} Hilbert $\B$-bundles in this paper. Hence, 
 by a Hilbert $\B$-bundle, we will always mean a right Hilbert $\B$-bundle.}  is a family of complex Banach spaces $\X = (X_r)_{r\in G}$ indexed over $G$
such that if $\X$ also denotes the disjoint union of the $X_r$'s,  we have maps \[(x, b) \mapsto x b : \X \times \B \to \X, \, \text{ and }\, (x, y) \mapsto \langle x, y\rangle_\B : \X \times \X \to \B,\] satisfying the following conditions: 
\begin{itemize}
\item[(1)] For each $r, s\in G$, we have $X_rB_s \subseteq X_{rs}$, and the map  $\,(x,b)\mapsto xb$ from $X_r \times B_s$ into $X_{rs}$ is bilinear.
\item[(2)] For each $r, s\in G$, we have  $\langle X_r,X_s\rangle_\B \subseteq B_{r^{-1}s}$ and, for each $x\in X_r$,  the map $  \, y\mapsto \langle x,y\rangle_\B$ from $X_s$ into $B_{r^{-1}s}$ is linear. 

\smallskip \noindent (Hence, in particular, we have $\langle X_r,X_r\rangle_\B \,\subseteq \,B_e$ for every $r\in G$.)
\item[(3)] $\langle x,yb\rangle_\B = \langle x,y\rangle_\B\, b$  and $(\langle x,y\rangle_\B)^* = \langle y,x\rangle_ \B$ for all $x,y \in \X$ and $b \in \B$.
\item[(4)] $\langle x,x\rangle_\B \geq  0$ in $B_e$ for all $x \in \X$, and $\langle x,x\rangle_\B = 0$ implies $x = 0$. 
\item[(5)] 
For each $r\in G$, the norm $x\mapsto \|x\|$ on $X_r$ satisfies that $\|x\|=\| \langle x,x\rangle_\B \|^{1/2}$ for all $x\in X_r$.
\end{itemize} 
Note that for each $r \in G$, conditions (1)-(4) imply that $X_r$ is an  inner product $B_e$-module, hence that 
$x\mapsto \| \langle x,x\rangle_\B \|^{1/2}$ gives a norm on $X_r$ (cf.~\cite{La1}).  
Since $X_r$ is complete by assumption, it follows from 
condition (5) that $X_r$ is a  Hilbert $B_e$-module. 
Moreover,  it can be shown  (cf.~\cite[Lemma 2.7]{AF}) that for all $x, y\in \X$ and $b,c\in \B$, we  also have
\begin{itemize}
\item[(a)] \,$\langle xb,y\rangle_\B =b^*\langle x,y\rangle_\B$,
\item[(b)] \, $\|xb\| \leq \|x\| \,\|b\|$,
\item[(c)] \, $\|\langle x,y\rangle_\B\| \leq \|x\| \,\|y\|$, 
\item[(d)] \, $(xb)c = x(bc)$.
\end{itemize}
\end{definition}

\begin{remark} \label{pre_H} 
If $\X = (X_r)_{r\in G}$ is a family of 
(complex) vector 
spaces equipped with maps $\X \times \B \to \X$ and  $\X \times \X \to \B$  satisfying conditions (1)-(3) and 
$\langle x,x\rangle_\B \geq  0$ in $B_e$ for all $x \in \X$ (where $\X$ denotes the disjoint union of the $X_r$'s), then $\X$ will be called a 
\emph{semi-inner product $\B$-bundle}.
The map $x\mapsto \|x\|:= \| \langle x,x\rangle_\B \|^{1/2}$ is then a seminorm on each $X_r$, and the proof of \cite[Lemma 2.7]{AF} makes it clear that properties (a)-(d) above still hold. After a separation and completion procedure,  one obtains from $\X$ a  Hilbert $\B$-bundle $\widetilde{\X}$ (see \cite[Proposition 5.1]{AF} for a similar result). If the map $x\mapsto \|x\|$ is already a norm on each $X_r$, that is, if (4) also holds,  then $\X$ is called an \emph{inner product $\B$-bundle}, and only completion is needed to construct $\widetilde{\X}$.
\end{remark} 

In the rest of this section, we consider a Fell bundle   $\A=(A_g)_{g\in G}$ over $G$, a Fell bundle $\B=(B_h)_{h\in H}$ over 
 $H$ and  $\varphi \in {\rm Hom}(G,H)$. 
\begin{definition}[\cite{BeCo8}]\label{BC-action-def}
 Let $\X=(X_h)_{h\in H}$ be a  Hilbert $\B$-bundle. Moreover, let $\F(\X) $ be the semigroup of maps from $\X$ into itself (w.r.t.~composition). A map $\rho: \A \to \F(\X)$ will be called an 
\emph{$(\A, \varphi)$-action on $\X$} 
when it satisfies the following conditions:
\begin{itemize}
\item[(i)] for each $g\in G$ and $h\in H$, we have $\rho(A_g)X_h \subseteq X_{\varphi(g)h}$, and the map
$(a, x) \mapsto \rho(a)x $ from $A_g \times X_h$ into $X_{\varphi(g)h}$ is bilinear;
\item[(ii)] $\rho(aa') = \rho(a)\rho(a')$ for all $a, a'\in \A$;
\item[(iii)] $\langle \rho(a) x, y\rangle_\B = \langle x, \rho(a^*) y\rangle_\B$ for all $a \in \A$ and $x, y\in \X$;
\item[(iv)]  $(\rho(a)x)b = \rho(a)(x b)$ for all $a\in \A$, $b\in \B$ and $x \in \X$.
\end{itemize}
When $G=H$ and $\varphi={\rm id}_G$, we will just say that $\rho$ is an 
\emph{$\A$-action} on the  Hilbert $\B$-bundle $\X$. 
\end{definition}

\begin{remark} \label{pre-act-def} One may define an $(\A, \varphi)$-\emph{pre-action} $\rho_0$ on a semi-inner product $\B$-bundle $\X^0= (X_h^0)_{h\in H}$ to be a map $\rho_0: \A\to \F(\X^0)$ satisfying the analogue of conditions  (i)-(iv) in Definition \ref{BC-action-def} and also that the inequality
\[ \langle \rho_0(a)x, \rho_0(a) x\rangle_\B \, \leq \, \|a\|^2  \langle x, x\rangle_\B\]
holds for all $a\in \A$ and $x \in \X^0$. Then it is not difficult to check that $\rho_0$ naturally induces an $(\A, \varphi)$-action $\rho$ on the Hilbert $\B$-bundle $\X$ obtained as the Hausdorff completion of $\X^0$.
\end{remark}

\begin{remark} \label{action-adjoint} 
We note that condition (iii) implies that if $a\in A_e$ and $h\in H$, then the restriction $\rho_h(a)$ of $\rho(a)$ to $X_h$ 
is a linear map from the Hilbert $B_e$-module $X_h$ into itself
which is adjointable, with $\rho_h(a)^* = \rho_h(a^*)$. One readily sees that the map $\rho_h:A_e \to \L_{B_e}(X_h)$ mapping each $a\in A_e$ to $\rho_h(a)$ is a $*$-homomorphism. 
 \end{remark} 
 
\begin{lemma}\label{Lance-like}
  Let $\X = (X_h)_{h\in H}$ be a  Hilbert $\B$-bundle and $\rho$ be an $(\A, \varphi)$-action on $\X$.
  Let $g\in G$,  $a\in A_g$, $\mathbf{g}= (g_1, \ldots, g_n) \in G^n$  and $x_j \in X_{\varphi(g_j)}$ for $j =1, \ldots, n$.  

\smallskip \noindent Set $\varphi(\mathbf{g}):= (\varphi(g_1), \ldots, \varphi(g_n))\in H^n$, $R := \big[ \langle x_i, x_j\rangle_\B\big]$ and $S:= \big[ \langle \rho(a) x_i, \rho(a) x_j\rangle_\B\big]$. 

\smallskip Then  $R, S\in M_{\varphi(\mathbf{g})}(\B)^+$ and 
 \[S \, \leq \|a\|^2\, R.\]
\end{lemma} 
\begin{proof}
 We know that this result holds when $g=e$ and $a \in A_e$, cf.~\cite[Lemma 5.5]{BeCo8}. The general case can be deduced as follows. We first note that 
for every $i,j$ we have 
\[ \langle \rho(a) x_i, \rho(a) x_j\rangle_\B \in \langle  X_{\varphi(g)\varphi(g_i)}, X_{\varphi(g)\varphi(g_j)}\rangle_\B \subseteq B_{(\varphi(g)\varphi(g_i))^{-1}\varphi(g)\varphi(g_j)}
= B_{\varphi(g_i)^{-1}\varphi(g_j)},\]
so $R, S \in M_{\varphi(\mathbf{g})}(\B)$, and $R, S\geq 0$ by \cite[Lemma 3.8]{BeCo8}.

Next, since $a^*a \in (A_e)^+$, we can find some $c\in A_e$ such that $a^*a = c^*c$. Then for every $i, j$ we have
\[   \langle \rho(a) x_i, \rho(a) x_j\rangle_\B =  \langle  x_i, \rho(a^*a) x_j\rangle_\B =  \langle  x_i, \rho(c^*c) x_j\rangle_\B =  \langle \rho(c) x_i, \rho(c) x_j\rangle_\B.\]
Thus we get 
\[S= \big[ \langle \rho(c) x_i, \rho(c) x_j\rangle_\B\big] \leq \|c\|^2 \, R = \|a\|^2 \, R.\]
\end{proof}

\begin{remark} \label{AF-action}  Let $\X = (X_h)_{h\in H}$ be a  Hilbert $\B$-bundle, $\rho$ be an $(\A, \varphi)$-action on $\X$ and $\mathbf{g}= (g_1, \ldots, g_n) \in G^n$. For later use, we mention that
one may construct an action of  the Abadie-Ferraro $C^*$-algebra $M_{\mathbf{g}}(\A) $ on a certain Hilbert $B_e$-module.
Indeed, set $\varphi(\mathbf{g}):= (\varphi(g_1), \ldots, \varphi(g_n))$
and  $\mathbf{X}_{\varphi(\mathbf{g})}:=\bigoplus_{i=1}^n X_{\varphi(g_i)^{-1}}$ (considered as a Hilbert $B_e$-module). 

Then it is easy to check that we get a $*$-homomorphism $L^{\rho}:  
M_{\mathbf{g}}(\A) \to \mathcal{L}_{B_e}(\mathbf{X}_{\varphi(\mathbf{g})})$ by setting
\[ L^\rho(R)(x_1, \ldots, x_n):= (x'_1, \ldots, x'_n), \ \text{ where } x'_i := \sum_{j=1}^n \rho(r_{ij})x_j \in X_{\varphi(g_i)^{-1}} \text { for each }i = 1, \ldots, n, \] 
for each $R=[r_{ij}] \in M_{\mathbf{g}}(\A)$ and $(x_1, \ldots, x_n) \in \mathbf{X}_{\varphi(\mathbf{g})}$.
\end{remark}

\medskip  
Finally, we recall that an important feature of
actions of Fell bundles is that they give rise to positive definite bundle maps in the following way (cf.~\cite[Proposition 5.3]{BeCo8}). 
Let  $\rho $ be an $(\A, \varphi)$-action  on a  Hilbert $\B$-bundle $\X$ and $x \in X_e$. For each $g\in G$, 
let $T_g: A_g \to B_{\varphi(g)}$ be given by \[T_g(a):= \langle x, \rho(a) x\rangle_\B\] for every $a\in A_g$. Then  
$T := (T_g)_{g\in G}$  is a positive definite $\A$-$\varphi$-$\B$ bundle map.
  
  \section{Direct sums of actions of Fell bundles} \label{DirectSum}
 Let $\A=(A_g)_{g\in G}$ and $\B=(B_h)_{h\in H}$ be Fell bundles over $G$ and $H$, and $\varphi:G\to H$ be a homomorphism. There is a natural way of forming the direct sum of $(\A, \varphi)$-actions on Hilbert $\B$-bundles.
  We consider first finite direct sums. 
  
 Let $\X^j = (X^j_g)_{g\in G}$ be Hilbert $\B$-bundles for $j=1, \ldots, n$. 
  For each $g\in G$, let $\bigoplus_{j=1}^n X^j_g$ be the direct sum of $X_g^1, \ldots, X_g^n$, considered as vector spaces, and equip $\bigoplus_{j=1}^n \X^j 
 :=\coprod_{g\in G}\big(\bigoplus_{j=1}^n X^j_g\big)$ (disjoint union) with the operations
 \[(\oplus_{j=1}^n \,x_j)\cdot b = \oplus_{j=1}^n \, (x_j\cdot b), \quad \big\langle \oplus_{j=1}^n x_j, \oplus_{j=1}^n \,y_j\big\rangle_\B = \sum_{j=1}^n \,\langle  x_j,  y_j \rangle_\B,\]
 where $b\in \B$, $g, g'\in G$,  $x_j\in X_g^j$ and $y_j\in X_{g'}^j$ for all $j=1, \ldots, n$. One readily verifies that  each $\bigoplus_{j=1}^n X^j_g$ is complete with respect to the norm 
 \[\|\oplus_{j=1}^n x_j\|:= \big\| \big\langle \oplus_{j=1}^n x_j, \oplus_{j=1}^n  \,x_j \big\rangle_\B\big\|^{1/2} = \Big\|\sum_{j=1}^n\, \langle  x_j,  x_j \rangle_\B\Big\|^{1/2},\]
and that $\bigoplus_{j=1}^n \X^j$ becomes a  Hilbert $\B$-bundle. 

Now, if $\rho_j$ is an $(\A, \varphi)$-action on $\X^j$ for $j=1, \ldots, n$, 
 it is straightforward to check that we get an $(\A, \varphi)$-action $\bigoplus_{j=1}^n \, \rho_j$ on $\bigoplus_{j=1}^n \X^j$ by setting
 \[ \big[(\bigoplus_{j=1}^n \, \rho_j)(a)\big] (\oplus_{j=1}^n x_j) :=  \oplus_{j=1}^n\, \rho_j(a) x_j\]
 for each $a\in \A$ and $ \oplus_{j=1}^n x_j \in \bigoplus_{j=1}^n \X^j$.

More generally, we can also form direct sums of families of actions.
Assume  $(\rho_j)_{j\in J}$ is a family of $(\A, \varphi)$-actions, where each $\rho_j$ acts on a Hilbert $\B$-bundle $\X^j$. The direct sum $\bigoplus_{j\in J} \rho_j$ is defined as follows. First, for each $g \in G$, set
\[ \bigoplus_{j\in J} X_g^j := \Big\{ (x_j)_{j\in J} \mid x_j \in X_g^j \text{ for  each $j$ and }\sum_{j\in J} \langle x_j, x_j\rangle_\B \text{ is norm-convergent in } B_e\Big\},\]
which is a subspace of the unrestricted direct sum of the vector spaces $X_g^j$ for all $j\in J$.
 One can then check that $\bigoplus_{j\in J} \X^j := \coprod_{g\in G} \big(\bigoplus_{j\in J} X_g^j\big)$ becomes a Hilbert $\B$-bundle w.r.t.~the operations
 \[ (x_j)_{j\in J}\cdot b =  (x_j\cdot b)_{j\in J}, \quad \big\langle (x_j)_{j\in J},  \,(y_j)_{j\in J}\big\rangle_\B = \sum_{j\in J} \,\langle  x_j,  y_j \rangle_\B,\]
 where $b\in \B$, $g, g'\in G$,  $x_j\in X_g^j$ and $y_j\in X_{g'}^j$ for all $j\in J$. 
 To see that $\sum_{j\in J} \,\langle  x_j,  y_j \rangle_\B$ is convergent in $B_e$, one may argue in a similar way as when dealing with Hilbert $C^*$-modules (see e.g.~\cite{La1}). Indeed, let $F$ be a finite subset of $J$ and consider the Hilbert $\B$-bundle $\bigoplus_{j\in F}\X^j = \coprod_{g\in G} \big(\bigoplus_{j\in F} X_g^j\big)$. Then $(x_j)_{j\in F}, (y_j)_{j\in F}$ belong to $\bigoplus_{j\in F}\X^j$, so the Cauchy-Schwarz inequality for Hilbert $\B$-bundles gives that
 \[\big\| \big\langle  (x_j)_{j\in F}, (y_j)_{j\in F} \big\rangle_\B\big\|^2 \, \leq \, \|(x_j)_{j\in F}\|^2 \, \|(y_j)_{j\in F}\|^2,\] 
 that is,
 \[ \Big\|\sum_{j\in F} \langle x_j, y_j\rangle_\B \Big\| \, \leq \, \Big\|\sum_{j\in F} \langle x_j, x_j\rangle_\B \Big\| \, \, \Big\|\sum_{j\in F} \langle y_j, y_j\rangle_\B \Big\|.\]
 It now follows that the net $\{\sum_{j\in F} \langle x_j, y_j\rangle_\B\}_F$ satisfies the Cauchy criterion (w.r.t.~the norm in $B_{g^{-1}g'}$), hence that it is convergent (in $B_{g^{-1}g'}$).
 
 We then get an $(\A, \varphi)$-action $\bigoplus_{j\in J} \, \rho_j$ on $\bigoplus_{j\in J} \X^j$ by setting
 \[ \big[(\bigoplus_{j\in J} \, \rho_j)(a)\big] (x_j)_{j\in J} :=  (\rho_j(a) x_j)_{j\in J}\]
 for each $a\in \A$ and $(x_j)_{j\in J} \in \bigoplus_{j\in J} \X^j$. 
 
\section{Tensor products of actions} \label{tensoraction}

For clarity of exposition, we will first consider tensor products of $\B$-actions before discussing the general case.
 
 \subsection{Tensor products of $\B$-actions on Hilbert $\B$-bundles} \label{tp-b-actions}
 
 In this subsection, we let $\B=(B_g)_{g\in G}$ be a Fell bundle over $G$ and consider two $\B$-actions $\rho$ and $\sigma$ on  Hilbert $\B$-bundles $\X=(X_g)_{g\in G}$ and $\Y=(Y_g)_{g\in G}$, respectively. A natural goal is then to form a $\B$-action which is a kind of tensor product  of $\rho$ and  $\sigma$.

As it will soon become apparent, this construction is related to the following problem.
 Pick $h \in G$, $x \in X_h$ and $y\in Y_{h^{-1}}$. Then 
for each $g\in G$, we get a linear map $R_g:B_g\to B_g$ by setting
\begin{equation}\label{R-map} R_g(b) := \big\langle y , \,\sigma\big(\langle x, \rho(b) x\rangle_\B\big) y\big\rangle_\B\end{equation}
for all $b \in B_g$. Note here that  $\langle x, \rho(b) x\rangle_\B \in B_{h^{-1}g h}$, so that \[R_g(b) \in B_{(h^{-1})^{-1}(h^{-1}gh)h^{-1}} = B_g\,.\]  Thus $R =(R_g)_{g\in G}$ is a  $\B$-bundle map, and one may wonder whether $R$ is positive definite.
 When $h=e$, this can be answered by first observing that if
 $S$ and $T$ are positive definite $\B$-bundle maps, then the $\B$-bundle map $S\cdot T:= (S_g\circ T_g)_{g\in G}$ is  positive definite too. (This can be checked directly, but it also follows from \cite[Corollary 3.18 and Remark 3.19]{BeCo7}).
Assuming that $h=e$, so $x \in X_e$ and $y\in Y_{e}$, let then $S=(S_g)_{g\in G}$ and $T=(T_g)_{g\in G}$ be the positive definite $\B$-bundle maps given by \[
S_g (b) = \langle y , \,\sigma(b)y\big\rangle_\B, \quad T_g(b) = \langle x, \rho(b) x\rangle_\B\]
for all $b\in B_g$. We  get that  
\[ R_g(b) = \big\langle y , \,\sigma\big(\langle x, \rho(b) x\rangle_\B\big) y\big\rangle_\B = S_g (T_g(b))= (S\cdot T)_g(b)\]
for all $b\in B_g$. Hence,  $R=S\cdot T$, so $R$ is positive definite. 

Below, after constructing 
$\rho\dot\otimes \sigma$,
the  tensor product of $\rho$ and  $\sigma$, 
we will see that the $\B$-bundle map $R$, defined as in (\ref{R-map}), can be written as a suitable diagonal coefficient of $\rho\dot\otimes \sigma$, which will give a way to deduce that $R$ is positive definite, not only when $h=e$, but for any choice of $h \in G$.

 To be able to form a tensor product $\B$-action $\rho \dot\otimes \sigma$, we must first have a certain (internal) tensor product Hilbert $\B$-bundle 
 $\X\otimes_\sigma \Y$ at our disposal. 
 Our construction of such a bundle  is largely inspired by the one given by Abadie and Ferraro in \cite[Section 5.1]{AF}, where they consider 
 so-called equivalence bundles between Fell bundles over a locally compact group. In our situation, many aspects of their discussion apply (and simplify). We will therefore be somewhat sketchy in our presentation.
  
For each $b \in B_e$ and $h \in G$, the restriction of $\sigma(b)$ to $Y_h$ is an adjointable map from $Y_h$ into itself (considered as a  Hilbert $B_e$-module), which we denote by $\sigma_h(b)$, cf.~Remark \ref{action-adjoint}. 

Moreover, the map $b\mapsto \sigma_h(b)$ is a $*$-homomorphism $\sigma_h$ from $B_e$ into $\L_{B_e}(Y_h)$.
 
 Thus, for any $g, h \in G$,  considering $X_g$ and $Y_h$ as  Hilbert $B_e$-modules, their internal tensor product $X_g \otimes_{\sigma_h} Y_h$ is then a  Hilbert $B_e$-module, cf.~\cite{La1}, such that
\[(x \cdot b) \dot\otimes y = x \dot\otimes \sigma_h(b)y \text{  and  } \langle x \dot\otimes y, x' \dot\otimes y' \rangle_{B_e} = \langle y, \sigma_h(\langle x, x' \rangle_\B) y' \rangle_\B\] for all  $x,x' \in X_g$,  $y,y' \in Y_h$ and $b \in B_e$. 
  
  For every $r \in G$, let $Z^0_r$ be the vector space given by
  \[Z^0_r := \Big\{ \xi \in \prod_{h \in G} (X_{r h^{-1}} \otimes_{\sigma_h} Y_h) \mid \xi \text{ is finitely supported}\Big\}\]
 and set $\mathcal{Z}^0 :=\coprod_{r\in G} Z^0_r$ (disjoint union).

  We will first define a right action of $\B$ on $\mathcal{Z}^0$, $(z,b)\mapsto z\cdot b$, and a $\B$-valued semi-inner product on $\mathcal{Z}^0$,
 $(z, z') \mapsto \langle z, z'\rangle_\B$, which will turn $\mathcal{Z}^0=(Z^0_r)_{r\in G}$ into a semi-inner product $\B$-bundle.
 
 Let $g \in G$ and $b \in B_g$. For $s, t \in G$, let $x \in X_{s}$, $y \in Y_t$.
  Note  that $y \cdot b \in Y_{tg}$, so we can define 
  \[ (x \dot\otimes y) \cdot b :=  x \dot\otimes (y \cdot b) \in X_{s} \otimes_{\sigma_{tg}} Y_{tg}.\]
One may now check that the map $x\dot\otimes y \mapsto (x \dot\otimes y) \cdot b$ extends to a linear map $z\mapsto z \cdot b$ from $X_{s} \otimes_{\sigma_t} Y_t$ into $X_{s} \otimes_{\sigma_{tg}} Y_{tg}$ such that $\|z\cdot b\| \leq \|z\| \, \|b\|$ for every $z \in X_{s} \otimes_{\sigma_t} Y_t$.

We can then define a linear map  $\xi \mapsto \xi \cdot b$ from $Z^0_r$ into $Z^0_{rg}$ by setting
\[ (\xi \cdot b)(h) := \xi(hg^{-1})\cdot b \quad \text{for all } \xi \in Z^0_r, b \in B_g \text{ and } h\in G.\]
Note here that for every $h\in G$, we have
\[\xi(hg^{-1}) \in X_{r(hg^{-1})^{-1}}\otimes_{\sigma_{hg^{-1}}}Y_{hg^{-1}} = X_{rgh^{-1}}\otimes_{\sigma_{hg^{-1}}}Y_{hg^{-1}},\] so 
\[ (\xi \cdot b)(h) = \xi(hg^{-1})\cdot b \in X_{rgh^{-1}}\otimes_{\sigma_{(hg^{-1})g}}Y_{(hg^{-1})g} = X_{rgh^{-1}}\otimes_{\sigma_{h}}Y_{h},\]
showing that  $\xi \cdot b \in Z^0_{rg}$.

It is not difficult to see that the resulting map $(\xi, b)\mapsto \xi\cdot b$ from $Z_r^0\times B_g$  into $Z^0_{rg}$ is also linear in the second variable.
 
Let now  $r, r' , h, h'\in G$. Then one checks that there is a unique well-defined bilinear map 
 $(z, z')\mapsto \langle z, z' \rangle_\B$ from $ (X_{r h^{-1}} \otimes_{\sigma_h} Y_h) \times (X_{r' h'^{-1}}\otimes_{\sigma_{h'}} Y_{h'})$ into  $B_{r^{-1}r'}$ such that  \[
 \big\langle x \dot\otimes y, x' \dot\otimes y' \big\rangle_\B  = \big\langle y, \sigma(\langle x,x' \rangle_\B) y' \big\rangle_\B 
 \]
 for all $x \in X_{r h^{-1}}$, $y \in Y_h$, $x' \in X_{r' {h'}^{-1}}$ and $y' \in Y_{h'}$. Moreover, we have 
 \[\|\langle z, z' \big\rangle_\B\| \, \leq \, \|z\| \, \|z'\|\]
for all $z \in X_{r h^{-1}} \otimes_{\sigma_h} Y_h, \, z'\in X_{r' h'^{-1}} \otimes_{\sigma_{h'}} Y_{h'}$. 
This leads to a bilinear map $(\xi, \xi')\mapsto 
  \langle \xi, \xi' \big\rangle_\B$ from $Z_r^0 \times Z_{r'}^0$ into $B_{r^{-1}r'}$ given by
  \[  \langle \xi, \xi' \big\rangle_\B = \sum_{h, h'\in G}  \big\langle \xi(h), \xi'(h') \big\rangle_\B.\]
  One can now verify without much trouble that we have
  \[  \langle \xi, \xi'\cdot b \big\rangle_\B =  \langle \xi, \xi' \big\rangle_\B \, b \quad \text{ and }  \big(\langle \xi, \xi' \big\rangle_\B\big)^* =  \langle \xi', \xi \big\rangle_\B\]
  for all $\xi \in Z_r^0, \xi'\in Z_{r'}^0$ and $b \in \B$. Moreover, following the argument in the proof of \cite[Lemma 5.9]{AF}, it can be shown that $\langle \xi, \xi \big\rangle_\B \, \geq 0$ (in $B_e$) for every $\xi \in Z_r^0$ (see also Lemma \ref{ip-ineq}). Thus $Z_r^0$ is a semi-inner product $B_e$-module with a seminorm given by $\|\xi\| := \|\langle \xi, \xi \big\rangle_\B \|^{1/2}$. We also have that
  \[ \|\xi\cdot b\|\leq  \|\xi\|\, \| b\| \quad \text{ and}
  \quad  \|\langle \xi, \xi' \big\rangle_\B \| \leq \, \|\xi\| \|\xi'\|\]
  for all $\xi \in Z_r^0, \xi'\in Z_{r'}^0$ and $b\in \B$. The first inequality is quite obvious. To see the that the second  holds, set $c:= \langle \xi, \xi' \big\rangle_\B \in B_{r^{-1}r'}$, so $\xi'$ and $\xi\cdot c$ both belong to  $Z_{r'}^0$. Then, using that the Cauchy-Schwarz inequality holds in a semi-inner product $B_e$-module, we get
  \[ \|c\|^2 = \|\langle \xi, \xi' \big\rangle_\B \|^2 = \| \langle \xi', \xi \big\rangle_\B \,c\| = 
  \|\langle \xi', \xi\cdot c \big\rangle_\B \| \leq \| \xi'\| \| \xi\cdot c \| \leq \|\xi'\| \|\xi\| \|c\| , \]
 and the desired inequality follows. 
 
 The maps  $\mathcal{Z}^0\times \B \to \mathcal{Z}^0$ and  $\mathcal{Z}^0 \times \mathcal{Z}^0 \to \B$ 
  constructed above satisfy all the properties required for $\mathcal{Z}^0$ to be a semi-inner product $\B$-bundle. After separation and completion,
  cf.~Remark \ref{pre_H}, we obtain a full-fledged  Hilbert $\B$-bundle ${\mathcal Z} = (Z_r)_{r \in G}$, that we will also denote by $\X\otimes_\sigma \Y$.

Concerning notation, we will sometimes use the following convention. 
Let $ k, l,  r \in G$ be such that $kl=r$, and let $x \in X_k$, $y \in Y_l$. Then we will identify $x\dot\otimes y \in X_k\dot\otimes_{\sigma_l}Y_l$ with the image in $Z_r$ of the element $\xi_{x, y}$ in $Z^0_r$ defined by 
\[ \xi_{x,y}(h)=\begin{cases}x\dot\otimes y \quad&\text{if } h= l,\\ \,\, \,0  &\text{if } h\neq l.\end{cases}\]
Note that this identification is coherent with the right action of $\B$ and with the inner product. Indeed, if $g \in G$ and $b\in B_g$, then  
\[ (\xi_{x, y}\cdot b)(h) = \xi_{x, y}(hg^{-1})\cdot b
=\begin{cases}  (x\dot\otimes y)\cdot b  &\text{if } h=lg,\\ \, \, 0 &\text{if } h\neq lg
\end{cases}
=\begin{cases}  x\dot\otimes (y\cdot b)  &\text{if } h=lg,\\ \, \, 0 &\text{if } h\neq lg
\end{cases}
=   \xi_{x, y\cdot b}(h).
\]
Moreover, if $ k', l', r' \in G$ are such that $k'l'=r'$, and  $x' \in X_{k'}$, $y \in Y_{l'}$, then
\[ \big\langle \xi_{x, y}, \xi_{x', y'}\big\rangle_\B = \sum_{h, h'\in G}  \big\langle \xi_{x,y}(h), \xi_{x', y'}(h') 
\big\rangle_\B =\langle x\dot\otimes y, x'\dot\otimes y'\rangle_\B .\] 
 
 We can now summarize our discussion above as follows. 
\begin{proposition}
There exists a Hilbert $\B$-bundle 
 $\X\otimes_\sigma \Y = (Z_r)_{r\in G}$ such that for each $r\in G$, $Z_r$ contains a subset $\{x \dot\otimes y: x \in X_{r h^{-1}}, y \in Y_h, 
 h\in G\}$ whose span is dense in $Z_r$, with operations satisfying \[(x \dot\otimes y)\cdot b = x \dot\otimes (y\cdot b), \quad \langle x \dot\otimes y, x' \dot\otimes y'\rangle_\B = \big\langle y, \sigma(\langle x, x'\rangle_\B)y'\big\rangle_\B\] for all $g, h, h', r, r' \in G$, $b \in B_g$, $x \in X_{r h^{-1}}$,  $y \in Y_h$, $x' \in X_{r' h'^{-1}}$,  $y' \in Y_{h'}$.
 \end{proposition}
 
Our next goal is to define a  $\B$-action $\rho\dot\otimes \sigma$ on $\X\otimes_\sigma \Y$.
 Let $g, h, r \in G$. For $b \in B_g$, $x \in X_{r h^{-1}}$, $y \in Y_h$, so $x \dot\otimes y \in X_{rh^{-1}} \otimes_{\sigma_h} Y_h$, set 
 \[
(\rho(b)\otimes \iota) (x \dot\otimes y) := (\rho(b) x) \dot\otimes y\, \in \,X_{(gr)h^{-1}} \otimes_{\sigma_h} Y_h.
 \]
 Note that if $r' , h' \in G$ and $x' \in X_{r' h'^{-1}}$, $y' \in Y_{h'}$, then
\begin{equation}\label{rho-adj}
\begin{split} 
\big\langle (\rho(b)\otimes \iota) (x \dot\otimes y),  x' \dot\otimes y' \big\rangle_\B & = \big\langle (\rho(b) x) \dot\otimes y , x' \dot\otimes y' \big\rangle_\B \\
& = \big\langle y, \sigma(\langle \rho(b) x, x'\rangle_\B )y' \big\rangle_\B \\
& =  \big\langle y, \sigma(\langle x, \rho(b^*) x'\rangle_\B )y' \big\rangle_\B \\
& = \big\langle  x \dot\otimes y, (\rho(b^*)x') \dot\otimes y' \big\rangle_\B \\
& = \big\langle  x \dot\otimes y, \rho_*(b^*)  (x' \dot\otimes y') \big\rangle_\B.
\end{split}
\end{equation}
 Extending by linearity we obtain a  map $z \mapsto (\rho(b)\otimes \iota)(z)$ from  $X_{rh^{-1}} \otimes_{\sigma_h} Y_h$ to $X_{(gr)h^{-1}} \otimes_{\sigma_h} Y_h$. To show that this map is well-defined, we need to show that if $\sum_{j=1}^n x_j \dot\otimes y_j = 0$ in $X_{rh^{-1}} \otimes_{\sigma_h} Y_h$, then $\sum_{j=1}^n (\rho(b)x_j) \dot\otimes y_j = 0$ in $X_{(gr)h^{-1}} \otimes_{\sigma_h} Y_h$. This is an immediate consequence of the following lemma.
 \begin{lemma} Let $g, h, r\in G$, $b \in B_g$,  $x_i\in X_{rh^{-1}}$ and $y_i \in Y_h$ for $i=1, \ldots n$. 
Then, with 
\[\text{$z := \sum_{ j=1}^n x_j \dot\otimes y_j \in X_{rh^{-1}} \otimes_{\sigma_h} Y_h$ and $z'= \sum_{j=1}^n (\rho(b)x_j) \dot\otimes y_j \in X_{(gr)h^{-1}} \otimes_{\sigma_h} Y_h$},\] we have
 $ \langle z', z'\rangle_\B \leq \|b\|^2 \, \langle z, z\rangle_\B  \, \text{ $($in $(B_e)^+)$, hence}\, \|z'\| \leq \|b\| \, \|z\|$.
 \end{lemma}
 
 \begin{proof}
 Consider $Y_h$ as a Hilbert $B_e$-module and set $\mathbf{y}:= (y_1, \ldots, y_n) \in (Y_h)^n$. Since the map $\sigma_h: B_e \to \mathcal{L}_{B_e}(Y_h)$ 
 is a $*$-homomorphism, we get a $*$-homomorphism $\widetilde\sigma_h: M_n(B_e) \to M_n\big(\mathcal{L}_{B_e}(Y_h) \big) = \mathcal{L}_{B_e}((Y_h)^n)$ by setting
 \[ \widetilde\sigma_h\big([b_{i,j}]\big) = \big[\sigma_h(b_{i,j})\big]\]
 for all $[b_{i,j}] \in M_n(B_e)$. Now, with $R= \big[\langle   x_i,  x_j \rangle_\B\big]$ and 
$S= \big[\langle  \rho(b) x_i,  \rho(b) x_j \rangle_\B\big]$, Lemma \ref{Lance-like} gives that 
\[  S \leq \|b\|^2 \,R \quad (\text{in $M_n(B_e)^+$})\]
 and thus 
 \[ \widetilde\sigma_h(S) \leq \|b\|^2 \, \widetilde\sigma_h(R).\]
 We therefore obtain
 \begin{align*}
  \langle z', z'\rangle_\B &= \sum_{i,j =1}^n \big\langle  (\rho(b) x_i) \dot\otimes y_i,  (\rho(b) x_j)\dot\otimes y_j \big\rangle_\B = \sum_{i,j =1}^n \big\langle y_i,  \sigma\big( \langle  \rho(b) x_i,  \rho(b) x_j \rangle_\B\big) y_j \big\rangle_\B\\
  &= \big\langle \mathbf{y},  \big[ \sigma_h(\langle  \rho(b) x_i,  \rho(b) x_j \rangle_\B)\big] \mathbf{y}\big\rangle_{B_e}= \big\langle \mathbf{y},  \widetilde\sigma_h(S) \mathbf{y}\big\rangle_{B_e}\\
  &\leq \|b\|^2 \, \big\langle \mathbf{y},  \widetilde\sigma_h(R) \mathbf{y}\big\rangle_{B_e}
  = \|b\|^2 \, \sum_{i,j =1}^n \big\langle   x_i \dot\otimes y_i,   x_j\dot\otimes y_j \big\rangle_\B \\ 
  &= \|b\|^2 \, \langle z, z\rangle_\B.
 \end{align*}
 
 \vspace{-4ex}\end{proof}
Proceeding further, for each $g\in G$ and $b\in B_g$, we get a linear map $(\rho\otimes \sigma)(b) : Z^0_r\to Z^0_{gr}$ by setting
\[ \big((\rho\otimes \sigma)(b)\xi\big)(h) := (\rho(b)\otimes \iota)(\xi(h)) \quad \text{for all } \xi \in Z^0_r \text{ and } h \in G.\]
 \begin{lemma} \label{ip-ineq}
 Let $\xi \in Z^0_r$, $g\in G$ and $b\in B_g$. Then we have \[0\leq \big\langle (\rho\otimes \sigma)(b)\xi, (\rho\otimes \sigma)(b)\xi \big\rangle_\B \, \leq\, 
 \|b\|^2\, \langle \xi, \xi\rangle_\B.\]
 \end{lemma}
 \begin{proof} 
 Let $h_1, \ldots, h_m\in G$ be an enumeration of the support of $\xi$.  
  For each $i \in \{1, \ldots, m\}$, we can write
 $\xi(h_i) = \sum_{j=1}^{n_i} x^i_j\dot\otimes y^i_j  \in X_{rh_i^{-1}} \otimes_{\sigma_{h_i}} Y_{h_i}$ for some $n_i\in \Naturali$. 
 Setting $N= \max\{n_i: i\in \{1, \ldots, m\}\}$ and adding zeroes whenever needed, we may suppose that \[\xi(h_i) = \sum_{j=1}^{N} x^i_j\dot\otimes y^i_j  \in X_{rh_i^{-1}} \otimes_{\sigma_{h_i}} Y_{h_i} \,  \text{ for every $i\in \{1, \ldots, m\}$}.\] 
Then 
 \begin{align*} \big\langle (\rho\otimes \sigma)(b)\xi, (\rho\otimes \sigma)(b)\xi \big\rangle_\B  
 &= \sum_{i,k=1}^m \big\langle ((\rho\otimes \sigma)(b)\xi)(h_i), ((\rho\otimes \sigma)(b)\xi)(h_k) \big\rangle_\B \\
  &= \sum_{i, k=1}^m \sum_{j, l=1}^{N}
 \big\langle (\rho\otimes \sigma)(b)(x^i_j\dot\otimes y^i_j ), (\rho\otimes \sigma)(b)(x^k_l\dot\otimes y^k_l)  \big\rangle_\B\\
  &= \sum_{i, k=1}^m \sum_{j, l=1}^{N}
 \big\langle (\rho(b)x^i_j)\dot\otimes y^i_j , (\rho(b)x^k_l)\dot\otimes y^k_l  \big\rangle_\B\\
  &= \sum_{i, k=1}^m \sum_{j, l=1}^{N}
 \big\langle y^i_j, \sigma\big( \langle \rho(b)x^i_j, \rho(b)x^k_l\rangle_\B\big) y_k^l\big\rangle_{B_e}
 \end{align*}
 Note that for every $i,j,k,l$, we have \[\langle x^i_j, x^k_l\rangle_\B \in B_{(rh_i^{-1})^{-1}rh_k^{-1}} = B_{h_ih_k^{-1}}, \text{ and, similarly, }\, 
 \langle \rho(b)x^i_j, \rho(b)x^k_l\rangle_\B \in B_{h_ih_k^{-1}}.\]   
 Set $\mathcal{I}:=\{1, \ldots, m\} \times \{1, \ldots, N\}$. Then  for each $(i,j)\in \mathcal{I}$, set $h_{(i, j)}:= h_i^{-1}$.  Moreover, set $\mathbf{H}:=[h_{(i, j)}]_{(i,j)\in  \mathcal{I}}$. 
 
 \smallskip We can now consider $\big[\langle x^i_j, x^k_l\rangle_\B\big]_{(i,j), (k,l) \in  \mathcal{I}}$ and $\big[\langle \rho(b)x^i_j, \rho(b)x^k_l\rangle_\B\big]_{(i,j), (k,l)\in  \mathcal{I}}$ as elements in the set
 \[M_\mathbf{H}(\B) := \Big\{ \big[b_{(i, j), (k,l)}\big]_{(i,j), (k,l)\in  \mathcal{I}} :  b_{(i, j), (k,l)} \in B_{{h_{(i,j)}^{-1}}h_{(k,l)}} = B_{h_i h_k^{-1}}\Big\},\]
 which can be organized as a $C^*$-algebra in an analogous way to the Abadie-Ferraro $C^*$-algebras, cf.~Section \ref{prelim}.\footnote{By linearly ordering the elements in $\mathcal{I}$,  say lexicographically, we can write $\mathbf{H}$ as an element $\mathbf{h}$ in $H^{mN}$ and the elements $\big[b_{(i, j), (k,l)}\big]_{(i,j), (k,l)\in  \mathcal{I}}$ as matrices in $M_{\mathbf{h}}(\B)$ correspondingly. Thus we can always reduce to considering a full-fledged Abadie-Ferraro $C^*$-algebra.}
 Now, let $\mathbf{Y}_{\mathbf{H}}$ be the Hilbert $B_e$-module given by 
 \[\mathbf{Y}_{\mathbf{H}} = \bigoplus_{(i, j) \in \mathcal{I}} \, Y_{{h_{(i,j)}^{-1}}},\]
 and note that $\mathbf{Y}:= \big(y^i_j\big)_{(i,j)\in \mathcal{I}} \in \mathbf{Y}_{\mathbf{H}}$.
 
 Now, by Lemma \ref{Lance-like} (with $\varphi={\rm id}_G$), we have that 
 \[\big[\langle \rho(b)x^i_j, \rho(b)x^k_l\rangle_\B\big]_{(i,j), (k,l) \in  \mathcal{I}} \, \leq \, \|b\|^2 \,\big[\langle x^i_j, x^k_l\rangle_\B\big]_{(i,j), (k,l) \in  \mathcal{I}}\] in $M_\mathbf{H}(\B) ^+$.
 Moreover, this order relation survives under the $*$-homomorphism $L^\sigma$ from $M_\mathbf{H}(\B)$ into  $\mathcal{L}_{B_e}(\mathbf{Y}_{\mathbf{H}})$
 defined as in Remark \ref{AF-action}. Skipping the external indices of the matrices involved, this implies that
\[ \big\langle \mathbf{Y}, L^\sigma\big(\big[\langle \rho(b)x^i_j, \rho(b)x^k_l\rangle_\B\big]\big) \mathbf{Y}  \big\rangle_{B_e} \leq 
 \|b^2\| \,  \big\langle \mathbf{Y}, L^\sigma\big(\big[\langle x^i_j, x^k_l\rangle_\B\big]\big) \mathbf{Y}  \big\rangle_{B_e}.\]
 Hence we get
 \begin{align*}
 \big\langle (\rho\otimes \sigma)(b)\xi, (\rho\otimes \sigma)(b)\xi \big\rangle_\B  
 & = \sum_{i, k=1}^m \sum_{j, l=1}^{N} \big\langle y^i_j, \sigma\big( \langle \rho(b)x^i_j, \rho(b)x^k_l\rangle_\B\big) y_k^l\big\rangle_{B_e} \\
 & =  \big\langle \mathbf{Y}, L^\sigma\big(\big[\langle \rho(b)x^i_j, \rho(b)x^k_l\rangle_\B\big]\big) \mathbf{Y}  \big\rangle_{B_e} \\
&  \leq \|b\|^2\,   \big\langle \mathbf{Y}, L^\sigma\big(\big[\langle x^i_j, x^k_l\rangle_\B\big]\big) \mathbf{Y}  \big\rangle_{B_e} \\
& = \|b\|^2\,\sum_{i, k=1}^m \sum_{j, l=1}^{N} \big\langle y^i_j, \sigma\big( \langle x^i_j, x^k_l\rangle_\B\big) y_k^l\big\rangle_{B_e} \\
& = \|b\|^2 \big\langle \xi, \xi \big\rangle_\B  \ . 
 \end{align*}
 Note that the computation above also gives that \[ 
  \big\langle (\rho\otimes \sigma)(b)\xi, (\rho\otimes \sigma)(b)\xi \big\rangle_\B =   
    \big\langle \mathbf{Y}, L^\sigma\big(\big[\langle x^i_j, x^k_l\rangle_\B\big]\big) \mathbf{Y}  \big\rangle_{B_e} \geq 0,\] because $\big[\langle x^i_j, x^k_l\rangle_\B\big] \geq 0$ in  $M_\mathbf{H}(\B)$, cf.~\cite[Lemma 3.8]{BeCo8}, so $L^\sigma\big(\big[\langle x^i_j, x^k_l\rangle_\B\big]\big) \geq 0$ in $\mathcal{L}_{B_e}(Y_{\mathbf{H}})$.
   \end{proof}
  In turn, using Lemma \ref{ip-ineq}, we get that for $g\in G$ and $b \in B_g$, the linear map $(\rho\otimes \sigma)(b) : Z^0_r\to Z^0_{gr}$ induces a  map 
  $(\rho\dot\otimes \sigma)(b) : Z_r\to Z_{gr}$.
  At the end,  the resulting map $\rho\dot\otimes \sigma : \B\to \mathcal{F}(\mathcal{Z})$ is easily seen to be a $\B$-action on $\mathcal{Z}=\X\otimes_\sigma \Y$. For example, condition (iii) of Definition \ref{rHilbertBbundle} follows  from the computation in equation (\ref{rho-adj}). 
  We have thus obtained the following result.
  \begin{theorem} \label{tensor-action}
 Let  $\rho$ and $\sigma$ be $\B$-actions on some Hilbert $\B$-bundles $\X = (X_g)_{g \in G}$ and $\Y = (Y_g)_{g \in G}$, respectively.  Then there is  
 a $\B$-action $\rho \dot\otimes \sigma$ on the Hilbert $\B$-bundle 
 $\X\otimes_\sigma \Y$ satisfying that 
 \[
(\rho \dot\otimes \sigma)(b) (x \dot\otimes y) = (\rho(b) x) \dot\otimes y
 \]
 for all $g, h, r \in G$, $b \in B_g$, $x \in X_{r h^{-1}}$ and $y \in Y_h$.
  \end{theorem}
 \begin{example}\label{G-trivial}
Let us consider the case where $G$ is trivial. Then the Fell bundle $\B$ can be identified with a $C^*$-algebra $B$, the Hilbert $\B$-bundle $\X$ (resp.~$\Y$) with a Hilbert $B$-module $X$ (resp.~$Y$), and the action $\rho$ (resp.\ $\sigma$)  of $\B$ on $\X$ (resp.\ $\Y$) with a $*$-representation of $B$ in $\L_B(X)$ (resp.~$\L_B(Y)$). It is now easy to see that $\X\otimes_\sigma\Y$ is the Hilbert $B$-module obtained by forming the internal tensor product $X\otimes_\sigma Y$, and  $\rho\dot\otimes \sigma$ is the $*$-representation of $B$ into  $\L_B(X\otimes_\sigma Y)$ satisfying  
\[(\rho\dot\otimes \sigma)(b) (x \dot\otimes y) = (\rho(b) x) \dot \otimes y\]
for all $b\in B,  x \in X$ and $y\in Y$, i.e., $(\rho\dot\otimes \sigma)(b) =  (\rho(b)\otimes \iota)$
(or $(\rho\dot\otimes \sigma)(b) = \sigma_{*}(\rho(b))$) in the notation used by Lance (cf.~\cite[p.~42]{La1}).
\end{example}

\begin{remark} \label{boxaction}
 For every $r \in G$, 
 we can also consider
  \[W^0_r := \Big\{ \xi \in \prod_{h \in H} (X_h \otimes_{\sigma_{h^{-1}r} } Y_{h^{-1}r}) \mid \xi \text{ is finitely supported}\Big\}\]
  and $\mathcal{W}^0:=\coprod_{r\in G} W^0_r$ (disjoint union).
 We can then define a linear map  $\xi \mapsto \xi \cdot b$ from $W^0_r$ into $W^0_{rg}$ by setting
\[ (\xi \cdot b)(h) := \xi(h)\cdot b \quad \text{for all } \xi \in W^0_r, b \in B_g \text{ and } h\in G.\]

If now  $r, r' , h, h'\in G$, then there is a unique well-defined bilinear map 
 $(z, z')\mapsto \langle z, z' \rangle_\B$ from $ (X_h \otimes_{\sigma_{h^{-1}r}} Y_{h^{-1}r}) \times 
(X_{h'} \otimes_{\sigma_{h'^{-1}r'}} Y_{h'^{-1}r'}) $ into  $B_{r^{-1}r'}$ such that  \[
 \big\langle x \dot\otimes y, x' \dot\otimes y' \big\rangle_\B  = \big\langle y, \sigma(\langle x,x' \rangle_\B) y' \big\rangle_\B .
 \]
 This induces a bilinear map $(\xi, \xi')\mapsto 
  \langle \xi, \xi' \big\rangle_\B$ from $W_r^0 \times W_{r'}^0$ into $B_{r^{-1}r'}$ given by
  \[  \langle \xi, \xi' \big\rangle_\B = \sum_{h, h'\in G}  \big\langle \xi(h), \xi'(h') \big\rangle_\B.\]

Moreover, for each $g\in G$ and $b\in B_g$, we get a linear map $(\rho\boxtimes \sigma)(b) : W^0_r\to W^0_{gr}$ by setting
\[ \big((\rho\boxtimes \sigma)(b)\xi\big)(h) := 
(\rho(b)\otimes \iota)\xi(g^{-1}h)
\quad \text{for all } \xi \in W^0_r \text{ and } h \in G,\]
where $(\rho(b)\otimes \iota)(x\dot\otimes y) := (\rho(b)x)\dot\otimes y$ whenever
$k, l\in G$ are such that $kl=r$, $x\in X_k$ and $y\in Y_{l}$. Note here that this definition of $(\rho(b)\otimes \iota)(x\dot\otimes y) $ is coherent with the one introduced previously, and that $(\rho(b)\otimes \iota)$ can be extended to a linear map from $X_k\otimes_{\sigma_l}Y_l$ to  $X_{gk}\otimes_{\sigma_l}Y_l$.

Proceeding as we did before Theorem \ref{tensor-action}, we obtain after separation and completion  a (right) Hilbert $\B$-bundle $\mathcal{W}$ and a $\B$-action $\rho\dot\boxtimes \sigma$ on $\mathcal{W}$. We claim that this action is unitarily equivalent to $\rho\dot\otimes \sigma$ (in the sense of \cite[Definition 4.18]{BeCo8}).

Set $\U^0 :=(U^0)_{r\in G}$, where for each $r\in G$, $U^0_r: Z^0_r\to W^0_r$ is defined by 
\[ (U^0_r\xi) (h) = \xi(h^{-1}r) \quad \text{for all } h \in G \text{ and } \xi \in Z^0_r \, .\]
Note that $U^0_r$ takes its values in $W^0_r$ since $\xi(h^{-1}r) \in X_{r(h^{-1}r)^{-1}} \otimes_{\sigma_{h^{-1}r}} Y_{h^{-1}r} = X_h \otimes_{\sigma_{h^{-1}r}} Y_{h^{-1}r}$.
It is now routine to check that $U^0_r$ is a bijective linear map respecting the $\B$-semi-inner products and intertwining the right actions of $\B$, the inverse map $V^0_r: W^0_r\to Z^0_r$ being given by 
\[ (V^0_r\xi) (h) = \xi(rh^{-1}) \quad \text{for all } h \in G \text{ and } \xi \in W^0_r \, .\] 
For all $g, h \in G$, $\xi \in Z^0_r$ and $b \in B_g$, we have
\begin{align*}
((U^0_r\xi) \cdot b)(h) & = (U^0_r\xi)(h) \cdot b = \xi(h^{-1}r) \cdot b \\
& = (\xi \cdot b)(h^{-1}r g) = U^0_{rg}(\xi \cdot b)(h),
\end{align*}
Thus, $(U^0_r\xi) \cdot b = U^0_{rg}(\xi \cdot b)$. Moreover, we have
\begin{align*} 
\big(\big((\rho\boxtimes \sigma)(b)U^0_r\big)\xi\big)(h) &=  \big((\rho\boxtimes \sigma)(b)(U^0_r\xi)\big)(h)\\
&=(\rho(b)\otimes \iota)\big((U^0_r\xi)(g^{-1}h)\big)\\
&=(\rho(b)\otimes \iota)\xi(h^{-1}gr) \\
&=\big((\rho\otimes \sigma)(b)\xi\big)(h^{-1}gr) \\
&=\big(U^0_{gr}\big(\rho\otimes \sigma)(b)\xi\big)\big)(h) 
\end{align*} 
hence, $(\rho\boxtimes \sigma)(b)U^0_r= U^0_{gr}(\rho\otimes \sigma)(b)$. It is now straightforward to check that $\mathcal{U}^0=(U^0_r)_{r\in G}$ induces a unitary $\mathcal{U}=(U_r)_{r\in G}$ from $\mathcal{Z}$ to $\mathcal{W}$ intertwining $\rho\dot\otimes \sigma$ and $\rho\dot\boxtimes \sigma$. 
\end{remark}
 
Let us now return to our discussion at the beginning of this subsection. 

\begin{proposition} \label{comp-pos-def}
Let $\B=(B_g)_{g\in G}$ a Fell bundle over $G$ and consider two $\B$-actions $\rho$ and $\sigma$ on  Hilbert $\B$-bundles $\X=(X_g)_{g\in G}$ and $\Y=(Y_g)_{g\in G}$, respectively.  
Let $h \in G$, $x \in X_{h^{-1}}$ and $y\in Y_{h}$ and consider the $\B$-bundle map $R=(R_g)_{g\in G}$ defined by 
\[ R_g(b)= \langle y, \sigma(\langle x,\rho(b) x\rangle_\B) y \rangle_\B \quad \text{ for all } g\in G \text{ and } b \in B_g.\]
Then $R$ is positive definite.
\end{proposition}
\begin{proof} For every $g\in G$ and $b \in \B_g$, we get that
\[ R_g(b)= \langle y, \sigma(\langle x,\rho(b) x\rangle_\B) y \rangle_\B = \langle x\dot\otimes y, (\rho(b) x) \dot\otimes y \rangle_\B =\big\langle x\dot\otimes y, [(\rho\dot\otimes \sigma)(b)] (x\dot\otimes y)\big\rangle_\B .\]
Thinking of $x \dot\otimes y$ as an element in the unit fiber of $\X\otimes_\sigma \Y$, we see that $R$ is a diagonal coefficient of $\rho\dot\otimes \sigma$ over this fibre. Thus,  \cite[Proposition 4.8]{BeCo8} gives that $R$ is positive definite.
\end{proof}

\subsection{Tensor product of actions associated to $C^*$-dynamical systems} \label{tp-equiv}

To illustrate the concept of tensor product of actions developed above, we will consider the Fell bundle $\B_\Sigma= (A\times\{g\})_{g\in G}$ associated to a discrete unital $C^*$-dynamical system $\Sigma=(A, G, \alpha)$.  
We recall that the operations and the norms are given by \[(a, g) (a', g') = (a\alpha_g(a'), gg'),  (a, g)^* = (\alpha_{g^{-1}}(a^*), g^{-1}) \text{ and } \|(a, g)\| = \|a\|. \]  As usual, we identify the unit fiber $A\times \{e\}$ with $A$ via the map $(a,e)\mapsto a$.  

To  each a Hilbert $A$-module $X$  one can then associate a Hilbert $\B_\Sigma$-bundle $\X=(X_g)_{g\in G}$ as  follows (cf.~\cite{BeCo8}).
 For each $g\in G$, set $X_g:= X\times \{g\}$  and organize $X_g$ as a Banach space in the obvious way. 
 Moreover, for $x,x' \in X, a\in A$ and $g, h, h' \in G$, set
\begin{align*}
(x, h) \cdot (a, g) &:= \big(x\,\alpha_h(a), hg\big), \\
\big\langle (x, h), \, (x', h')\big\rangle_{\B_\Sigma} &:= \big(\alpha_h^{-1}(\langle x, x'\rangle_A), h^{-1}h'\big).
\end{align*}

\noindent Assume now $(\varrho, v)$ (resp.~$(\varsigma, w)$)  is an equivariant representation of  $\Sigma$ on a Hilbert $A$-module $X$ (resp.~$Y$), cf.~\cite{BeCo3, BeCo4, BeCo6}. Let then  $\X=(X_g)_{g\in G}$ (resp.~$\Y=(Y_g)_{g\in G}$) denote the  Hilbert $\B_\Sigma$-bundle associated to $X$ (resp.~$Y$), and $\rho_{(\varrho,v)}$ (resp.~$\rho_{(\varsigma, w)}$)
be the associated $\B_\Sigma$-action on $\X$ (resp.~$\mathcal{Y}$), cf.~\cite{BeCo8}, which is defined 
 by \[\rho_{(\varrho,v)}((a, g)) (x, h) = \big(\varrho(a)v(g) x, gh\big)\] for every $(a, g) \in \B_\Sigma$ and $(x, h) \in \X$, (resp. \[\rho_{(\varsigma,w)}((a, g)) (y, h) = \big(\varsigma(a)w(g) y, gh\big)\] for every $(a, g) \in \B_\Sigma$ and $(y, h) \in \Y$).
 
 For brevity, we set $\rho:=\rho_{(\varrho,v)}$ (resp.~$\sigma:=\rho_{(\varsigma, w)}$).
A natural task is to compare the tensor action $\rho\dot\otimes\sigma$  with the $\B_\Sigma$-action $\tau:= \rho_{(\varrho\otimes\varsigma, v\otimes w)}$ associated to the tensor product equivariant representation $(\varrho\otimes\varsigma, v\otimes w)$ of $\Sigma$ on the Hilbert $A$-module $X \otimes_{\varsigma} Y$, as defined in \cite[subsection 2.2]{BeCo6}. 

 The tensor product action $\rho \dot\otimes \sigma$ acts on the Hilbert $\B_\Sigma$-bundle $\X\otimes_\sigma\mathcal{Y} = (Z_r)_{r\in G}$, where each $Z_r$ arises from
 \begin{align*}
 Z^0_r &= \Big\{ \xi \in \prod_{h \in G}  X_{r h^{-1}} \otimes_{\sigma_h} Y_h
  \mid \xi \text{ is finitely supported}\Big\} \\
  &=\Big\{ \xi \in \prod_{h \in G} 
  (X\times \{r h^{-1}\}) \otimes_{\sigma_h} (Y \times \{h\})
  \mid \xi \text{ is finitely supported}\Big\}
  \end{align*}
  Let $r\in G$. For $h, k \in G$ such that $hk=r$,  the right action of $\B_\Sigma$ on $Z_r$ is determined by
\begin{align*} 
\big((x, h) \dot\otimes (y,k)\big) \cdot (a,g) &= (x, h) \dot\otimes ((y,k) \cdot (a,g))\\ 
&= (x, h) \dot\otimes \,(y\alpha_k(a), kg) \in  Z_{rg} 
\end{align*}
 where $g\in G$, $a \in A$, $x \in X$ and $y \in Y$.

 The inner product on $\X\otimes_\sigma\mathcal{Y}$ is determined for $x, x'\in X$, $y, y'\in Y$, $r, r', h, h', k, k'\in G$ satisfying $hk=r$ and $h'k'=r'$ by 

\vspace{-3ex}  
\begin{align*} \big\langle (x, h) \dot\otimes (y,k), (x', h') &\dot\otimes (y',k') \big\rangle_{\B_\Sigma} 
  = \langle (y,k), \sigma(\langle (x, h), (x', h') \rangle_{\B_\Sigma}) (y',k') \rangle_{\B_\Sigma} \\
&  =  \langle (y,k), \sigma\big((\alpha^{-1}_{h} (\langle x, x' \rangle_A),h^{-1}h')\big) (y',k') \rangle_{\B_\Sigma} \\
&  = \langle (y,k), (\varsigma(\alpha_{h^{-1}} (\langle x, x' \rangle_A)) w(h^{-1}h') y', h^{-1}h'k') \rangle_{\B_\Sigma}  \\
& = \big( \alpha^{-1}_k \big(\langle y, \varsigma(\alpha_{h^{-1}} (\langle x, x' \rangle_A)) w(h^{-1}h') y' \rangle_A \big), (hk)^{-1}h'k' \big).
 \end{align*}
Further, the $\B_\Sigma$-action $\rho\dot\otimes\sigma$ on $\X\otimes_\sigma\mathcal{Y}$ is determined  by
\begin{align*}(\rho\dot\otimes\sigma)((a,g)) ((x, h)\dot\otimes (y,k)) &= (\rho((a,g))(x, h))\dot\otimes (y,k) =
 (\varrho(a)v(g)x, gh)\dot\otimes (y,k)
 \end{align*}
for $a \in A, x\in X, y\in Y,   g, h, k, r\in G$ with $hk=r$.

\medskip On the other hand, we can form the equivariant representation $(\varrho\otimes\varsigma, v\otimes w)$ of $\Sigma$ on $X\otimes_{\varsigma}Y$, which is determined  by
 \[ (\varrho\otimes\varsigma)(a) (x\dot\otimes y) = \varrho(a)x\dot\otimes y, \quad (v\otimes w)(g)  (x\dot\otimes y) =  v(g)x\dot\otimes w(g)y\,\]
 for $a\in A$, $g\in G$, $x\in X$ and $y\in Y$, cf.~\cite{BeCo6}.
 
 For $r\in G$, set $(X\otimes_{\varsigma}Y)_r = (X\otimes_{\varsigma}Y)\times \{r\}$. The operations on the  Hilbert $\B_\Sigma$-bundle $ \mathcal{W}=( (X\otimes_{\varsigma}Y)_r)_{r\in G}$ associated to $(\varrho\otimes\varsigma, v\otimes w)$ satisfy 
\begin{align*}
(x\dot\otimes y, r) \cdot (a, g) &= \big((x\dot\otimes y)\,\alpha_r(a), rg\big) =  \big(x\dot\otimes (y\alpha_r(a)), rg\big), \\
\big\langle (x\dot\otimes y, r), \, (x'\dot\otimes y', r')\big\rangle_{\B_\Sigma} &= \big(\alpha_r^{-1}(\langle x\dot\otimes y, x'\dot\otimes y'\rangle_A), r^{-1}r'\big) \\
& = \big(\alpha_r^{-1}(\langle y, \varsigma(\langle x, x'\rangle_A)y'\rangle_A), r^{-1}r'\big)
\end{align*}
and the $\B_\Sigma$-action $\tau$ associated to $(\varrho\otimes\varsigma, v\otimes w)$ on $\mathcal{W}=( (X\otimes_{\varsigma}Y)_r)_{r\in G}$ is determined by
\[ [\tau(a, g)] (x\dot\otimes y, r) = \big( (\varrho\otimes\varsigma)(a)(v\otimes w)(g) (x\dot\otimes y), gr\big) = \big( (\varrho(a)v(g)x)\dot\otimes (w(g) y), gr\big)\]
for $a \in A, g, r, r' \in G, x\in X$ and $ y\in Y$.

\begin{definition} \label{unitequiv}
Let $\B=(B_g)_{g\in G}$ be a Fell bundle over $G$ and $\X=(X_g)_{g \in G}, \X' = (X'_g)_{g \in G}$ be Hilbert $\B$-bundles. We will denote the $\B$-valued inner product on $\X'$ by $\langle \cdot, \cdot\rangle_\B'$. Let $V=(V_g)_{g\in G}$ be a $\B$-bundle map from $\X$ to $\X'$ (i.e., each $V_g$ is a bounded linear map from $X_g$ into $X'_g$).  

We will say that $V$ is a \emph{Hilbert  $\B$-bundle embedding} of $\X$ into $\X'$ 
when the following conditions hold:
 \begin{itemize}
 \item $V_g(x) b = V_{gh}(xb)$
for all $g,h \in G$, $x \in X_g$ and $b \in B_h$.
 \item  $\langle V_g(x), V_h(y)\rangle_\B' = \langle x, y\rangle_\B$
   for all $g, h \in G,  x\in X_g, y\in X_h$. 
   \end{itemize}
 Clearly, each $V_g$ is then isometric on $X_g$. Moreover, if every $V_g$ is surjective (hence bijective), then $V$ is said to be \emph{unitary}, cf.~\cite[Definition 3.9]{BeCo8}.

If $\rho$ (resp.~$\rho'$) is a $\B$-action on $\X$ (resp.~$\X'$), we will say that a $\B$-bundle map $V$ \emph{intertwines $\rho$ and $\rho'$} (equivalently, $V$ is an intertwiner between 
$\rho$ and $\rho'$) whenever it satisfies
  \begin{itemize}
\item $\rho'(b)V_g(x)  = V_{hg}(\rho(b)x)$
for all $g,h \in G$, $x \in X_g$ and $b \in B_h$.
\end{itemize}
We will say for short that $V$ is an \emph{isometric intertwiner} from $\rho$ to $\rho'$ if it is a Hilbert $\B$-bundle embedding  of $\X$ into $\X'$ intertwining $\rho$ and $\rho'$.
 \end{definition}
 
The two tensor actions $\rho\dot\otimes \sigma$ and $\tau$ of $\B_\Sigma$ can be related to each other in the following way. 

\begin{lemma}\label{U_g}
For each $g \in G$, there exists a unitary  $U(g) = (U(g)_r)_{r\in G}: \X \otimes_\sigma \Y \to \X \otimes_\sigma \Y$ satisfying 
\begin{equation}\label{Ugr} 
U(g)_r \big((x,h) \dot\otimes (y,k)\big) =  (x,g^{-1}h) \dot\otimes (w(h^{-1}gh) y,h^{-1}ghk)
\end{equation}
for all $x \in X, y \in Y, h, k, r\in G$ with $hk=r$.  
\end{lemma}

\begin{proof}
Let $ x, x'\in X, y, y'\in Y, g, h,h', k, k',  r, r'\in G$ with $hk=r$ and $h'k'=r'$. Then we have
\begin{align*}
\big\langle &(x,g^{-1}h) \dot\otimes (w(h^{-1}gh) y,h^{-1}ghk), (x',g^{-1}h') \dot\otimes (w(h'^{-1}gh') y,h'^{-1}gh'k')\big\rangle_{\B_\Sigma}\\
&=\big( \alpha^{-1}_{h^{-1}ghk} \big(\langle w(h^{-1}gh) y, \varsigma(\alpha_{h^{-1}g} (\langle x, x' \rangle_A)) w((g^{-1}h)^{-1} g^{-1}h') w(h'^{-1}gh') y \rangle_A \big), r^{-1}r' \big)\\
&=\big( \alpha^{-1}_{h^{-1}ghk} \big(\langle w(h^{-1}gh) y, \varsigma(\alpha_{h^{-1}g} (\langle x, x' \rangle_A)) w(h^{-1}gh') y \rangle_A \big), r^{-1}r' \big)\\
&=\big(\langle w(k^{-1})y, w(k^{-1}h^{-1}g^{-1}h)\varsigma(\alpha_{h^{-1}g} (\langle x, x' \rangle_A)) w(h^{-1}gh') y \rangle_A \big), r^{-1}r' \big)\\
&=\big(\langle w(k^{-1})y, w(k^{-1}h^{-1})\varsigma(\langle x, x' \rangle_A) w(h') y \rangle_A \big), r^{-1}r' \big)\\
&= \big( \alpha^{-1}_k\big(\langle y, w(h^{-1}) \varsigma(\langle x, x' \rangle_A) w(h')y' \rangle_A \big), r^{-1}r' \big)\\
&=\big( \alpha^{-1}_k \big(\langle y, \varsigma(\alpha_{h^{-1}} (\langle x, x' \rangle_A)) w(h^{-1}h') y' \rangle_A \big), r^{-1}r' \big)\\
 &= \big\langle (x, h) \dot\otimes (y,k), (x', h') \dot\otimes (y',k') \big\rangle_{\B_\Sigma}. 
\end{align*}
It follows readily that there exists a $\B_\Sigma$-bundle map $U(g) = (U(g)_r)_{r\in G}: \X \otimes_\sigma \Y \to \X \otimes_\sigma \Y$ satisfying (\ref{Ugr}) for all $r, h\in G, x \in X, y \in Y$ and such that
\[\big\langle U(g)_r z, U(g)_{r'}z'\big\rangle_{\B_\Sigma} = \langle z, z'\rangle_{\B_\Sigma}\]
 for all $r, r'\in G$, $ z\in Z_r$ and $z' \in Z_{r'}$. Moreover,  for all $ b \in A, x\in X, y\in Y, g, h, k, l, r\in G$ with $hk=r$, we have
 \begin{align*}
 U(g)_{rl} \big(((x,h) \dot\otimes&  (y,k))\cdot(b,l)\big)
 = U(g)_{rl} \big((x,h) \dot\otimes  (y\alpha_k(b),kl)\big)\\
  &=  (x,g^{-1}h) \dot\otimes (w(h^{-1}gh)(y\alpha_k(b)), h^{-1}ghkl)\\ 
 &=   (x,g^{-1}h) \dot\otimes ((w(h^{-1}gh)y) \alpha_{h^{-1}gh}(\alpha_k(b)), h^{-1}ghkl)\\  
    &=   (x,g^{-1}h) \dot\otimes ((w(h^{-1}gh) y) \alpha_{h^{-1}ghk}(b),h^{-1}ghkl)\\
  &=   \big((x,g^{-1}h) \dot\otimes (w(h^{-1}gh) y,h^{-1}ghk)\big)\cdot (b,l)\\
  &=   \big(U(g)_r\big((x,h) \dot\otimes  (y,k)\big)\big)\cdot (b,l). 
   \end{align*}
   It follows that  $U(g)_{rl} \big(z\cdot(b, l)\big) = (U(g)_{r}z)\cdot(b,l)$ for all $g, r, l\in G$, $z\in Z_r$ and $b\in A$.
 
 Thus, each $U(g)$ is a Hilbert  $\B_\Sigma$-bundle embedding of $\X\otimes_\sigma \Y$ into itself. To see that each $U(g)$ is unitary, we have to show that  $U(g)_r$ is a bijection from $Z_r$ onto itself for every $g, r\in G$. For this, it suffices to observe that $U(g^{-1})_r$ and  $U(g)_r$ are inverses of each other for any $r$, as a reader can verify without difficulty.
\end{proof}

\begin{proposition}
For each $(a, g)\in \B_\Sigma$, define $\tau'(a,g):   \X \otimes_\sigma \Y \to  \X \otimes_\sigma \Y$ by
\[ \tau'(a,g)z = U(g)_{gr} (\rho \dot\otimes \sigma)(a,g) z\]
for all $r\in G$ and $z\in Z_r$, where $U(g)=(U(g)_r)_{r\in G}$ is the unitary as in Lemma \ref{U_g}. 
Then $\tau'$ is a $\B_\Sigma$-action on $\X \otimes_\sigma \Y$ satisfying
\[\tau'(a,g) \big((x,h) \dot\otimes (y,k)\big) = (\varrho(a)v(g)x, h) \dot\otimes (w(h^{-1}gh)y, h^{-1} ghk) \]
for all $a \in A$, $x\in X$, $y\in Y$, $g, h, r\in G$ with $hk =r$. 
\end{proposition}

\begin{proof} First, for $a \in A$, $x\in X$, $y\in Y$, $g, h, r\in G$ with $hk =r$,  we have
\[\tau'(a,g) \big((x,h) \dot\otimes (y,k)\big) = U(g)_{gr} (\rho \dot\otimes \sigma)(a,g) \big((x,h) \dot\otimes (y,k)\big)
= U(g)_{gr} (\varrho(a)v(g)x, gh)\dot\otimes (y,k)\]
\[=  (\varrho(a)v(g)x, h) \dot\otimes (w(h^{-1}gh) y, h^{-1} ghk)\]
Next, for $g, r\in G$, it is clear that $\tau'(a,g)$ maps each $Z_r$ to $Z_{gr}$ and that
$((a,g), z) \mapsto \tau'(a,g)z $ is a bilinear map from $(A\times\{g\}) \times Z_r$ into $Z_{gr}$. Below, we check the other conditions for $\tau'$ to be a $\B_\Sigma$-action. 
 
Let $a, a', b \in A$, $g, g', r, h, k\in G$ with $hk=r$. Then we have 
\begin{align*} \tau'\big((a,g)\cdot (a', g')) \big((x, h) \dot\otimes (y, k)\big) &= \tau'(a\alpha_g(a'), gg') \big((x,h) \dot\otimes (y,k)\big)\\
&= \big(\varrho(a\alpha_g(a'))v(gg')x, h)\dot\otimes (w(h^{-1}gg'h)y, h^{-1} gg'hk).
\end{align*}
On the other hand,  
\begin{align*} \big(\tau'(a,g)\tau'(a', g')&\big) \big((x,h) \dot\otimes (y,k)\big)=
\tau'(a,g)\big((\varrho(a')v(g')x, h) \dot\otimes (w(h^{-1}g'h)y, h^{-1} g'hk)\big)\\
&=(\varrho(a) v(g) \varrho(a')v(g')x, h) \dot\otimes (w(h^{-1}gh)w(h^{-1}g'h)y, h^{-1} ghh^{-1} g'hk)\\
&=(\varrho(a\alpha_g(a'))v(gg')x, h) \dot\otimes (w(h^{-1}gg'h)y, h^{-1} gg'hk).
\end{align*}
It follows that  $\tau'\big((a,g)\cdot (a', g'))  = \tau'(a,g)\tau'(a', g')$. 

Moreover, we have
\begin{align*}
\langle  \tau'(a,g)&(\big((x,h) \dot\otimes (y,k)\big), (x',h') \dot\otimes (y',k')\big) \rangle_{\B_\Sigma}\\
&=\langle (\varrho(a)v(g)x,h) \dot\otimes (w(h^{-1}gh)y, h^{-1} ghk), (x', h') \dot\otimes (y',k')\big) \rangle_{\B_\Sigma}\\
&= \big( \alpha^{-1}_{h^{-1} ghk} \big(\langle w(h^{-1}gh)y, \varsigma(\alpha_{h^{-1}} (\langle \varrho(a)v(g)x, x' \rangle_A)) w(h^{-1}h') y' \rangle_A \big), (ghk)^{-1} h'k' \big)\\
&=  \big( \alpha^{-1}_{k} \big(\langle y, w(h^{-1}g^{-1}h) \varsigma(\alpha_{h^{-1}} (\langle \varrho(a)v(g)x, x' \rangle_A)) w(h^{-1}h') y' \rangle_A \big), (ghk)^{-1} h'k' \big)\\
&=  \big( \alpha^{-1}_{k} \big(\langle y, w(h^{-1}g^{-1}) \varsigma(\langle \varrho(a)v(g)x, x' \rangle_A) w(h') y' \rangle_A \big), (hk)^{-1} g^{-1}h'k' \big)\\
&= \big( \alpha^{-1}_k \big(\langle y, w(h^{-1})\varsigma(\alpha_{g^{-1}}(\langle \varrho(a)v(g) x, x' \rangle_A)) w(g^{-1}h') y' \rangle_A \big), (hk)^{-1} g^{-1}h'k' \big).
\end{align*}
On the other hand,
\begin{align*}
\big\langle (x,h) & \dot\otimes (y,k)\big), \tau'((a,g)^*)\big((x',h') \dot\otimes (y',k')\big) \big\rangle_{\B_\Sigma} \\
& = \big\langle  (x,h) \dot\otimes (y,k) , \tau'(\alpha_{g^{-1}}(a^*), g^{-1})\big((x', h') \dot\otimes (y',k')\big) \rangle_{\B_\Sigma} \\
& = \big\langle  (x,h) \dot\otimes (y,k) , (\varrho(\alpha_{g^{-1}}(a^*)v(g^{-1})x', h') \dot\otimes (w(h'^{-1}g^{-1}h') y', h'^{-1} g^{-1}h'k') \big\rangle_{\B_\Sigma} \\
&= \big( \alpha^{-1}_k \big(\langle y, \varsigma(\alpha_{h^{-1}} (\langle x, \varrho(\alpha_{g^{-1}}(a^*))v(g^{-1})x' \rangle_A)) w(h^{-1}h') w(h'^{-1}g^{-1}h') y' \rangle_A \big), (hk)^{-1} g^{-1}h'k' \big)\\
&= \big( \alpha^{-1}_k \big(\langle y, w(h^{-1})\varsigma(\langle x, \varrho(\alpha_{g^{-1}}(a^*))v(g^{-1})x' \rangle_A) w(g^{-1}h') y' \rangle_A \big), (hk)^{-1} g^{-1}h'k' \big)\\
&= \big( \alpha^{-1}_k \big(\langle y, w(h^{-1})\varsigma(\langle\varrho(\alpha_{g^{-1}}(a)) x, v(g^{-1})x' \rangle_A) w(g^{-1}h') y' \rangle_A \big), (hk)^{-1} g^{-1}h'k' \big)\\
&= \big( \alpha^{-1}_k \big(\langle y, w(h^{-1})\varsigma(\langle v(g^{-1})\varrho(a)v(g) x, v(g^{-1})x' \rangle_A)) w(g^{-1}h') y' \rangle_A \big), (hk)^{-1} g^{-1}h'k' \big)\\
&= \big( \alpha^{-1}_k \big(\langle y, w(h^{-1})\varsigma(\alpha_{g^{-1}}(\langle \varrho(a)v(g) x, x' \rangle_A)) w(g^{-1}h') y' \rangle_A \big), (hk)^{-1} g^{-1}h'k' \big).
\end{align*}
It follows that $\langle  \tau'(a,g) z, z'\rangle_{\B_\Sigma} = \langle  z,  \tau'((a,g)^*)z'\rangle_{\B_\Sigma}$ for all $z, z' \in \X\otimes_\sigma \Y$. 

Finally, we have
\begin{align*}
\big(\tau'(a,g)& \big((x, h) \dot\otimes (y,k)\big) \big) \cdot (b,l) = \big((\varrho(a)v(g)x, h) \dot\otimes (w(h^{-1}gh) y, h^{-1} ghk)\big)  \cdot (b,l) \\
&= (\varrho(a)v(g)x, h) \dot\otimes ((w(h^{-1}gh) y) \alpha_{h^{-1} ghk}(b), h^{-1} ghkl),
\end{align*}
while
\begin{align*}
\tau'(a,g)& \big(\big((x, h) \dot\otimes (y,k)\big)  \cdot (b,l) \big) = \tau'(a,g) \big((x, h) \dot\otimes \,(y\alpha_k(b), kl) \big)\\
&=(\varrho(a)v(g)x, h) \dot\otimes \,(w(h^{-1}gh)(y\alpha_k(b)), h^{-1}ghkl)\\
&=(\varrho(a)v(g)x, h) \dot\otimes \,((w(h^{-1}gh)y) \alpha_{h^{-1}gh}(\alpha_k(b)), h^{-1}ghkl)\\
&= (\varrho(a)v(g)x, h) \dot\otimes ((w(h^{-1}gh) y) \alpha_{h^{-1} ghk}(b), h^{-1} ghkl).
\end{align*}
It follows that $(\tau'(a,g) z)\cdot(b,k) = \tau'(a,g) (z\cdot(b,k))$ for all $z \in  \X\otimes_\sigma \Y$.

This shows that $\tau'$ is a $\B_\Sigma$-action on  $\X\otimes_\sigma \Y$.
\end{proof}

\begin{remark} \label{1-coc}
 Let $\B=(B_g)_{g\in G}$ be a Fell bundle over $G$ and $\rho$ be a $\B$-action on a Hilbert $\B$-bundle $\X$. Assume $U(g) =(U(g)_r)_{r\in G}$ is a unitary on $\X$ for each $g$ in $G$
 such that $U(g)_rU(g^{-1})_r = {\rm Id}_{X_r}$ for every $g, r \in G$. Define then $\rho':\B\to \mathcal{F}(\X)$ by
\[ \rho'(a) x = U(g)_{gr}(\rho(a)x)\]
for all $g,r \in G$, $a\in B_g$ and $x\in X_r$. One may now wonder under which conditions $\rho'$ is still a $\B$-action on $\X$. 
In that case, such a family $\{U(g)\}_{g\in G}$ of unitaries on $\X$ could then be called a $1$-cocycle for $\rho$, and $\rho'$ a 1-cocycle perturbation of $\rho$.

In fact, it is not difficult to check that $\rho'$ becomes a $\B$-action if and only if the $U(g)$'s satisfy the following conditions:
\begin{itemize}
\item[1)] $U(gh)_{ghr} \,\rho(ab) = U(g)_{ghr}\,\rho(a)\,U(h)_{hr}\,\rho(b)$ on $X_r$,
\item[2)] $U(g)_{gr} \rho(a) = \rho(a) U(g)_r$ on $X_r$
\end{itemize}
for all $g, h, r\in G$, $a\in B_g$ and $b\in B_h$. 
We note that by combining 1) and 2), we get
\[ \rho(ab)\,U(gh)_r = U(g)_{ghr}\,\rho(ab) \,U(h)_{r} \, \, \text{on $X_r$, since} \]
\[ \rho(ab)\,U(gh)_r = U(gh)_{ghr} \,\rho(ab) =U(g)_{ghr}\,\rho(a)\,U(h)_{hr}\,\rho(b) = U(g)_{ghr}\,\rho(ab) \,U(h)_{r}.\]
For completeness, we include the proof of the equivalence stated above. 

\noindent Let $g, h, r\in G$, $a\in B_g$ and $b\in B_h$. Then on $X_r$ we have 
$\rho'(a)\rho'(b) = U(g)_{ghr}\rho(a)U(h)_{hr}\rho(b)$, while $\rho'(ab) =U(gh)_{ghr}\rho(ab)$. So $\rho'(a)\rho'(b) = \rho'(ab) $ if and only if 1) holds. Further, if $x\in X_r$, then 
\[\rho'(a)(xb)= U(g)_{grh}(\rho(a)(xb)) = U(g)_{grh}((\rho(a)x)b) = (U(g)_{gr}\rho(a)x)b = (\rho'(a)x)b,\]
i.e., property (iv) for $\rho'$ to be a $\B$-action always holds. This is clearly true also for property (i).
Finally, if $s\in G$ and $y\in X_s$,
\begin{align*} 
\langle \rho'(a)x, y\rangle_\B = \langle U(g)_{gr}\rho(a)x, y\rangle_\B 
& = \langle U(g)_{gr}\rho(a)x, U(g)_s U(g^{-1})_s y\rangle_\B 
=  \langle \rho(a)x, U(g^{-1})_s y\rangle_\B \\
& =  \langle x, \rho(a^*)U(g^{-1})_s y\rangle_\B  ,
\end{align*}
while
\[\langle x, \rho'(a^*)y\rangle_\B = \langle x, U(g^{-1})_{g^{-1}s}\rho(a^*)y\rangle_\B.\]
Thus, $\rho'$ satisfies property (iii) if and only if  $\rho(a^*)U(g^{-1})_s = U(g^{-1})_{g^{-1}s}\rho(a^*)$, i.e., 2) holds
(replacing $a^*$ with $a$ and  $g^{-1}$ with $g$).
\end{remark}
In view of Remark \ref{1-coc}, we may think of $\tau'$ as a  ``cocycle deformation'' of $ \rho \dot\otimes \sigma$.

\begin{theorem}\label{dynsyst-emb}
For each $h \in G$,  there is a Hilbert $\B_\Sigma$-bundle embedding $V^h=(V^h_r)_{r \in G}$ of 
$ \mathcal{W} = ( (X\otimes_{\varsigma}Y)_r)_{r\in G}$ into $\X\otimes_\sigma \Y = (Z_r)_{r\in G}$
intertwining the $\B_\Sigma$-actions $\tau=\rho_{(\varrho\otimes\varsigma, v\otimes w)}$ and $\tau'$.
\end{theorem}

\begin{proof}
Let $h\in G$. For each $r\in G$, let $V^{h, 0}_r:X\odot Y \times\{r\} \to Z_r$ be the linear map determined by
 \[  V^{h,0}_r(x\otimes y, r) =  (x, h) \dot\otimes (w(h^{-1})y, h^{-1}r) \quad (\text{identified as an element in } Z_r)\]
 for all $x \in X$ and $y\in Y$.  
   For $a\in A, x\in X, y\in Y$ and $r\in G$, we have
   \begin{align*}
    V^{h, 0}_r((x\cdot a)\otimes y, r) &= (x\cdot a, h) \dot\otimes (w(h^{-1})y, h^{-1}r)\\
    & = ((x, h)\cdot (\alpha_{h^{-1}}(a),e)) \dot\otimes (w(h^{-1})y, h^{-1}r)\\
    &= (x, h) \dot\otimes ( [\sigma(\alpha_{h^{-1}}(a), e)](w(h^{-1})y, h^{-1}r))\\
    & = (x, h) \dot\otimes (\varsigma(\alpha_{h^{-1}}(a))w(h^{-1}) y, h^{-1}r)\\
    &= (x, h) \dot\otimes w(h^{-1})\varsigma(a) y, h^{-1}r)\\
    &= V^{h, 0}_r(x\otimes \varsigma(a)y, r)
    \end{align*} 
   It follows that $V^{h,0}_r$ induces a linear map  $V^h_r:X\otimes_A Y \times\{r\} \to Z_r$ satisfying  \[  V^h_r(x\dot\otimes y, r) = (x, h) \dot\otimes (w(h^{-1})y, h^{-1}r)\] 
  for all $x \in X$ and $y\in Y$. 
   Now, let $x, x'\in X, y, y'\in Y$ and $r, r'\in G$. Note that 
 \begin{align*}
 \langle V^h_r(x\dot\otimes y, r), V^h_{r'}(x'\dot\otimes y', r')\rangle_{\B_\Sigma}
 &=\langle (x, h) \dot\otimes (w(h^{-1})y, h^{-1}r), (x', h) \dot\otimes (w(h^{-1})y', h^{-1}r')\big\rangle_{\B_\Sigma}\\
 &= \big( \alpha^{-1}_{h^{-1}r} \big(\langle w(h^{-1}) y, \varsigma(\alpha_{h^{-1}} (\langle x, x' \rangle_A))  w(h^{-1})y' \rangle_A \big), r^{-1}r' \big)\\
 &= \big( \alpha^{-1}_{h^{-1}r}  \big(\langle w(h^{-1}) y, w(h^{-1}) \varsigma(\langle x, x' \rangle_A) y' \rangle_A \big), r^{-1}r' \big)\\
  &= \big( \alpha^{-1}_{r} \big(\langle y, \varsigma(\langle x, x' \rangle_A)  y' \rangle_A \big), r^{-1}r' \big)\\
&=  \langle (x \otimes y, r), (x'\otimes y', r')\big\rangle_{\B_\Sigma}.
\end{align*}
By linearity, we get that 
   \begin{equation}\label{vrr}
 \langle V^h_r(z), V^h_{r'}(z')\rangle_{\B_{\Sigma}} = \langle z, z'\rangle_{\B_\Sigma}
 \end{equation}
   for all $r, r' \in G,  z\in X\otimes_A Y \times \{r\}$ and $ z' \in X\otimes_A Y \times \{r'\}$.
   In particular, $V^h_r$ is isometric, so it extends uniquely to an isometry, also denoted by $V^h_r$, from  $(X\otimes_{\varsigma}Y)_r=X\otimes_{\varsigma}Y \times\{r\}$ to $Z_r$. By continuity, we get that (\ref{vrr}) holds for all $r, r' \in G,  z\in (X\otimes_{\varsigma}Y)_r$ and $ z' \in (X\otimes_{\varsigma}Y)_{r'}$.

Moreover, for any $r, g\in G, x\in X, y\in Y$ and $a\in A$,  we get 
 \begin{align*}
 V^h_r\big( x \dot\otimes y, r\big)\cdot (a,g)&=
 \big((x, h) \dot\otimes (w(h^{-1})y, h^{-1}r)\big)\cdot (a,g)\\
 & = (x, h)\dot\otimes \,((w(h^{-1})y)\alpha_{h^{-1}r}(a), h^{-1}rg)\\
  & = (x, h) \dot\otimes (w(h^{-1})(y \alpha_r(a)), h^{-1}rg) \\
&= V^h_{rg}\big(x\dot\otimes (y\alpha_r(a), rg\big) 
= V^h_{rg}\big((x \dot\otimes y, r)\cdot (a,g)\big),
 \end{align*}
 and also
 \begin{align*}
\tau'(a,g) V_r^h(x \dot\otimes y, r) & = \tau'(a,g) \big((x, h) \dot\otimes (w(h^{-1})y, h^{-1}r)\big) \\
& = (\rho(a)v(g)x,h) \dot\otimes (w(h^{-1}gh) w(h^{-1})y, h^{-1}gr) \\
& = (\rho(a)v(g)x,h) \dot\otimes (w(h^{-1})w(g)y, h^{-1}gr) \\
& = V_{gr}^h  (\rho(a)v(g)x \dot\otimes w(g)y, gr) \\
& = V_{gr}^h (\tau(a,g) (x \dot\otimes y,r)).
\end{align*}
It is then clear that $V^h:=(V^h_r)_{r\in G}$ is a Hilbert $\B_\Sigma$-bundle embedding of 
$\mathcal{W} =\big((X\otimes_{\varsigma}Y)_r\big)_{r\in G}$
into $\X\otimes_\sigma \Y=(Z_r)_{r\in G}$ intertwining $\tau$ and $\tau'$.
\end{proof}

Another evidence illustrating that the $\B_\Sigma$-actions $\tau=\rho_{(\varrho\otimes\varsigma, v\otimes w)}$ and $\rho\dot\otimes\sigma$ are related comes from the comparison of the associated (elementary) diagonal coefficients over the unit fibres, as we discuss below.

Let $h\in G$, $x \in X$ and  $y\in Y$. Then, thinking of $(x, h) \dot\otimes (y, h^{-1})$ as an element in the unit fibre of $\X\otimes_\sigma\mathcal{Y}$, the associated (elementary) diagonal coefficient $T = (T_g)_{g\in G}$ of $\rho\dot\otimes \sigma$ is given by
\begin{align*} T_g((a,g)) &= \langle (x, h) \dot\otimes (y, h^{-1}), (\rho\dot\otimes\sigma)((a,g)) ((x, h) \dot\otimes (y, h^{-1})) \rangle_{\B_\Sigma}\\
&=\langle (x, h) \dot\otimes (y, h^{-1}),  (\varrho(a)v(g)x, gh) \dot\otimes (y, h^{-1}) \rangle_{\B_\Sigma}\\
&= \big(  \alpha_{h}(\langle y ,  \varsigma(\alpha_{h^{-1}}(\langle x , \varrho(a)v(g)x\rangle_A))w(h^{-1}gh) y\rangle_A), g\big)
\\&= \big(\langle w(h) y ,  w(h)\varsigma(\alpha_{h^{-1}}(\langle x , \varrho(a)v(g)x\rangle_A)w(h^{-1}gh) y\rangle_A, g\big)
\\&= \big(\langle w(h) y ,  \varsigma(\langle x , \varrho(a)v(g)x\rangle_A)w(g)\, w(h) y\rangle_A, g\big)\,
\end{align*}
for every $g\in G$ and $a\in A$. 

On the other hand, $(x\dot\otimes y , e)$ belongs to the unit fibre $(X\otimes_{\varsigma}Y)_e$ of $\mathcal{W}$, and we have
\begin{align*}
\big\langle (x\dot\otimes y , e), \, \tau((a, g)) (x\dot\otimes y , e)\big\rangle_{\B_\Sigma} & = 
\big\langle (x\dot\otimes y , e), \, (\varrho(a)v(g)x\dot\otimes w(g) y, g) \big\rangle_{\B_\Sigma} \\
& = \big(\langle  y, \varsigma(\langle x, \varrho(a)v(g)x\rangle_A)w(g) y\rangle, g\big).
\end{align*}
Replacing $y$ by $w(h)y$ in this computation, we see that 
\[ \big\langle (x\dot\otimes w(h)y , e), \, \tau((a, g)) (x\dot\otimes w(h)y , e)\big\rangle_{\B_\Sigma} = T_g((a, g))\] 
for every $g\in G$ and $a\in A$. This shows that an (elementary) diagonal coefficient $T$ of $\rho\dot\otimes \sigma$ over the unit fibre of $\X\otimes_\sigma\mathcal{Y}$ can also be expressed as an (elementary) diagonal coefficient of $\tau$ over the unit fibre of  $\mathcal{W}=( (X\otimes_{\varsigma}Y)_g)_{g\in G}$. Moreover, the converse statement is easily seen to be true as well.

 We also note that 
 the diagonal coefficient $\tilde T = (\tilde T_g)_{g\in G}$ of the equivariant representation $(\varrho\otimes\varsigma, v\otimes w)$ of 
 $\Sigma$ on $X\otimes_{\varsigma} Y$ associated to the vector $x\dot\otimes w(h)y$  is given by 
 \begin{align*} \tilde T_g(a) &= \big\langle x\dot\otimes w(h)y, \,(\varrho\otimes\varsigma)(a) (v\otimes w)(g) (x\dot\otimes w(h)y)\big\rangle_A\\
 &=  \big\langle x\dot\otimes w(h)y, \varrho(a) v(g)x \dot\otimes w(g)w(h)y\big\rangle_A \\
 & =\big\langle  w(h)y, \varsigma\big(\langle x, \varrho(a) v(g)x\rangle_A\big) w(g)w(h)y\big\rangle_A. \end{align*}

  Letting $T' = (T'_g)_{g\in G}$ be the induced $\B_\Sigma$-bundle map given by $T'_g((a,g)) = (\tilde T_g(a) , g)$ for all $g\in G$ and $a\in A$,
  we get
\[T'_g((a,g)) = \big(\big\langle  w(h)y, \varsigma\big(\langle x, \varrho(a) v(g)x\rangle_A\big) w(g)w(h)y\big\rangle_A, g\big) = T_g((a, g)),\]
i.e., $T'$ agrees with $T$. This shows that the positive definite $\B_\Sigma$-bundle maps obtained from  $\rho\dot\otimes \sigma$ 
as above can be recovered by working only with tensor product of equivariant representations of $\Sigma$.

\begin{remark}
Although the relationship between $\rho_{(\varrho,v)}\dot\otimes\rho_{(\varsigma, w)}$ and $\rho_{(\varrho,v) \otimes (\varsigma,w)}$ discussed above is more complicated than expected, it is somewhat reassuring to observe that there is a perfect matching when dealing with direct sums. Namely, it can be checked without too much trouble that 
if, for each $i$ in some index set $I$, $(\varrho_i,v_i)$ is an  an equivariant representation of $\Sigma$ on a Hilbert $A$-module $X_i$, 
and $\big(\oplus_{i \in I} \varrho_i, \oplus_{i \in I} v_i\big)$ is the direct sum of these equivariant representations on $\oplus_{i \in I} X_i$ (see \cite[Section 2]{BeCo6}),  then
the $\B_\Sigma$-action $\rho_{(\oplus_{i \in I} \varrho_i, \oplus_{i \in I} v_i)}$ on the Hilbert $\B_\Sigma$-bundle  $\Big(\big(\oplus_{i \in I} X_i\big) \times \{g\}\Big)_{g \in G}$
is unitarily equivalent (in the sense of Definition \ref{unitequiv}) to the direct sum action $\oplus_{i \in I} \rho_{(\varrho_i,v_i)}$ on the Hilbert $\B_\Sigma$-bundle $\Big(\oplus_{i \in I} (X_i \times \{g\})\Big)_{g \in G}$  defined  in Section \ref{DirectSum}.
\end{remark}

\subsection{Tensor products of actions of Fell bundles} \label{tp-general}
We will now discuss the more general case. Let $G, H, K$ be groups,  $\varphi \in {\rm Hom}(G,H)$ and $\psi \in {\rm Hom}(H, K)$.  
 Consider some Fell bundles  $\A=(A_g)_{g\in G}$, $\B=(B_h)_{h\in H}$, and $\C=(C_k)_{k\in K}$, and  assume we are given  
an $(\A, \varphi)$-action $\rho$ on  a  Hilbert $\B$-bundle $\X=(X_h)_{h\in H}$ and  
 a $(\B,\psi)$-action $\sigma$ on a  Hilbert $\C$-bundle $\Y = (Y_k)_{k\in K}$.

Then, proceeding essentially along the same lines as in subsection \ref{tp-b-actions}, we may  form a  Hilbert $\C$-bundle $\X\boxtimes_{\sigma} \Y$,  
and define an $(\A,\psi\circ \varphi)$-action $\rho \dot\boxtimes \sigma$ on  $\X\boxtimes_\sigma \Y$. 
The construction follows the one described in Remark \ref{boxaction} for the case where $G=H=K$, $\varphi= \psi ={\rm id}_G$ and $\A=\B=\C$, as it still makes sense in the general situation. For completeness, we provide some details, indicating only changes when necessary.

For each $b \in B_e$ and $k \in K$, the restriction of $\sigma(b)$ to $Y_k$ is an adjointable map from $Y_k$ into itself (considered as a  Hilbert $C_e$-module), which we denote by $\sigma_k(b)$. Then the map $b\mapsto \sigma_k(b)$ is a $*$-homomorphism $\sigma_k$ from $B_e$ into $\L_{C_e}(Y_k)$.
  Thus, for any $h \in H, k \in K$,  considering $X_h$ as  a Hilbert $B_e$-module and $Y_k$ as a $C_e$-module, their internal tensor product 
 $X_h \otimes_{\sigma_k} Y_k$ is then a  Hilbert $C_e$-module such that
\[(x \cdot b) \dot\otimes y = x \,\dot\otimes \sigma_k(b)y \text{  and  } \langle x \dot\otimes y, x' \dot\otimes y' \rangle_{C_e} = \langle y, \sigma_k(\langle x, x' \rangle) y' \rangle_{C_e}\] for all  $x,x' \in X_h$,  $y,y' \in Y_k$ and $b \in B_e$. 

 For every $r \in K$, set 
  \[W^0_r := \Big\{ \xi \in \prod_{h \in H} (X_h \otimes_{\sigma_{\psi(h)^{-1}r} } Y_{\psi(h)^{-1}r}) \mid \xi \text{ is finitely supported}\Big\}\]
 and $\mathcal{W}^0 :=\coprod_{r\in G} W^0_r$.
 If $x \in X_{s}$ and $y\in Y_k$ for some $s\in H$ and $k\in K$ such that $\psi(s)k =r$, we will frequently use $x\dot\otimes y$ 
 to denote the element $\xi_{x, y} \in W^0_r$ given by 
\[ \xi_{x, y}(h) = \begin{cases}x\dot\otimes y & \text{if } h=s,\\ \, \, 0 &\text{if } h\neq s.
\end{cases} 
\]
 
  We will now define a right action of $\C$ on $\mathcal{W}^0$,
 $(z,c)\mapsto z\cdot c$, and a $\C$-valued semi-inner product on $\mathcal{W}^0$,
 $(z, z') \mapsto \langle z, z'\rangle_\C$, which will turn $\mathcal{W}^0=(W^0_r)_{r\in K}$ into a 
semi-inner product $\C$-bundle.

Let $h \in H$, $r\in K$,  $x \in X_{h}$, $y \in Y_{\psi(h)^{-1}r}$ and  $c \in C_k$. 
  Note  that $y \cdot c \in Y_{\psi(h)^{-1}rk}$, so we can set
  \[ (x \dot\otimes y) \cdot c := x \dot\otimes (y \cdot c) \in X_{h} \otimes_{\psi(h)^{-1}rk} Y_{\psi(h)^{-1}rk}.\]
One routinely checks that the map $x\dot\otimes y \mapsto (x \dot\otimes y) \cdot c$
extends to a linear map $z\mapsto z \cdot c$ from $X_{h} \otimes_{\sigma_{\psi(h)^{-1}r}} Y_{\psi(h)^{-1}r}$ into $X_{h} \otimes_{\sigma_{\psi(h)^{-1}rk}} Y_{\psi(h)^{-1}rk}$ such that $\|z\cdot c\| \leq \|z\| \, \|c\|$.
We then obtain a bilinear map  $(\xi, c) \mapsto \xi \cdot c$ from $W^0_r\times C_k$ into $W^0_{rk}$ by setting
\[ (\xi \cdot c)(h) := \xi(h)\cdot c \quad \text{for all } h\in H.\]

Next, let $h, h'\in H$, $r, r'\in K$. Then one checks that there is a unique well-defined bilinear map 
 $(z, z')\mapsto \langle z, z' \rangle_\C$ from $ (X_h \otimes_{\sigma_{\psi(h)^{-1}r} } Y_{\psi(h)^{-1}r}) \times (X_{h'} \otimes_{\sigma_{\psi(h')^{-1}r'} } Y_{\psi(h')^{-1}r'})$ into  $C_{r^{-1}r'}$ such that  
 \[\big\langle x \dot\otimes y, x' \dot\otimes y' \big\rangle_\C  = \big\langle y, \sigma(\langle x,x' \rangle_\B) \,y' \big\rangle_\C \]
 for all $x \in X_{h}$, $y \in Y_{\psi(h)^{-1}r}$, $x' \in X_{h'}$ and $y' \in Y_{\psi(h')^{-1}r'}$. Indeed, since $\langle x,x' \rangle_\B \in B_{h^{-1}h'}$, we get 
 \[\sigma(\langle x,x' \rangle_\B) y' \in Y_{\psi(h^{-1}h')\psi(h')^{-1}r'} = Y_{\psi(h)^{-1}r'},\]
 hence that 
 \[\big\langle y, \sigma(\langle x,x' \rangle_\B)\, y' \big\rangle_\C \in C_{(\psi(h)^{-1}r)^{-1}\psi(h)^{-1}r'} = C_{r^{-1}r'}.\]
 Moreover, as 
 \[\|\langle z, z' \big\rangle_\C\| \, \leq \, \|z\| \, \|z'\|\]
for all $z \in X_h \otimes_{\sigma_{\psi(h)^{-1}r} } Y_{\psi(h)^{-1}r}$ and $z'\in X_{h'} \otimes_{\sigma_{\psi(h')^{-1}r'} } Y_{\psi(h')^{-1}r'}$, 
we get a bilinear map\\ $(\xi, \xi')\mapsto 
  \langle \xi, \xi' \big\rangle_\C$ from $W_r^0 \times W_{r'}^0$ into $C_{r^{-1}r'}$ given by
  \[  \langle \xi, \xi' \big\rangle_\C = \sum_{h, h'\in G}  \langle \xi(h), \xi'(h') \big\rangle_\C,\]
  for all $\xi \in W_r^0$ and $\xi'\in W_{r'}^0$. This map is easily seen to satisfy
  \[  \langle \xi, \xi'\cdot c \big\rangle_\C =  \langle \xi, \xi' \big\rangle_\C \, c \quad \text{ and }  \big(\langle \xi, \xi' \big\rangle_\C\big)^* =  \langle \xi', \xi \big\rangle_\C\]
  for all $\xi \in W_r^0, \xi'\in W_{r'}^0$ and $c \in \C$. As in subsection \ref{tp-b-actions}, one may adapt the argument in the proof of \cite[Lemma 5.9]{AF} (or the one in the proof of Lemma \ref{ip-ineq}) to show that $\langle \xi, \xi \big\rangle_\C \, \geq 0$ (in $C_e$) for every $\xi \in W_r^0$. Thus $W_r^0$ is a
  semi-inner product $C_e$-module with  seminorm given by $\|\xi\| := \|\langle \xi, \xi \big\rangle_\C \|^{1/2}$, satisfying that
  \[ \|\xi\cdot c\|\leq  \|\xi\|\, \| c\| \quad \text{ and} 
  \quad \|\langle \xi, \xi' \big\rangle_\C \| \leq \, \|\xi\| \|\xi'\|\]
  for all $\xi \in W_r^0, \xi'\in W_{r'}^0$ and $c\in \C$. 
 
 We have thereby constructed maps  $\mathcal{W}^0\times \B \to \mathcal{W}^0$ and  $\mathcal{W}^0 \times \mathcal{W}^0 \to \B$ 
  satisfying all the properties required for $\mathcal{W}^0$ to be a semi-inner product $\C$-bundle. After separation and completion,
   we obtain a Hilbert $\C$-bundle ${\mathcal W} = (W_r)_{r \in K}$, that we will denote by $\X\boxtimes_\sigma \Y$. 
 Recall that if $h\in H$ and $r\in K$, we may think of $X_h\otimes\sigma_{\psi(h)^{-1}r} Y_{\psi(h)^{-1}r}$ as a subspace of $W_r^0$. 
 For $x \in X_h$ and $y\in Y_{\psi(h)^{-1}r}$, we will also write $x\dot\otimes y$ to denote its image in $W_r$.
 
  The  $(\A, \psi\circ\varphi)$-action $\rho\boxtimes \sigma$ on $\X\boxtimes_\sigma \Y$
  is defined as follows. 
Let $g\in G, r \in K$ and  $a \in A_g$. Moreover, let $h\in H$, $x \in X_{h}$ and $y \in Y_{\psi(h)^{-1}r}$, so
  $x \dot\otimes y \in W_r^0$. Then $\rho(a) x \in X_{\varphi(g)h}$, so 
 \[(\rho(a) x)\dot\otimes y \in X_{\varphi(g)h} \otimes_{\sigma_{\psi(h)^{-1}r}} Y_{\psi(h)^{-1}r} =
  X_{\varphi(g)h} \otimes_{\sigma_{\psi(\varphi(g)h)^{-1} \psi(\varphi(g))r}} Y_{\psi(\varphi(g)h)^{-1}\psi(\varphi(g))r}.\] 
  We therefore set
 \[(\rho(a)\otimes \iota) (x \dot\otimes y) := (\rho(a) x) \dot\otimes y\,.\]
 Arguing as in subsection \ref{tp-b-actions}, we get a linear map \[\rho(a)\otimes \iota: X_{h} \otimes_{\psi(h)^{-1}r} Y_{\psi(h)^{-1}r} \to 
  X_{\varphi(g)h} \otimes_{\sigma_{\psi(\varphi(g)h)^{-1} \psi(\varphi(g))r}} Y_{\psi(\varphi(g)h)^{-1}\psi(\varphi(g))r},\]
 and in turn a linear map $(\rho\boxtimes \sigma)(a) : W^0_r\to W^0_{\psi(\varphi(g))r}$ by setting 
  \[ \big([(\rho\boxtimes \sigma)(a)]\xi\big)(h) := (\rho(a)\otimes \iota)\xi(\varphi(g)^{-1}h)\]
 for all $\xi\in W^0_r$ and $h\in H$.
 After adapting the proof of Lemma \ref{ip-ineq} to the present situation, one can conclude that $\rho\boxtimes \sigma$ is a pre-$(\A, \psi\circ \varphi)$-action on $\mathcal{W}^0$ (see Remark \ref{pre-act-def}). Thus, it can be extended to an $(\A, \psi\circ \varphi)$-action $\rho\dot\boxtimes \sigma$ on $\mathcal{W}=\X\boxtimes_\sigma \Y$.  
 
When $G=H=K$, $\varphi= \psi ={\rm id}_G$, and $\A=\B=\C$, we recover the construction in subsection \ref{tp-b-actions} (up to unitary equivalence). 
For referential purposes, we sum up our discussion in this subsection.

\begin{theorem} \label{gen-tensor-action}
 Let $G, H, K$ be groups,  $\varphi \in {\rm Hom}(G,H)$ and $\psi \in {\rm Hom}(H, K)$,  
and  $\A=(A_g)_{g\in G}$, $\B=(B_h)_{h\in H}$, $\C=(C_k)_{k\in K}$ be Fell bundles.  
Further, let $\rho$ be an $(\A, \varphi)$-action on  a  Hilbert $\B$-bundle $\X=(X_h)_{h\in H}$ and  
$\sigma$ be a $(\B,\psi)$-action  on a  Hilbert $\C$-bundle $\Y = (Y_k)_{k\in K}$.
 
 Then there exists a Hilbert $\C$-bundle 
 $\X\boxtimes_\sigma \Y = (W_r)_{r\in K}$ such that for each $r\in K$, $W_r$ contains a subset $\{x \dot\otimes y: 
 x \in X_h, y\in  Y_{\psi(h)^{-1}r}, \text{ for some } 
 h\in H\}$ whose span is dense in $W_r$, with operations satisfying 
 \[(x \dot\otimes y)\cdot c = x \dot\otimes (y\cdot c), 
 \quad \langle x \dot\otimes y, x' \dot\otimes y'\rangle_\C = \big\langle y, \sigma(\langle x, x'\rangle_\B)y'\big\rangle_\C\] 
 for all $ h, h'\in H,  r, r'$,  $x \in X_{h}$,  $y \in Y_{\psi(h)^{-1}r}$, $c \in \C$, $x' \in X_{h'}$ and $y' \in Y_{{\psi(h')^{-1}r'}}$. 
 
 Moreover, there is an $(\A, \psi\circ \varphi)$-action $\rho \dot\boxtimes \sigma$ on 
 $\X\boxtimes_\sigma \Y$ satisfying that 
 \[
(\rho \dot\boxtimes \sigma)(a) (x \dot\otimes y) = (\rho(a) x) \dot\otimes y
 \]
 for all $a \in \A, h \in H,  r \in K$,  $x \in X_{h}$ and $y \in Y_{\psi(h)^{-1}r}$.
  \end{theorem}

\section{Tensor product actions in action}
In this section we illustrate how  tensor product actions naturally appear in various situations.

\subsection{$C^*$-correspondences of tensor product actions}
In \cite[Section 6.1]{BeCo8} we showed how actions of Fell bundles can be used to build natural $C^*$-correspondences.
Here we discuss the relationship between the $C^*$-correspondence associated to a tensor product action and the internal tensor product of the $C^*$-correspondences associated to the given actions.

\medskip Let $G, H, K$ be groups,  $\varphi \in {\rm Hom}(G,H)$ and $\psi \in {\rm Hom}(H, K)$.  
 Consider some Fell bundles  $\A=(A_g)_{g\in G}$, $\B=(B_h)_{h\in H}$, and $\C=(C_k)_{k\in K}$. 
Assume we are given an $(\A, \varphi)$-action $\rho$ on  a  Hilbert $\B$-bundle $\X=(X_h)_{h\in H}$ and  
 a $(\B,\psi)$-action $\sigma$ on a  Hilbert $\C$-bundle $\Y = (Y_k)_{k\in K}$.

We can then form the  $C^*(\A)$-$C^*(\B)$ correspondence $\X\rtimes\B$ associated with $\rho$ and the $C^*(\B)$-$C^*(\C)$ correspondence $\Y\rtimes\C$ associated to $\sigma$. In turn, we can consider the  $C^*(\A)$-$C^*(\C)$ correspondence
\[ (\X\rtimes\B)\otimes_{C^*(\B)} (\Y\rtimes\C).\]
On the other hand, we can also form the $(\A, \psi\circ\varphi)$-action $\rho\dot\boxtimes \sigma$ on the Hilbert $\C$-bundle $\X\boxtimes_\sigma\Y$, which gives rise to a $C^*(\A)$-$C^*(\C)$ correspondence $(\X\boxtimes_\sigma\Y)\rtimes{\C}$.   

\medskip We recall that the $C^*(\A)$-$C^*(\B)$-correspondence $\X\rtimes\B$ is the completion of the inner-product $C_c(\A)$-$C_c(\B)$-module $C_c(\X)$ consisting of all finitely supported sections $\xi$  from $H$ to $\X$, equipped with the operations given by
\begin{align*} 
(f\cdot \xi)(h) &= \sum_{g\in G} \rho(f(g)) \xi(\varphi(g)^{-1}h),\\
(\xi\cdot f')(h)  &:= \sum_{h'\in H} \xi(h') \cdot f'({h'}^{-1}h),\\
\langle \xi, \eta\rangle_{C_c(\B)}(h) &:= \sum_{h'\in H} \, \langle \xi(h'), \eta(h'h) \rangle_\B,
\end{align*}
for all  $f\in C_c(\A)$, $f'\in C_c(\B)$, $\xi, \eta \in C_c(\X)$ and $h\in H$.
To simplify the notation, we have here identified $C_c(\B)$  with its canonical copy in $C^*(\B)$ (via the map $\kappa^\B$). The norm of $f\in C_c(\B)$ in $C^*(\B)$  will be denoted by $\|f\|_{C^*(\B)}$. The $C^*(\B)$-$C^*(\C)$-correspondence $\Y\rtimes\C$ is defined similarly as the completion of $C_c(\Y)$, with operations given by
\begin{align*} 
(f'\cdot \xi')(k) &= \sum_{h\in H} \sigma(f'(h)) \xi'(\psi(h)^{-1}k),\\
(\xi'\cdot f'')(k)  &:= \sum_{k'\in K} \xi'(k') \cdot f''({k'}^{-1}k),\\
\langle \xi', \eta'\rangle_{C_c(\C)}(k) &:= \sum_{k'\in K} \, \langle \xi'(k'), \eta'(k'k) \rangle_\C,
\end{align*}
for all  $f'\in C_c(\B)$, $f''\in C_c(\C)$, $\xi', \eta' \in C_c(\Y)$ and $k\in K$.

We then get that $(\X\rtimes\B)\otimes_{C^*(\B)} (\Y\rtimes\C)$ can be identified with the completion of the inner product $C_c(\X) \otimes_{C_c(\A)} C_c(\C)$-module satisfying that 
\[\xi \dot\otimes (f'\cdot \xi')= (\xi\cdot f') \dot\otimes \xi',\] 
whose operations are determined by
\begin{align*} f\cdot (\xi \dot\otimes \xi') &= (f\cdot \xi) \dot\otimes \xi',\\
(\xi \dot\otimes \xi')\cdot f{''}&=\xi \dot\otimes (\xi'\cdot f''),\\
\big\langle \xi \dot\otimes \xi', \eta\dot\otimes \eta'\big\rangle_{C_c(\C)} &= \big\langle \xi',  \langle \xi, \eta\rangle_{C_c(\B)}\cdot\eta'\big\rangle_{C_c(\C)},
\end{align*} 
 for all $f \in C_c(\A), f'\in C_c(\B), f{''}\in C_c(\C), \xi, \eta \in C_c(\X)$ and $\xi', \eta'\in C_c(\Y)$.
 
 On the other hand, set $\mathcal{W} :=\X\boxtimes_\sigma\Y$. Recall that 
   \[W^0_k := \Big\{ \xi \in \prod_{h \in H} (X_h \otimes_{\sigma_{\psi(h)^{-1}k} } Y_{\psi(h)^{-1}k}) \mid \xi \text{ is finitely supported}\Big\}\]
   and that $\mathcal{W}$ is obtained from $(W^0_k)_{k\in K}$ after separation and completion.
   Then the $C^*(\A)$-$C^*(\C)$ correspondence $\mathcal{W}\rtimes{\C}$ is the completion of $C_c(\mathcal{W})$ with operations similar to those above. 
   
   Let now $\xi\in C_c(\X), \xi'\in C_c(\Y)$. Then for each $k\in K$ we can define $\zeta^{\xi, \xi'}(k)\in W^0_k$ by
 \[ [\zeta^{\xi, \xi'}(k)](h) = \xi(h) \dot\otimes \xi'(\psi(h)^{-1}k) \in X_h \otimes_{\sigma_{\psi(h)^{-1}k} } Y_{\psi(h)^{-1}k}\]
 for all $h\in H$.  Set $S= {\rm supp}(\xi) \subset H$, $S' = {\rm supp}(\xi') \subset K$, which both are finite. Note that if $k \not \in \bigcup_{h \in S} \psi(h)S'$, then $\zeta^{\xi, \xi'}(k) = 0$. Thus, $\zeta^{\xi, \xi'} \in C_c(\mathcal{W})$.  Moreover, 
 for $\xi, \eta \in C_c(\X)$ and $\xi', \eta'\in C_c(\Y)$ , we have  
 \[\big\langle \zeta^{\xi, \xi'}, \zeta^{\eta, \eta'}\big\rangle_{C_c(\C)} = \big\langle \xi\dot\otimes \xi', \eta\dot\otimes \eta'\big\rangle_{C_c(\C)}.\]
 Indeed, for every $k\in K$, we have
 
\begin{align*}
\big\langle \zeta^{\xi, \xi'}, \zeta^{\eta, \eta'}\big\rangle_{C_c(\C)}(k) &= \sum_{k'\in K} \big\langle \zeta^{\xi, \xi'}(k'), \zeta^{\eta, \eta'}(k'k)\big\rangle_\C\\
&= \sum_{k'\in K} \sum_{h, h'\in H} \big\langle [\zeta^{\xi, \xi'}(k')](h'), [\zeta^{\eta, \eta'}(k'k)](h)\big\rangle_\C\\
&= \sum_{k'\in K} \sum_{h, h'\in H} \big\langle \xi(h') \dot\otimes \xi'(\psi(h')^{-1}k'),\eta(h) \dot\otimes \eta'(\psi(h)^{-1}k'k)\big\rangle_\C\\
&= \sum_{k'\in K} \sum_{h, h'\in H} \big\langle \xi'(\psi(h')^{-1}k'), \sigma\big(\langle \xi(h'), \eta(h)\rangle_\B\big) \eta'(\psi(h)^{-1}k'k)\big\rangle_\C\\
&=  \sum_{h, h'\in H} \sum_{k'\in K} \big\langle \xi'(\psi(h')^{-1}k'), \sigma\big(\langle \xi(h'), \eta(h)\rangle_\B\big) \eta'(\psi(h)^{-1}k'k)\big\rangle_\C,
\end{align*}
 while 
 \begin{align*}
\big\langle \xi\dot\otimes \xi', \eta\dot\otimes \eta'\big\rangle_{C_c(\C)}(k)&=\big\langle \xi',  \langle \xi, \eta\rangle_{C_c(\B)}\cdot\eta'\big\rangle_{C_c(\C)}(k)
 \\
 &= \sum_{k'\in K} \big\langle \xi'(k'),  \big(\langle \xi, \eta\rangle_{C_c(\B)}\cdot\eta'\big)(k'k)\big\rangle_\C\\
  &= \sum_{k'\in K} \sum_{h\in H} \big\langle \xi'(k'),  \sigma\big(\langle \xi, \eta\rangle_{C_c(\B)}(h)\big) \eta'(\psi(h)^{-1}k'k)\big\rangle_\C\\
  &= \sum_{k'\in K} \sum_{h, h'\in H} \big\langle \xi'(k'),  \sigma\big(\langle \xi(h'), \eta(h'h) \rangle_\B\big) \eta'(\psi(h)^{-1}k'k)\big\rangle_\C\\
    &= \sum_{k'\in K} \sum_{ h', h''\in H} \big\langle \xi'(k'),  \sigma\big(\langle \xi(h'), \eta(h'') \rangle_\B\big) \eta'(\psi({h'}^{-1}h'')^{-1}k'k)\big\rangle_\C\\
   &=  \sum_{ h', h''\in H} \sum_{k'\in K} \big\langle \xi'(k'),  \sigma\big(\langle \xi(h'), \eta(h'') \rangle_\B\big) \eta'(\psi({h'}^{-1}h'')^{-1}k'k)\big\rangle_\C\\
&=  \sum_{ h', h''\in H} \sum_{k''\in K} \big\langle \xi'(\psi(h')^{-1}k''),  \sigma\big(\langle \xi(h'), \eta(h'') \rangle_\B\big) \eta'(\psi(h'')^{-1}k''k)\big\rangle_\C,
 \end{align*}
where we have replaced $h$ with ${h'}^{-1}h''$ at the fifth step, and $k'$ with $\psi(h')^{-1}k''$ at the last step, this being possible as all sums are finite.

Consider now $f''\in C_c(\C)$, $\xi \in C_c(\X)$ and $\xi' \in C_c(\Y)$. Then for every $k\in K$ and $h\in H$, we have 
\begin{align*}
[\zeta^{ \xi, \xi' \cdot f''}(k)](h) &= \xi(h) \dot\otimes (\xi' \cdot f'')(\psi(h)^{-1}k)\\ 
&=\xi(h)\dot\otimes  \Big(\sum_{k'\in K} \xi'(k') \cdot f'' (k'^{-1} \psi(h)^{-1}k) \Big)\\
&= \sum_{k'\in K} \xi(h)\dot\otimes \big(\xi'(\psi(h)^{-1}k') \cdot f'' (k'^{-1} k) \big)\\\
&= \sum_{k' \in K} \Big(\xi(h) \dot\otimes  \xi'(\psi(h)^{-1}k')\Big) \cdot f''(k'^{-1}k)  \\
&= \sum_{k' \in K} \big[\zeta^{\xi, \xi'}(k')\big](h) \cdot f''(k'^{-1}k)  \\
&= \Big[\sum_{k' \in K} \zeta^{\xi, \xi'}(k') \cdot f''(k'^{-1}k) \Big](h) \\
& = [(\zeta^{\xi, \xi'} \cdot f'')(k)](h),
\end{align*} 
where in the third line we put $k' = \psi(h)^{-1}k''$ and then replaced back $k''$ with $k'$. 

Thus we get
\begin{align*} 
\Xi( (\xi \dot\otimes \xi')\cdot f'') &= \Xi(\xi \dot\otimes (\xi'\cdot f''))= \zeta^{\xi, \xi'\cdot f''} =
\zeta^{\xi, \xi'}\cdot f''= \Xi( \xi \dot\otimes \xi')\cdot f'' \, .
\end{align*}  

 It follows now readily that the map $\xi\dot\otimes \xi'\mapsto \zeta^{\xi, \xi'}$ extends to a morphism $\Xi$ of  Hilbert $C^*$-modules from $(\X\rtimes\B)\otimes_{C^*(\B)} (\Y\rtimes\C)$ to $(\X\boxtimes_\sigma\Y)\rtimes{\C}$, meaning that $\Xi$  is a linear map preserving the inner-product and intertwining the right actions. 
 
\medskip 

The induced left action of $C^*(\A)$ on 
 $\mathcal{W}\rtimes \C= (\X \boxtimes_\sigma \Y)\rtimes \C$ is determined by
 \[ (f\cdot \zeta)(k) = \sum_{g\in G} [(\rho\boxtimes\sigma)(f(g))] \zeta(\psi(\varphi(g))^{-1}k)\]
 for all $f\in C_c(\A)$, $\zeta\in C_c(\mathcal{W})$ and $k\in K$.
  Let us now check that we get the desired intertwining property.
Let $f\in C_c(\A)$, $\xi \in C_c(\X)$ and $\xi' \in C_c(\Y)$. Then for every $k\in K$ and $h\in H$, we have
\begin{align*} 
[\zeta^{ f\cdot \xi, \xi'}(k)](h) &= (f\cdot \xi)(h) \dot\otimes \xi'(\psi(h)^{-1}k)\\ &= \sum_{g\in G} \rho(f(g)) \xi(\varphi(g)^{-1}h)\dot\otimes \xi'(\psi(h)^{-1}k)\\
&=  \sum_{g\in G} 
(\rho(f(g))\otimes \iota) \Big(\xi(\varphi(g)^{-1}h) \dot\otimes \xi'(\psi(h)^{-1}k)\Big)\\
&=  \sum_{g\in G} (\rho(f(g))\otimes \iota) \Big(\xi(\varphi(g)^{-1}h) \dot\otimes \xi'(\psi(\varphi(g)^{-1}h)^{-1}\psi(\varphi(g))^{-1}k)\Big)\\
&= \sum_{g\in G} (\rho(f(g))\otimes \iota) \Big([\zeta^{\xi, \xi'}(\psi(\varphi(g))^{-1}k)](\varphi(g)^{-1}h) \Big)\\
&= \sum_{g\in G} \Big([(\rho\boxtimes\sigma)(f(g))] \zeta^{\xi, \xi'}(\psi(\varphi(g))^{-1}k)\Big)(h) \\
& = [(f\cdot \zeta^{\xi, \xi'})(k)](h).
\end{align*} 
Thus we get
\begin{align*} 
\Xi(f\cdot (\xi \dot\otimes \xi')) &= \Xi((f\cdot \xi) \dot\otimes \xi') = \zeta^{ f\cdot \xi, \xi'} =
f\cdot \zeta^{\xi, \xi'} = f\cdot \Xi( \xi \dot\otimes \xi'),
\end{align*}
 as desired. It follows that $\Xi$ intertwines the left actions.
 Hence we have shown the following:
 
 \begin{theorem} \label{corr-morph} Let $G, H, K$ be groups, $\A=(A_g)_{g\in G}$, $\B=(B_h)_{h\in H}$, and $\C=(C_k)_{k\in K}$ be  Fell bundles,  $\varphi \in {\rm Hom}(G,H)$ and $\psi \in {\rm Hom}(H, K)$.  

Assume 
$\rho$ is an $(\A, \varphi)$-action  on  a  Hilbert $\B$-bundle $\X=(X_h)_{h\in H}$ and  
 $\sigma$ is  a $(\B,\psi)$-action on a  Hilbert $\C$-bundle $\Y = (Y_k)_{k\in K}$.
Let  $\X\rtimes\B$ be the $C^*(\A)$-$C^*(\B)$ correspondence associated with $\rho$,  $\Y\rtimes\C$ be the $C^*(\B)$-$C^*(\C)$ correspondence associated to $\sigma$, and $(\X\boxtimes_\sigma\Y)\rtimes{\C}$ be the $C^*(\A)$-$C^*(\C)$ correspondence associated to the $(\A, \psi\circ\varphi)$-action $\rho\dot\boxtimes \sigma$ on the Hilbert $\C$-bundle $\X\boxtimes_\sigma\Y$.

Then the map $\xi\dot\otimes \xi'\mapsto \zeta^{\xi, \xi'}$ from $C_c(\X) \otimes_{C_c(\B)} C_c(\Y)$ to $C_c(\X\boxtimes_\sigma\Y)$, determined for each $k\in K$ by 
 \[ [\zeta^{\xi, \xi'}(k)](h) = \xi(h) \dot\otimes \xi'(\psi(h)^{-1}k) \in X_h \otimes_{\sigma_{\psi(h)^{-1}k} } Y_{\psi(h)^{-1}k}\]
 for all $h\in H$, 
 extends to a morphism $\Xi$ of  $C^*$-correspondences from $(\X\rtimes\B)\otimes_{C^*(\B)} (\Y\rtimes\C)$ to $(\X\boxtimes_\sigma\Y)\rtimes{\C}$, meaning that $\Xi$  is a linear map preserving the inner-product and intertwining the left and the right actions.
  \end{theorem}
 It seems unlikely that the above map $\Xi$ is surjective in general.
  Although we will not discuss it here, we expect that a result similar to Theorem \ref{corr-morph}  holds in the reduced case as long as the $C^*$-correspondences associated with the reduced cross-sectional $C^*$-algebras exist (see \cite[Section 6.3]{BeCo8} for some sufficient conditions). 

\subsection{Induction of representations I} Assume  $\A=(A_g)_{g\in G}$ and $\B=(B_g)_{g\in G}$ are Fell bundles over $G$ and $\rho$ is an
$\A$-action on a  Hilbert $\B$-bundle $\X=(X_g)_{g\in G}$.  Moreover, assume $\pi = (\pi_g)_{g\in G}$ is a
$*$-representation of $\B $ on some  Hilbert $C$-module $Y$, where $C$ is a $C^*$-algebra. Then, letting $\C= (C \times \{g\})_{g\in G}$  be the natural Fell bundle over $G$ associated to $C$,  we may form the $\B$-action $\rho_\pi$ associated to $\pi$ on the Hilbert $\C$-bundle $\Y = (Y\times \{g\})_{g\in G}$, given by
\[\rho_\pi(b)(y, h) = \big(\pi_g(b)y, gh\big) \] 
for all $g, h\in G$, $b \in B_g$ and $y\in Y$ (cf.~\cite[Example 4.15]{BeCo8}). 

We may now form the tensor product of $\rho$ and $\rho_\pi$ and obtain an $\A$-action $\rho\dot\boxtimes\rho_\pi$ on the Hilbert $\C$-bundle $\X\boxtimes_{\rho_{\pi}} \Y$. 
Such an action $\rho\dot\boxtimes\rho_\pi$ may be seen as a kind of analog of 
the tensor product of an equivariant representation with a covariant representation
 of a unital $C^*$-dynamical system introduced in  \cite[Lemma 4.9]{BeCo3}.

The action $\rho\dot\boxtimes\rho_\pi$ is in fact related to an $\A$-action arising from a $*$-representation of $\A$ on a certain Hilbert $C^*$-module. To be more precise, let $\X \rtimes \B$ be the $C^*(\A)$-$C^*(\B)$-correspondence associated to $\rho$, cf.~\cite[Section 6.1]{BeCo8}, and denote the left action of $C^*(\A)$ on $\X \rtimes \B$ 
by $\Phi_\rho$.  Moreover, let $\widetilde\pi: C^*(\B)\to \mathcal{L}_C(Y)$ be the $*$-representation associated to $\pi$. Then 
we get an induced $*$-representation
$\widetilde{{\rm Ind}_\rho(\pi)}: C^*(\A) \to \mathcal{L}_C((\X \rtimes \B)\otimes_{\widetilde\pi} Y)$, determined by
\[ (\widetilde{{\rm Ind}_\rho(\pi)}f)(\xi \dot\otimes y) = (\Phi_\rho(f)\xi) \dot\otimes y\]
for all $f\in C_c(\A)$, $\xi\in C_c(\X)$ and $y\in Y$.  
The  disintegrated form of $\widetilde{{\rm Ind}_\rho(\pi)}$ is a $*$-representation  ${\rm Ind}_\rho(\pi)$ of $\A$ on the Hilbert $C$-module $(\X \rtimes \B)\otimes_{\widetilde\pi} Y$, satisfying 
\[ [\big({\rm Ind}_\rho(\pi)\big)(a)](\xi \dot\otimes y) =(\Phi_\rho(a \odot g)\xi) \dot\otimes y\]
for all $g\in G$, $a\in A_g$, $\xi\in C_c(\X)$ and $y\in Y$. 
 In turn, ${\rm Ind}_\rho(\pi)$ induces an $\A$-action
$\rho_{{\rm Ind}_\rho(\pi)}$ on the Hilbert $\C$-bundle $\big(((\X \rtimes \B)\otimes_{\widetilde\pi} Y) \times \{g\}\big)_{g \in G}$, satisfying
\[\rho_{{\rm Ind}_\rho(\pi)}(a)(\Xi \dot\otimes y, h) 
= ((\Phi_\rho(a \odot g)\Xi) \dot\otimes y, gh)\]
for all $g,h \in G$, $a \in A_g$, 
$\Xi \in \X \rtimes \B$ and $y \in Y$. 

\begin{theorem} \label{Ind-prop} The $\A$-action $\rho_{{\rm Ind}_\rho(\pi)}$  on the Hilbert $\C$-bundle $\big(((\X \rtimes \B)\otimes_{\widetilde\pi} Y) \times \{g\}\big)_{g \in G}$ is  unitarily equivalent to the $\A$-action $\rho\dot\boxtimes\rho_\pi$ on $\X\boxtimes_{\rho_{\pi}} \Y$.
\end{theorem}

\begin{proof} We only sketch the proof, as adding more details do not bring any new insight.  Let $r\in G$ and set
\[W^0_r := \Big\{ \xi \in \prod_{h \in H} (X_h \otimes_{({\rho_\pi})_{h^{-1}r} } (Y \times \{h^{-1}r\}) \mid \xi \text{ is finitely supported}\Big\}.\] 
Then we have
 \[ \big((\rho\boxtimes \rho_\pi)(a)\xi\big)(h) := (\rho(a)\otimes \iota)\xi(g^{-1}h), \  h \in H \]
 for all $\xi \in W^0_r, g\in G, a \in A_g$. Also, $\X\boxtimes_{\rho_{\pi}} \Y = (W_r)_{r\in G}$, where each $W_r$ is obtained as the Hausdorff completion of $W^0_r$.
   
   Let $U^0_r: (C_c(\X)\otimes_{\widetilde\pi} Y) \times \{r\} \to W^0_r$ be the linear map satisfying that
$$[U^0_r(\eta \dot\otimes y, r)](h) = \eta(h) \dot\otimes (y,h^{-1}r)$$
for all $\eta\in C_c(\X)$, $y\in Y$ and $h\in H$.  Then it can be checked that $U_r^0$ induces a map  $U_r$ from $((\X\rtimes \B)\otimes_{\widetilde\pi} Y) \times \{r\}$ to $W_r$, such that $U:=(U_r)_{r\in G}$ is a unitary between $\big(((\X \rtimes \B)\otimes_{\widetilde\pi} Y) \times \{r\}\big)_{r \in G}$ and $\X\boxtimes_{\rho_{\pi}} \Y$
 intertwining $\rho_{{\rm Ind}_\rho(\pi)}$ and $\rho\dot\boxtimes\rho_\pi$.
In order to see this, the following steps suffice up to some routine arguments:

For all $\eta, \eta' \in C_c(\X)$, $y, y' \in Y$, $r,r' \in G$, we have
\begin{align*} \langle U_r^0(\eta \dot\otimes y, r),  U_{r'}^0(\eta'\dot\otimes y', r')\rangle_\C &=  
\sum_{h, h'\in G}  \langle \eta(h) \dot\otimes (y,h^{-1}r),  \eta'(h') \dot\otimes (y',h'^{-1}r') \big\rangle_\C \\
&=  \sum_{h, h'\in G} \langle (y,h^{-1}r), \rho_\pi(\langle\eta(h),\eta'(h')\rangle_\B) (y',h'^{-1}r') \big\rangle_\C \\
&= \sum_{h, h'\in G} \langle (y,h^{-1}r), \pi_{h^{-1}h'}(\langle\eta(h),\eta'(h')\rangle_\B)y',h^{-1}r') \big\rangle_\C \\
&= \sum_{h, h'\in G} \big( \langle y, \pi_{h^{-1}h'}(\langle\eta(h),\eta'(h')\rangle_\B)y'\rangle_C, r^{-1}r'\big)\\
&= \sum_{h, h'\in G} \big( \langle y, \pi_{h}(\langle\eta(h'),\eta'(h'h)\rangle_\B)y'\rangle_C, r^{-1}r'\big)\\
&= \sum_{h\in G}\Big(\big\langle y, \big[\pi_h(\langle\eta, \eta'\rangle_{C_c(\B)}(h))\big]y' \big\rangle_C, r^{-1}r'\Big)\\
&= \Big(\big\langle y, \tilde\pi(\langle\eta, \eta'\rangle_{C_c(\B})y' \big\rangle_C, r^{-1}r'\Big)\\
&= \langle (\eta \dot\otimes y, r),  (\eta'\dot\otimes y', r')\rangle_\C. 
\end{align*}
Moreover, for all $\eta \in C_c(\X)$, $y \in Y$, $r,g, h \in G$,  $c\in C$ and $a \in A_g$, we have
\begin{align*} 
\big[\big(U_r^0(\eta \dot\otimes y, r)\big) \cdot (c,g)\big](h) & = \big(\big[U_r^0(\eta \dot\otimes y, r)\big](h)\big)\cdot (c,g)
= \big( \eta(h) \dot\otimes (y,h^{-1}r) \big)\cdot (c,g)\\ 
& = \eta(h) \dot\otimes (yc,h^{-1}rg) 
 = \big[U_{rg}^0 \big( \eta \dot\otimes yc, rg \big)\big](h) \\
& = \big[ U_{rg}^0 \big( (\eta \dot\otimes y, r)\cdot (c,g) \big)\big](h),
\end{align*}
and
\begin{align*}
\Big[(\rho\dot\boxtimes\rho_\pi(a)) U_r^0(\eta \dot\otimes y, r)\Big](h) 
& = (\rho(a) \otimes \iota) [U_r^0(\eta \dot\otimes y, r)](g^{-1}h) \\
& = (\rho(a) \otimes \iota) \big( \eta(g^{-1}h) \dot\otimes (y, h^{-1}gr)\big)\\
& =   \big( \rho(a)\eta(g^{-1}h) \big) \dot\otimes (y, h^{-1}gr) \\
&  = (\Phi_\rho(a \odot g)\eta)(h) \dot\otimes (y, h^{-1}gr) \\
& =U_{gr}^0\big( \Phi_\rho(a \odot g)\eta \dot\otimes y, gr\big)(h) \\
& = U_{gr}^0\big(\rho_{{\rm Ind}_\rho(\pi)}(a)(\eta \dot\otimes y, r)\big)(h).
\end{align*}
The surjectivity of $U_r$ follows from the fact that $U^0_r$ is surjective, which is almost immediate as $U^0_r$ is linear and 
$W_r^0$ is spanned by $\{x\dot\otimes (y, k^{-1}r):  k\in G, x \in X_k, y\in Y\}$.
\end{proof}

\subsection{Induction of representations II}
Assume  $\A=(A_g)_{g\in G}$ is a Fell bundle over $G$,  $\pi = (\pi_g)_{g\in G}$ is a
$*$-representation of $\A $ on a  Hilbert $B$-module $X$,
 and $\pi'$ is a $*$-representation  of $B$ on a Hilbert $C$-module $Y$, 
 where $B$ and $C$ are $C^*$-algebras.  
 We will also denote by $\pi'$ the associated $*$-representation $(\pi'_g)_{g\in G}$  of 
   the trivial Fell bundle $\B = (B \times {g})_{g \in G}$ 
   on $Y$, given by $\pi'_g((b, g))y = \pi'(b)y$ for all $g\in G$, $b\in B$ and $y\in Y$.  
  
  Then, as in the previous subsection, we can form 
the $\A$-action $\rho_\pi$ associated to $\pi$ on the canonical Hilbert $\B$-bundle $\X := (X\times \{g\})_{g\in G}$.
Further, setting $\C:= (C\times \{g\})_{g\in G}$, we can also form  
 the $\B$-action $\rho_{\pi'}$ on the canonical Hilbert $\C$-bundle  $\Y:= (Y\times \{g\})_{g\in G}$, which is given by 
 \[  [\rho_{\pi'}(b, g)](y, h) =(\pi'_g(b,g)y, gh) =  (\pi'(b)y, gh)\]
for all $b\in B$, $g, h\in G$ and $y\in Y$.

Finally, we consider the induced $*$-representation 
$\pi''=(\pi''_g)_{g\in G}$ of $\A$ on the Hilbert $C$-module $X\otimes_{\pi'} Y$, determined by
\[ \pi''_g(a)(x\dot\otimes y) = (\pi_g(a)x) \dot\otimes y\]
for all $g \in G$, $a\in A_g, x \in X $ and $y\in Y$. Then we have:

\begin{proposition}\label{rhopitensrhopi'}
The $\A$-action $\rho_\pi \dot\boxtimes\rho_{\pi'}$ on the Hilbert $\C$-bundle $\X\boxtimes_{\rho_{\pi'}}\Y$ is unitary equivalent to the $\A$-action $\rho_{\pi''}$ associated to $\pi''$ on the Hilbert $\C$-bundle $((X\otimes_{\pi'} Y) \times \{g\})_{g\in G}$.
\end{proposition}
\begin{proof} Using Theorem \ref{Ind-prop} we get that $\rho_\pi \dot\boxtimes\rho_{\pi'}$ is unitary equivalent to the $\A$-action $\rho_{{\rm Ind}_{\rho_{\pi}}(\pi')}$, which acts  on the Hilbert $\C$-bundle $\big(((\X \rtimes \B)\otimes_{\widetilde{\pi'}} Y) \times \{g\}\big)_{g \in G}$. So it suffices to show that $\rho_{\pi''}$ is unitarily equivalent to $\rho_{{\rm Ind}_{\rho_{\pi}}(\pi')}$, i.e., that the $*$-representation $\pi''$ of $\A$ on the Hilbert $C$-module $X\otimes_{\pi'} Y$ is unitarily equivalent to  the $*$-representation ${\rm Ind}_{\rho_{\pi}}(\pi')$ of $\A$ on the Hilbert $C$-module $(\X \rtimes \B)\otimes_{\widetilde{\pi'}} Y$. 

Let us identify $C^*(\B)$ with $B\otimes_{\rm max} C^*(G)$, 
(cf. \cite[Example 4.15]{BeCo8}). Then one checks that $\X \rtimes \B$ can be identified with the $B\otimes_{\rm max} C^*(G)$-module $X\otimes_{\rm max} C^*(G)$ obtained by completing the natural $ B\odot C_c(G)$-inner product module $X\odot C_c(G)$, whose right action and inner-product are determined by 
\[ (x\otimes h)\cdot (b \otimes h') = xb \otimes hh', \]  
\[\langle x\otimes h, x'\otimes h'\rangle_{B\odot C_c(G)} = \langle x, x'\rangle_B \otimes h^{-1}h'\]
for all $x,x'\in X$, $b \in G$ and $h,h'\in G$, 
where we consider $G$ as canonically embedded in $C^*(G)$.

Moreover, $\widetilde{\pi'}$ corresponds then to the $*$-representation of $B\otimes_{\rm max} C^*(G)$ on $Y$ determined by 
\[\widetilde{\pi'}(b\otimes h)y = \pi'(b)y \]
for all $b\in B$ and $ h\in G$. 
Now one can readily verify that in $(X\otimes_{\rm max} C^*(G))\otimes_{\widetilde{\pi'}} Y$, we have $(x\otimes h) \dot\otimes y = (x\otimes e) \dot\otimes y$ for all $x \in X, h\in G$ and $y\in Y$. Hence $(X\otimes_{\rm max} C^*(G))\otimes_{\widetilde{\pi'}} Y$ can be identified with  $X\otimes_{\pi'} Y$ via the map $V: (X\otimes_{\rm max} C^*(G))\otimes_{\widetilde{\pi'}} Y \to X\otimes_{\pi'} Y$ determined by 
\[ V\big((x\otimes e) \dot\otimes y\big) = x\dot\otimes y\] for all $x \in X$ and $y\in Y$. 

Let $g\in G, a\in A_g$. Under the identification of 
$(X\rtimes \B)\dot\otimes_{\widetilde{\pi'}} Y$ 
with $(X\otimes_{\rm max} C^*(G))\otimes_{\widetilde{\pi'}} Y$, we get that 
\[ [({\rm Ind}_{\rho_{\pi}}(\pi'))(a)]((x\otimes h)\dot\otimes y) = ((\pi_g(a)x) \otimes gh)\dot\otimes y\]
for all $x\in X, h\in G$ and $y\in Y$, hence that
\[ [({\rm Ind}_{\rho_{\pi}}(\pi'))(a)]((x\otimes e)\dot\otimes y) = ((\pi_g(a)x) \otimes g)\dot\otimes y = ((\pi_g(a)x) \otimes e)\dot\otimes y\]
Thus, 
\begin{align*} 
(\pi''(a)V)((x\otimes e) \dot\otimes y) &=  \pi''(a)(x\dot\otimes y) = (\pi_g(a)x) \dot\otimes y
= V\big(((\pi_g(a)x)\otimes e) \dot\otimes y\big) \\
&=\big(V[{\rm Ind}_{\rho_{\pi}}(\pi')](a)\big)((x\otimes e) \dot\otimes y)
\end{align*}
Hence, $V$ is a unitary intertwining $\pi''$ and ${\rm Ind}_{\rho_{\pi}}(\pi')$, as desired.
\end{proof}

\subsection{Induction of actions via conditional expectations} \label{RieffelInd}
The tensor product construction may also be used to produce a Rieffel-like induction of actions of Fell bundles via conditional expectations.  For simplicity, we assume $H=G$ and $\varphi={\rm id}_G$.  

Let us first recall that if $B$ is a $C^*$-subalgebra of a $C^*$-algebra $A$ such that there exists a conditional expectation $E$ from $A$ to $B$, then, following Rieffel \cite{Rie74},  one can construct  an $A$-$B$ correspondence $A_E$ naturally arising from the right $B$-module $A$ equipped with the semi-inner product  $\langle a,a'\rangle_B = E(a^*a')$.
This $C^*$-correspondence can then be used to induce a $*$-representation ${\rm Ind}_E(\pi)$ of $A$ on $A_E\otimes_\pi Y$ from a $*$-representation $\pi$ of $B$ on a Hilbert $C^*$-module $Y$.

Let us consider a Fell bundle $\A$ over $G$ having  a sub-Fell bundle $\B$ such that there exists a conditional expectation $\E : \A \to \B$, cf.~\cite[Definition 21.19]{Exel2}. Let $\rho_\E$ be the $\A$-action on the Hilbert $\B$-bundle $\A_\E$ 
(cf.~\cite[Example 4.18]{BeCo8}) and let $\A_\E\rtimes\B$ be the $C^*(\A)$-$C^*(\B)$ correspondence associated with $\rho_\E$. We note that if $F_\E: C^*(\A) \to C^*(\B)$ is the conditional expectation associated to $\E$ (cf.~\cite[Theorem  21.30]{Exel2} and \cite[Example 4.7]{BeCo7}), then it is not difficult to verify that the $C^*(\A)$-$C^*(\B)$ correspondence $C^*(\A)_{F_\E} $ is isomorphic to $\A_\E\rtimes\B$.  
 
Assume further that $\psi \in {\rm Hom}(G, K)$ for some group $K$ and $\sigma$ is a $(\B, \psi)$-action on a Hilbert $\C$-bundle $\Y$ for some Fell bundle $\C$ over $K$. 
\begin{definition}
The \emph{$(\A, \psi)$-action ${\rm Ind}_\E(\sigma)$ on the Hilbert $\C$-bundle $\A_\E \boxtimes_\sigma \Y$ induced by $\sigma$ via $\E$} 
is given by \[{\rm Ind}_\E(\sigma):= \rho_\E \dot\boxtimes \sigma.\] 
\end{definition}
If $K=G$ and $\psi = {\rm id}_G$, we just call ${\rm Ind}_\E(\sigma)$ the $\A$-action induced by $\sigma$ via $\E$.
This notion is a generalization of Rieffel's induction \cite{Rie74}. Indeed,  when $G$ and $K$ are trivial, we may identify $\A$, $\B$ and $\C$ with $C^*$-algebras $A$, $B$, $C$, respectively. Then $\E$ is a conditional expectation $E$ from $A$ to $B$, $\A_\E$ is the natural Hilbert $B$-module $A_E$ obtained from the $B$-module $A$ equipped with the semi-inner product  $\langle a,a'\rangle_B = E(a^*a')$, and $\rho_\E(a)$ arises from left multiplication with $a$. Moreover, $\sigma$ is a $*$-representation $\pi$ of $B$ on a Hilbert $C$-module $Y$, cf.~Example \ref{G-trivial}. It is then clear that  
${\rm Ind}_\E(\sigma)$ is exactly Rieffel's induced representation ${\rm Ind}_E(\pi)$ of $A$ on the Hilbert $C$-module $A_E\otimes_\pi Y$.

Going back to the general situation, let $\Y\rtimes\C$ be the $C^*(\B)$-$C^*(\C)$ correspondence associated to $\sigma$ and  $\Phi_\sigma$ denote the left action of $C^*(\B)$ on $\Y \rtimes \C$. Further, let $(\A_\E\boxtimes_\sigma\Y)\rtimes{\C}$ be the $C^*(\A)$-$C^*(\C)$ correspondence associated to the 
 $(\A,\psi)$-action ${\rm Ind}_\E(\sigma)$,  and $\Phi_{{\rm Ind}_\E(\sigma)}$ denote  the left action of $C^*(\A)$ on  $(\A_\E\boxtimes_\sigma\Y)\rtimes{\C}$.
Then Theorem \ref{corr-morph}  gives 
 that there is a morphism \[\Xi:(\A_\E\rtimes\B)\otimes_{C^*(\B)} (\Y\rtimes\C) \to (\A_\E\boxtimes_\sigma\Y)\rtimes{\C}\] of  $C^*$-correspondences.
 Now, identifying $(\A_\E\rtimes\B)\otimes_{C^*(\B)} (\Y\rtimes\C)$ with  
 $C^*(\A)_{F_\E} \otimes_{\Phi_\sigma} (\Y \rtimes \C)$ (using our observation above), the left action of $C^*(\A)$ on  $(\A_\E\rtimes\B)\otimes_{C^*(\B)} (\Y\rtimes\C)$ is given by  ${\rm Ind}_{F_\E}(\Phi_\sigma)$.Thus we conclude that 
$\Xi$ intertwines  ${\rm Ind}_{F_\E}(\Phi_\sigma)$ with $\Phi_{{\rm Ind}_\E(\sigma)}$. 

The above procedure can be used to induce actions from subgroups.  Assume $\A$ is a Fell bundle over $G$, $H$ is a subgroup of $G$, and $\iota_H:H\to G$ denotes the canonical embedding of $H$ in $G$. Let $\A_H$ denote the Fell bundle over $H$ obtained from $\A$ by restriction, and 
 $\B$ denote the Fell bundle over $G$ obtained from $\A_H$ by trivial extension, i.e., $B_g= A_g$ if $g \in H$ and
$B_g= 0$ otherwise.
  Let then $\E=(E_g)_{g\in G}$ be the conditional expectation from $\A$ to $\B$ given by $E_g= {\rm id}_{A_g}$ if $g\in H$ and $E_g= 0$ otherwise. 

Assume now $\tau$ is an $(\A_H, \iota_H)$-action on a Hilbert $\C$-bundle $\Y$ for some Fell bundle $\C$ over $G$.
Denote by  $\widetilde\tau$  the trivial extension of $\tau$ to a $\B$-action on $\Y$, given for all $y\in \Y$ by
\[ \widetilde\tau(b)y=\begin{cases} \tau(b)y \quad & \text{ if $g\in H$ and  $b\in B_g$},\\ \, \, 0 \quad &\text{ if $g\in G\setminus H$ and  $b\in B_g$}.\end{cases}\]
  \begin{definition}
  The  \emph{$\A$-action induced by $\tau$ from $H$ to $G$} is 
 the $\A$-action on  the Hilbert $\C$-bundle $\A_\E \boxtimes_{\widetilde\tau} \Y$ given by \[{\rm Ind}_H^G(\tau):={\rm Ind}_\E(\widetilde\tau).\] 
 \end{definition}

As a concrete example, let  us consider  the group bundle $\A_G$ over $G$, so $(\A_G)_H=\A_H$ is the group bundle over $H$,  and let $\pi$ be a unitary representation of $H$ on a Hilbert space $Y$.

Let then $\Y=(Y_g)_{g\in G}$ be the trivial Hilbert $\A_G$-bundle over $G$ associated to $Y$, that is,  $Y_g:=Y\times \{g\}$ for all $g\in G$, $(y, g)\cdot (z, g')= (zy, gg')$ and $\langle (y, g), (y', g')\rangle_{\A_G} = (\langle y, y'\rangle, g^{-1}g')$ for all $y, y'\in Y$, $z\in \mathbb{C}$ and $g, g'\in G$.
Then we get an $(\A_H, \iota_H)$-action $\tau_\pi$ on the Hilbert $\A_G$-bundle $\Y$, defined by
\[ [\tau_\pi(z, h)](y, g) = (z\pi_h(y), hg)\]
for all  $(z, h) \in \mathbb{C}\times H$ and $(y, g) \in Y\times G$.  
Letting  $\B$ be  the Fell bundle over $G$ obtained from $\A_H$ by trivial extension as above,  and $\widetilde{\tau_\pi}$  be the trivial extension of $\tau_\pi$ to a $\B$-action on $\Y$,   we can form the $\A_G$-action ${\rm Ind}_H^G({\tau_\pi})$  induced by $\tau_\pi$ from $H$ to $G$, given by  \[{\rm Ind}_H^G({\tau_\pi})= {\rm Ind}_\E(\widetilde{\tau_\pi}) = \rho_\E \dot\boxtimes \widetilde{\tau_\pi}.\]
Here $\E$ is the obvious conditional expectation from $\A_G$ to $\B$, and $\rho_\E$ is the associated $\A_G$-action on the Hilbert $\B$-bundle $(\A_G)_\E$ (cf.~\cite[Examples 3.7 and 4.18]{BeCo8}). 
In the present situation, $(\A_G)_\E$ is the natural $\B$-bundle $ (\mathbb{C}\times\{g\})_{g\in G}$, with inner product given by 
\[ \langle (z, g), (z', g')\rangle_\B = E_{g^{-1}g'} \big(  (z, g)^* (z', g')\big) = E_{g^{-1}g'} (\bar z z', g^{-1} g')= \begin{cases} (\bar z z', g^{-1} g') & \text{ if }  g^{-1}g' \in H,\\ 0 &\text{ otherwise}\end{cases}\] 
for all $(z, g), (z', g') \in (\A_G)_\E$. 
(Since the norm of an element $(z, g)\in (\A_G)_\E$ is given by $\|(z,g)\| =  \|\langle (z, g), (z, g)\rangle_\B\|^{1/2} = \| (|z|^2, e)\|^{1/2} = |z|$, no completion is needed here.) Moreover, if $(z, g) \in \A_G$ and $(z', g')\in (\A_G)_\E$, then $[\rho_\E(z,g)](z',g') = (zz', gg')$.

\smallskip The $\A_G$-action ${\rm Ind}_H^G({\tau_\pi})= \rho_\E \dot\boxtimes \widetilde{\tau_\pi}$ acts on the Hilbert $\A_G$-bundle $(\A_G)_\E \boxtimes_{\widetilde{\tau_\pi}} \Y = (W_r)_{r\in G}$, where $W_r$ is the closure of the span of  $\{(z', k)\dot\otimes (y, k^{-1}r) :  z'\in \mathbb{C}, k\in G, y\in Y\}$ for each $r\in G$, by
\[(\rho_\E \dot\boxtimes \widetilde{\tau_\pi})(z, g)((z', k)\dot\otimes (y, k^{-1}r)) = \big((\rho_\E(z,g))(z', k)\big) \dot\otimes (y, k^{-1}r) =  (zz', gk)\dot\otimes (y, k^{-1}r)  \]
for all  $z, z'\in \mathbb{C}, g, k, r\in G$ and  $y\in Y$. Then we have:

\begin{proposition} Let ${\rm Ind}_H^G(\pi)$ denote the unitary representation induced by $\pi$ from $H$ to $G$. Then the $\A_G$-action ${\rm Ind}_H^G({\tau_\pi})$ is unitary equivalent to the  $\A_G$-action associated with ${\rm Ind}_H^G(\pi)$.
\end{proposition}

\begin{proof} Let us first describe the Hilbert space $\widetilde Y$ on which ${\rm Ind}_H^G(\pi)$ acts. We follow here Raeburn and Williams \cite{RW}. Let $\tilde\pi$ denote the integrated form of $\pi$, so $\tilde\pi$ is a representation of $C^*(H)$ on $Y$. Then $\widetilde Y := X \otimes_{\tilde\pi} Y$, where $X$ is the Hilbert $C^*(H)$-module obtained as the Hausdorff completion of the $C_c(H)$-module $C_c(G)$, the  right action of $C_c(H)$ and the semi-inner product on $C_c(G)$ with values in $C_c(H)$ being given by 
\[ (f\cdot b)(g) = \sum_{h\in H} f(gh^{-1})b(h), \quad \langle f, f'\rangle_{C_c(H)}(s) = \sum_{k\in G} \overline{f(k)}f'(ks)\]
 for all $f, f'\in C_c(G), b\in C_c(H)$, $g\in G$ and $s\in H$. Setting $N:=\{ f\in C_c(G): \langle f, f\rangle_{C_c(H)}=0\}$, the quotient space $C_c(G)/N$ is then a dense subspace of $X$, by construction.
 
 Now $G$ acts on $C_c(G)$ by left translation, say $g\cdot f=  f_g$ for $g\in G$ and $f\in C_c(G)$. It can be shown, cf.~\cite[Section 2.4 and Appendix C]{RW}, that one gets  a unitary representation $u$ of $G$ on $X$, determined by $u_g(f+N) =f_g + N$ for all $g\in G$ and $ f\in C_c(G)$. Then ${\rm Ind}_H^G(\pi)$  is  the unitary representation of $G$ on $\widetilde Y=X \otimes_{\tilde\pi} Y$ determined by
 \[ [{\rm Ind}_H^G(\pi)(g)]((f+N)\dot\otimes y) = u_g(f+N) \dot \otimes y = (f_g + N) \dot\otimes y\]
 for all $g\in G$, $f \in C_c(G)$ and $y \in Y$.  
 
 The $\A_G$-action $\sigma$ associated with   ${\rm Ind}_H^G(\pi)$ acts on the trivial Hilbert $\A_G$-bundle $(\widetilde Y \times \{r\})_{r\in G}$ associated to $\widetilde Y$ and is determined by 
 \[ (\sigma(z,g)) ((f+N)\dot\otimes y, r) =  (z({\rm Ind}_H^G(\pi)(g))(f+N)\dot\otimes y , gr) = (( zf_g + N)\dot\otimes y, gr)\]
 for all $z\in \mathbb{C}, g, r\in G, f\in C_c(G)$ and $y\in Y$, hence by
 \[  (\sigma(z,g)) ((z'\odot k+N)\dot\otimes y, r) = (((zz'\odot gk) + N)\dot\otimes y), gr)\]
  for all $z, z'\in \mathbb{C}, g, k, r\in G$,  and $y\in Y$. 
  
  We want to prove that the $\A_G$-actions ${\rm Ind}_H^G(\tau_\pi)$ and $\sigma$ are unitarily equivalent.\\
Let  $V=(V_r)_{r\in G}$ be the unitary from  $(\widetilde Y \times \{r\})_{r\in G}$ to  $(W_r)_{r\in G}$ determined by
  \[ V_r ((f+N)\dot\otimes y, r) =  \sum_{k\in G} (f(k), k)\dot\otimes (y, k^{-1} r). \]
  
  Note here that for all $r, r'\in G$, $f, f'\in C_c(G)$ and $y, y'\in Y$, we have
  \begin{align*}
  \big\langle V_r ((f+N)\dot\otimes y, r) , & V_{r'} ((f'+N)\dot\otimes y', r') \big\rangle_{\A_G} = \sum_{k,k'\in G} \langle (f(k), k)\dot\otimes (y, k^{-1} r) , (f'(k'), k')\dot\otimes (y', k'^{-1} r') \rangle_{\A_G}\\
  & = \sum_{k,k'\in G} \langle (y, k^{-1} r), \widetilde{\tau_\pi} \big( \langle (f(k), k), (f'(k'), k')\rangle_\B \big) (y', k'^{-1} r') \rangle_{\A_G} \\
  & =  \sum_{k,k'\in G, \,k^{-1}k' \in H}  \langle (y, k^{-1} r), (\overline{f(k)}f'(k')\pi_{k^{-1}k'}y, k^{-1}r') \rangle_{\A_G} \\
  & =  \sum_{k\in G, \,s \in H}  \langle (y, k^{-1} r), (\overline{f(k)}f'(ks)\pi_{s}y', k^{-1}r') \rangle_{\A_G}\\
  & = \Big( \big\langle y, \sum_{s \in H} \sum_{k \in G} \overline{f(k)} f'(ks) \pi_s y' \big\rangle, r^{-1}r'\Big) \\
  & = \Big( \big\langle y, \tilde{\pi}(\langle f, f' \rangle_{C_c(H)})y' \big\rangle, r^{-1}r'\Big) \\
 & = \big\langle ((f+N)\dot\otimes y, r) , ((f'+N)\dot\otimes y', r') \big\rangle_{\A_G}.
  \end{align*}
  It follows that each $V_r$ is well-defined and isometric on 
  ${\rm span} \{((f+N)\dot\otimes y, r): f\in C_c(G), y\in Y\}$, which is a dense subspace of $\widetilde Y \times \{r\}$.
   Thus each $V_r$ extends to an isometry, also denoted by $V_r$, from $\widetilde Y \times \{r\}$ to $W_r$. As
each $V_r$ is easily seen to be surjective, we get that $V=(V_r)_{r\in G}$ is unitary and satisfies
\[  \big\langle V_r (\tilde y, r) ,  V_{r'} (\tilde{y}', r') \big\rangle_{\A_G} = \big\langle (\tilde y, r) , (\tilde{y}', r') \big\rangle_{\A_G}\]
for all $ r, r'\in G$ and $\tilde y, \tilde y' \in \widetilde Y$.
   Moreover, we have
  \begin{align*}
  V_{rg}\big( ((f+N) \otimes y, r) & \cdot (z,g)\big)  = V_{rg} \big(((f+N) \otimes zy, rg) \big) = \sum_{k\in G} (f(k), k)\dot\otimes (zy, k^{-1} rg) \\
  =  & \sum_{k\in G} (f(k), k)\dot\otimes \big( (y, k^{-1} r) \cdot (z,g) \big) 
  = V_r\big(((f+N) \otimes y, r) \big) \cdot (z,g),  
  \end{align*}
 
 \vspace{-2ex}
  \noindent and
   \begin{align*}
(\rho_\E \dot\boxtimes \widetilde{\tau_\pi})(z, g)  \big(V_r((f+N)\dot\otimes y, r)\big)
&= \sum_{k\in G} (\rho_\E \dot\boxtimes \widetilde{\tau_\pi})(z, g)(f(k), k)\dot\otimes (y, k^{-1} r)\\
&=  \sum_{k\in G} (zf(k), gk)\dot\otimes (y, k^{-1} r)
= \sum_{l\in G} ( zf(g^{-1}l), l)\dot\otimes (y, l^{-1}gr)\\&
= \sum_{l\in G} ( zf_g(l), l)\dot\otimes (y, l^{-1}gr)
= V_{gr} \big((( zf_g + N)\dot\otimes y, gr)\big)\\&
= V_{gr} \big( \sigma(z,g)((f+N)\dot\otimes y, r)\big)  
  \end{align*}
  for all $r, g \in G, f\in C_c(G)$, $y\in Y$ and $z\in \mathbb{C}$. This implies that 
$ V_{rg}\big((\tilde y, r)\cdot (z, g)\big)= \big(V_r(\tilde y, r)\big)\cdot (z, g)$ and 
$(\rho_\E \dot\boxtimes \widetilde{\tau_\pi})(z, g) \big(V_r(\tilde y, r)\big)  = V_{gr}\big(\sigma(z, g)(\tilde y, r)\big)$
for all $r,g \in G$, $\tilde y \in \widetilde Y$ and  $z\in \mathbb{C}$.
   \end{proof}

\section{The Fourier-Stieltjes algebra and the Fourier algebra of a Fell bundle over a group}\label{FS}
Let $\B$ be a Fell bundle over a group $G$. We recall from \cite[Remark 3.17]{BeCo7} that one may 
associate  a Fourier-Stieltjes algebra $B(\B)$ to $\B$ as follows. The set $L(\B)$ of all $\B$-bundle maps becomes an algebra with respect to pointwise operations, i.e.,  
\[(T_g)_{g\in G} + (T'_g)_{g\in G} = (T_g+T'_g)_{g\in G}, \, \lambda\cdot  (T_g)_{g\in G} = (\lambda T_g)_{g\in G} \text{ and } (T_g)_{g\in G} \cdot (T'_g)_{g\in G}= (T_g\circ T'_g)_{g\in G}.\]
 Let $P(\B)$ be the subset of $L(\B)$ consisting of all positive definite $\B$-bundle maps. Then  \[ B(\B) := {\rm Span}(P(\B))\] 
is a subalgebra of $L(\B)$. The fact that $\B(B)$ is closed under multiplication can be seen by using that the set of completely positive maps on a $C^*$-algebra is closed under composition and applying \cite[Corollary 3.18 and Remark 3.19]{BeCo7}.
(When $\B$ is unital this can also be deduced by combining \cite[Theorem 5.1]{BeCo8} with Proposition \ref{comp-pos-def}.) 
Both algebras $L(\B)$ and $B(\B)$  are unital, with unit $I=(I_g)_{g\in G}$, where $I_g:B_g\to B_g$ denotes the identity map
for each $g\in G$. We also note that $P(\B)$ is a cone in $L(\B)$ such that $P(\B)\cap (-P(\B)) = \{0\}$. 
 
When $\B$ is unital, the  Fourier-Stieltjes algebra $B(\B)$ may also be described in terms of coefficient functions of $\B$-actions, similarly to the case of Fourier-Stieltjes algebras associated to unital $C^*$-dynamical systems, cf.~\cite{BeCo6}.  
We first need to introduce some notation.  When  $\rho$ is a $\B$-action on a Hilbert $\B$-bundle $\X$ and  $x, y\in X_e$, we will denote by $T^{\rho, x, y}$ the $\B$-bundle map defined by \[ T_g^{\rho, x, y}(b) = \langle x, \rho(b) y\rangle_\B\]
for all $g\in G$ and $b\in B_g$, and refer to such maps as coefficients of $\rho$ (over the unit fibre). Then we have:

\begin{theorem} \label{FS-coeff} Assume $\B$ is unital.  Then
\[B(\B) =\big\{ T^{\rho, x, y}: \rho \text{ is an action of } \B \text{ on a Hilbert $\B$-bundle $\X$ and } x, y\in X_e\big\}.\]
\end{theorem}

\begin{proof} The inclusion $\supseteq$
follows from \cite[Proposition 4.5]{BeCo8} and the polarization identity
\[ 
\langle x, \rho(b) y \rangle_\B 
= \frac 1 4 \sum_{k=0}^3 (-i)^k \langle x + i^k y, \rho(b)(x + i^k y) \rangle_\B,
\]
which is valid for any action $\rho$ of  $\B$ on a Hilbert $\B$-bundle $\X$ and any $b \in \B, x, y\in X_e$.

To prove the opposite inclusion, assume $T \in B(\B)$, so $T= \sum_{i=1}^n \lambda_iT^i$ for some $\lambda_i \in \Complessi$ and $ T^i \in P(\B)$ for $i=1, \ldots , n$. Using \cite[Theorem 5.1]{BeCo8}, we can write $T^i = T ^{\rho_i, x_i, x_i}$ for some action $\rho_i$ on a Hilbert $\B$-bundle $\X^i$ and $x_i\in X^i_e$ for each $i=1, \ldots, n$. Then for all $g\in G$ and $b\in B_g$, we have
\begin{align*} T_g(b) &= \sum_{i=1}^n \lambda_i T^i_g(b) = \sum_{i=1}^n \lambda_i T ^{\rho_i, x_i, x_i}_g(b)
 = \sum_{i=1}^n \lambda_i \langle x_i, \rho_i(b) x_i\rangle_\B\\  &= \big\langle\oplus_{i=1}^n x_i, \big((\oplus_{i=1}^n \rho_i)(b)\big)(\oplus_{i=1}^n \lambda_ix_i) \big\rangle_\B
= T_g^{\oplus_{i=1}^n \rho_i,\, \oplus_{i=1}^n x_i,\, \oplus_{i=1}^n\lambda_ix_i }(b).
\end{align*} 
Thus $T = T^{\oplus_{i=1}^n \rho_i,\, \oplus_{i=1}^n x_i,\, \oplus_{i=1}^n\lambda_ix_i}$, which shows the desired inclusion.
\end{proof}
 
\begin{remark} If $\B$ is unital, then 
$I_g(b) =\langle 1_{B_e}, \rho_\B(b)1_{B_e}\rangle_\B = \langle 1_{B_e}, \rho_\B(1_{B_e})b\rangle_\B$ 
for all $g \in G$ and $b\in B_g$, where $\rho_\B$ is the trivial $\B$-action on $\B$. If $\B$ is not unital, it is not clear %to us 
whether $I = (I_g)_{g\in G}$ can 
always be written as a coefficient of some $\B$-action (over the unit fibre). Anyhow, if $(e_j)_{j\in J} $ is a bounded approximate unit for $B_e$, we 
have $I_g(b) = \lim_{j} \langle e_j , \rho_\B(b) e_j\rangle_\B= \lim_{j} \langle e_j , \rho_\B(e_j)b\rangle_\B$ for all $g \in G$ and $b\in B_g$. 
\end{remark}
For $j=1, 2$, let $\rho_j$ be a $\B$-action on a Hilbert $\B$-bundle $\X^j$ and $x_j, y_j\in X^j_e$. Following our previous convention,  we identify $x_2\dot\otimes x_1$ and $y_2\dot\otimes y_1$ with their canonical images in the unit fibre of $\X^2\otimes_{\rho_1}\X^1$. Then we have 
\begin{equation}\label{TensorCoeff} T^{\rho_1, x_1, y_1}\cdot T^{\rho_2, x_2, y_2} = T^{\rho_2\dot\otimes\rho_1, x_2\dot\otimes x_1, y_2\dot \otimes y_1}.\end{equation}
Indeed, let $g\in G, b\in B_g$. Then 
\begin{align*} 
T_g^{\rho_2\dot\otimes\rho_1, x_2\dot\otimes x_1, y_2\dot \otimes y_1}(b) 
&=\langle x_2\dot\otimes x_1, [(\rho_2\dot\otimes\rho_1)(b)](y_2\dot \otimes y_1)\rangle_\B
= \langle x_2\dot\otimes x_1, (\rho_2(b)y_2)\dot \otimes y_1\rangle_\B\\
&= \big\langle x_1, \rho_1(\langle x_2, \rho_2(b)y_2\rangle_\B) y_1\big\rangle_\B
=\big\langle x_1, \rho_1\big(T_g^{\rho_2, x_2, y_2}(b)\big) y_1\big\rangle_\B\\
&= T_g^{\rho_1, x_1, y_1}\big(T_g^{\rho_2, x_2, y_2}(b)\big)
= (T^{\rho_1, x_1, y_1}\cdot T^{\rho_2, x_2, y_2})_g(b)
\end{align*}
When $\B$ is unital,  using Proposition \ref{FS-coeff} in combination with  formula (\ref{TensorCoeff}), we get another proof that $B(\B)$ is closed under product. 
Moreover, we may introduce 
\[ \|T\|:= \inf\big\{\|x\|\|y\|: T= T^{\rho,x,y} \text{ for some $\B$-action $\rho$ on some $\X$ and $x, y\in X_e$}\big\}\]
for each $T\in B(\B)$. Then we have: 
\begin{theorem}\label{FS-Banach}
Assume $\B$ is unital. Then
the Fourier-Stieltjes algebra $B(\B)$ is a unital Banach algebra w.r.t.~$\|\cdot\|$. 
\end{theorem}
\begin{proof} 
The proof is very similar to the proof of the corresponding statement in \cite[Proposition 3.1]{BeCo6}. Completeness follows from the fact that one may form infinite direct sums of actions, cf.~Section \ref{DirectSum}, while submultiplicativity of the norm follows by using (\ref{TensorCoeff}). 
\end{proof}

As usual, we let $B(G)$ denote the Fourier-Stieltjes algebra of $G$, cf.~\cite{KL}.  
\begin{proposition}\label{embFS(G)}
Assume $\B$ is unital. For each $\phi \in B(G)$, let $T^\phi$ be the $\B$-bundle map given by 
\[ T_g^\phi(a) = \phi(g)\,a \quad \text{for all } g\in G \text{ and } a\in B_g.\]
Then $T^\phi \in B(\B)$ and the map $\phi\mapsto T^\phi$ is an injective, contractive algebra-homomorphism of $B(G)$ into $\B(B)$.
\end{proposition}
\begin{proof} Let $\phi\in B(G)$. Then $\phi(g) = \langle \xi', U(g)\eta'\rangle$ for some unitary representation $U$ of $G$ on some Hilbert space $\H$ and $\xi', \eta'\in \H$. Also, $\|\phi\|$ is the infimum of $\|\xi'\|\, \|\eta'\|$ over all such possible decompositions of $\phi$. 

Now, set $\H_\B:= (\H \otimes B_g)_{g\in G}$ where, for each $g\in G$, $\H\otimes B_g$ denotes the external tensor product of the Hilbert $\Complessi$-module $\H$ with the Hilbert $B_e$-module $B_g$. One easily sees that $\H_\B$ is a Hilbert $\B$-bundle, obtained as  the completion of the right inner product $\B$-bundle $(\H\odot B_g)_{g\in G}$ equipped  with the operations determined by
\[ (\xi\otimes b)\cdot c = \xi\otimes bc \,\text{ and }\, \big\langle \xi\otimes b, \eta\otimes c\big\rangle_\B= \langle \xi, \eta\rangle\,  b^*c\]
for all $\xi, \eta \in \H$ and $b, c\in \B$.  Moreover, we get a $\B$-action $\rho_U$ on $\H_\B$,  determined by
\[ \rho_U(a)(\xi\otimes b):= U_g(\xi)\otimes ab \quad \text{for } g, h \in G, a\in B_g, b\in B_h, \text{ and } \xi \in \H.\]
Then we have
\begin{align*} 
T_g^{\,\rho_U,\, \xi'\otimes 1_{B_e},\, \eta'\otimes 1_{B_e}}(a)&= \big\langle \xi'\otimes 1_{B_e}, \rho_U(a)(\eta'\otimes 1_{B_e})\big\rangle_\B 
= \big\langle \xi'\otimes 1_{B_e}, U_g(\eta')\otimes a\big\rangle_\B \\
&= 
\langle \xi', U_g(\eta')\rangle\, a = \phi(g) \,a = T_g^\phi(a)
\end{align*}
for all $g\in G$ and $a\in B_g$. Thus,  $T^\phi = T^{\,\rho_U,\, \xi'\otimes 1_{B_e},\, \eta'\otimes 1_{B_e}} \in B(\B)$ by Proposition \ref{FS-coeff}. Further, 
\[\|T^\phi\| \, \leq \, \|\xi'\otimes 1_{B_e}\|\, \|\eta'\otimes 1_{B_e}\| = \|\xi'\|\, \|\eta'\|.\]
Taking the infimum over all possible decompositions of $\phi$, we get $\|T^\phi\|\leq \|\phi\|$. 
The remaining assertions can be routinely checked.    
\end{proof}

\medskip Classically, one way to define the Fourier algebra $A(G)$ is to let $A(G)$ be the closed span of all positive definite complex functions on $G$ having finite support. 
Recalling that a $\B$-bundle map $T=(T_g)_{g\in G}$ 
has \emph{finite support} when $T_g= 0$ for all but finitely many $g$'s in $G$, we define analogously  the \emph{Fourier algebra of a Fell bundle $\B=\{B_g\}_{g\in G}$} by 
\[A(\B):= \overline{{\rm span}\{T \in P(\B): T \text{ has finite support}\}}^{\|\cdot\|} \, \subseteq B(\B).\]

Then it can readily be checked that $A(\B)$ is a closed two-sided ideal of $B(\B)$. Now, there are other equivalent definitions of $A(G)$ (see e.g.~\cite{KL}) and an interesting problem is whether similar characterizations of $A(\B)$ can be obtained. 
Following Ravnanger's discussion \cite{Ravn} of the Fourier algebra associated to a unital discrete $C^*$-dynamical system, $A(\B)$ may be described as follows:
\begin{proposition} Assume $\B$ is unital. Then 
\[A(\B):= \overline{\{T \in B(\B): T \text{ has finite support}\}}^{\|\cdot\|}.\] 
\end{proposition}
\begin{proof}
Let $A'(\B)$ denote the right-hand side. As any linear combination of $\B$-bundle maps with finite support has finite support, it is obvious that $A(\B) \subseteq A'(\B)$.
The converse inclusion can be shown by adapting the analogous argument used to prove \cite[Theorem 3.3 and Corollary 3.4]{Ravn}. 

Indeed, suppose $T\in B(\B)$ has a non-empty finite support $S \subseteq G$. Using Proposition \ref{FS-coeff}, we can write $T=T^{\rho,x,y}$ for some $\B$-action $\rho$ on a Hilbert $\B$-bundle $\X=(X_r)_{r\in G}$ and some $x, y \in X_e$. Let then $\widecheck\rho$ denote the regular $\B$-action on the 
regular $\B$-bundle $\widecheck\X=(\widecheck X_r)_{r\in G}$ associated to $\X$ and $\rho$, cf.~\cite[Examples 3.4 and 4.11]{BeCo8}. 
We recall that $C_c(G, X_r)$ is dense in $\widecheck X_r$, $\langle \xi, \eta\rangle_\B = \sum_{t\in G} \langle\xi(t), \eta(t)\rangle_\B$ and $(\widecheck\rho(b)\xi)(h) = \rho(b) \xi(g^{-1}h)$ for all $r, s , g, h\in G$, $\xi \in C_c(G, X_r)$, $\eta \in C_c(G, X_s)$ and $b\in B_g$.
 
Then, defining $\widecheck{x}$ and $\widecheck{y}$ in $\widecheck X_e$ by 
\[ \widecheck x(t) = \begin{cases} x \quad &\text{if } t \in S,\\ \, 0 \quad &\text{otherwise} \end{cases} \quad \text{ and }  \quad
\widecheck y(t) = \begin{cases} y \quad &\text{if } t = e, \\ \, 0 \quad &\text{otherwise} \end{cases}\,,\]
 one easily computes that $T^{\widecheck\rho,\widecheck{x},\widecheck{y}} = T^{\rho, x, y}= T$.
Using polarization, we can write 
\[T=T^{\widecheck\rho,\widecheck{x},\widecheck{y}} = 
\frac 1 4 \sum_{k=0}^3 (-i)^kT^{\widecheck\rho,\,\widecheck{x} + i^k \widecheck{y},\,\widecheck{x} + i^k \widecheck{y}}.\]
Since $\widecheck{x} + i^k \widecheck{y}$ has finite support for each $k=0,\ldots,3$, it is easy to see that $T^{\widecheck\rho,\,\widecheck{x} + i^k \widecheck{y},\,\widecheck{x} + i^k \widecheck{y}}$ has finite support for each $k=0,\ldots,3$ as well. 
We have thus shown that every $T \in B(\B)$ with finite support can be written as a finite linear combination of elements in $P(\B)$ with finite support. Taking closures,  we get that $A'(\B) \subseteq A(\B)$.
\end{proof}

The proof above also implies that  $A(\B)$ is contained in the norm-closure of the set
\[ \{ T^{\widecheck\rho, \xi, \eta}  \in B(\B) :  \rho \text{ is a } \B\text{-action on } \X \text{ and } \xi, \eta \in \widecheck X_e\}.\] 
To show that equality holds, as it is the case when $\B$ is the Fell bundle associated to a unital discrete $C^*$-dynamical system, cf.~\cite[Theorem 3.3]{Ravn}, 
one will probably need to study $\rho\dot\otimes\widecheck\rho_\B$ versus $\widecheck\rho$, where $\rho$ is a $\B$-action, $\widecheck\rho_\B$ is the regular $\B$-action and $\widecheck\rho$ is the regular $\B$-action associated to $\rho$, cf.~\cite{BeCo8}.   
We also note that a possibly different definition of $A(\B)$ is proposed in \cite{BKMS} 
when $\B$ is the Fell bundle associated with a twisted action of an \'etale groupoid.  

\begin{remark}
%Needless to say, 
It readily follows from Proposition \ref{embFS(G)} 
that when $\B$ is unital we also get an injective, contractive algebra homomorphism from $A(G)$ into $A(\B)$. 
\end{remark}

%\medskip
We end up this section by discussing the relationship between the Fourier-Stieltjes algebra $B(\Sigma)$ of a unital discrete $C^*$-dynamical system $\Sigma=(A,G,\alpha)$ introduced in \cite{BeCo6} and the Fourier-Stieltjes algebra $B(\B_\Sigma)$ of the Fell bundle $\B_\Sigma$ canonically associated to $\Sigma$. 

The definition of $B(\Sigma)$ in \cite{BeCo6}  is as follows. Let $(\varrho, v)$ be an equivariant representation of  $\Sigma$ on a Hilbert $A$-module $X$, cf.~\cite{BeCo3, BeCo4, BeCo6}, and $x, y\in X$. Define
$\tilde T_{\varrho, v, x, y}: A\times G \to A$  by 
\[ \tilde T_{\varrho, v, x, y}(a, g) = \langle x, \varrho(a)v(g)y\rangle_A\]
for all $(a, g) \in A\times G$. Then  $B(\Sigma)$ consists of all possible maps of the form $\tilde T_{\varrho, v, x, y}$. 

\begin{proposition} Assume $\Sigma=(A,G,\alpha)$ is a unital discrete $C^*$-dynamical system. Then there is a contractive Banach algebra isomorphism $\theta$ from $B(\Sigma)$ onto $B(\B_\Sigma)$ such that
\begin{equation}\label{theta} \theta (\tilde T_{\varrho, v, x, y}) = T^{\rho, (x, e), (y, e)}\end{equation}
whenever $(\varrho, v)$ is an equivariant representation of  $\Sigma$ on a Hilbert $A$-module $X$, $x, y\in X$,   and $\rho=\rho_{(\varrho, v)}$ is 
 the corresponding $\B_\Sigma$-action on  the  Hilbert $\B_\Sigma$-bundle $\X=(X\times\{g\})_{g\in G}$ associated to $X$, cf.~Section \ref{tp-equiv}.
\end{proposition}
\begin{proof} 
Let $M(\Sigma)$ consists of all families  $(\tilde T_g)_{g\in G}$, where each $\tilde T_g$ is a bounded linear map from $A$ into itself, 
organized as an algebra with operations 
\[ (\tilde S_g)_{g\in G} + (\tilde T_g)_{g\in G} = (\tilde S_g +\tilde T_g)_{g\in G}, \, \lambda (\tilde T_g)_{g\in G} = (\lambda \tilde T_g)_{g\in G}, \, (\tilde S_g)_{g\in G}\cdot(\tilde T_g)_{g\in G} = (\tilde S_g \circ \tilde T_g)_{g\in G}.\]
Then it is routine to verify that the map  $\tilde T\mapsto (\tilde T_g)_{g\in G}$, where $\tilde T_g(a):= \tilde T(a,g)$ for all $g\in G$ and $a\in A$, gives an injective algebra-homomorphim $\Theta$ from  $B(\Sigma)$ into $M(\Sigma)$.
Moreover, the map $\Theta': M(\Sigma) \to L(\B_\Sigma)$ defined by  
\[\Theta'\big((\tilde T_g)_{g\in G}\big):= (T_g)_{g\in G},\]
where $(T_g)_{g\in G}$ is the $\B_\Sigma$-bundle map given by $T_g((a,g)) := (\tilde{T}_g(a),g)$ for all $(a,g)\in \B_\Sigma$, is easily seen to be an algebra-isomorphism. 
Thus we get an injective algebra homomorphism 
$\theta := \Theta'\circ\Theta: B(\Sigma) \to L(\B_\Sigma)$.
 
Now let $\widetilde T \in B(\Sigma)$.  Write $\widetilde T=  \tilde T_{\varrho, v, x, y}$ for some $(\varrho, v)$ and $x, y\in X$ as in the statement of the proposition and set $\rho:=\rho_{(\varrho, v)}$. Then it is almost immediate that  $\theta (\tilde T_{\varrho, v, x, y}) = T^{\rho, (x, e), (y, e)}$, i.e., (\ref{theta}) holds. Since $T^{\rho, (x, e), (y, e)} \in B(\B_\Sigma)$, we get that $\theta(B(\Sigma))\subseteq B(\B_\Sigma))$. We also get 
\[ \|\theta(\widetilde T)\| \leq \|(x, e)\|\, \|(y, e)\| = \|x\|\, \|y\|.\]
Taking the infimum over all possible $\varrho, v, x$ and $y$, we obtain
  $\|\theta(\widetilde T)\| \leq \|\widetilde T\|$.
  
  It only remains to show that $B(\B_\Sigma))\subseteq \theta(B(\Sigma))$. By \cite[Corollary 4.5]{BeCo6}, we have that  $B(\Sigma) = {\rm Span}\, P(\Sigma)$, where $P(\Sigma)$ consists of the elements of $B(\Sigma)$ that are positive definite (w.r.t.~$\Sigma$). By polarization, it therefore suffices to show that $P(\B_\Sigma)\subseteq \theta(P(\Sigma))$. Let 
  $T\in P(\B_\Sigma)$. For each $g\in G$, let then $\widetilde T_g:A\to A$ be the unique bounded linear map satisfying
  $T((a, g)) = (\widetilde T_g(a), g)$ for all $a\in A$, and  define $\widetilde T: A\times G\to A$ by  $\widetilde T (a, g) = \widetilde T_g(a)$
  for all $(a, g)\in A\times G$.  Then  $\widetilde T \in P(\Sigma)$
  (cf.~\cite[Example 4.6]{BeCo7} and \cite[Corollary 4.4]{BeCo6})
  and (almost tautologically) $T=\theta(\widetilde T)$.
\end{proof}
 
\section{Some questions} \label{final}
We mention a few questions that might be worth being investigated.

\medskip $\bullet$ Can the tensor product construction(s) in Section \ref{tensoraction} be put  in a categorical framework? Of course, here one should at least check the general validity of some (probably weak) form of  associativity and, in more specific situations, also  commutativity whenever it makes sense. 

\medskip 
$\bullet$ Let $\rho$ be an $(\A,\varphi)$-action on a Hilbert $\B$-bundle $\X$, as in Definition \ref{BC-action-def}. Recall from  \cite{BeCo8} that $\rho_\A$ denotes the trivial $\A$-action on the Hilbert $\A$-bundle $\A$.  
What can be said about $\rho_\A \dot\boxtimes \rho$, or $\rho \dot\boxtimes \rho_\B$? What about tensoring $\rho$ on the left with the regular action $\widecheck{\rho_\A}$, or on the right with $\widecheck{\rho_\B}$ ? (See \cite{BeCo8} for the definition of regular actions.) Do we get some kind of Fell absorption principle(s)?

\medskip 
$\bullet$ Let $\B$ be a Fell bundle over a group $G$. Can one define weak containment for $\B$-actions on Hilbert $\B$-bundles?
And is there a meaningful way to define the conjugate of a $\B$-action? 
If so, can one then define a concept of amenability for a $\B$-action, similar to what Bekka did for unitary representations of groups \cite{Bek}, for instance by saying that the trivial $\B$-action is weakly contained in the tensor product of the $\B$-action with its conjugate?

 \bigskip

\noindent {\bf Competing interests} \ The authors have no conflict of interest to declare that are relevant to the content of this article. 
There are no financial and non-financial conflicts of interest.

\medskip

\noindent{\bf Data availability} Data sharing is not applicable to this article as no new data have been created or analyzed in this research.

\bigskip
{\parindent=0pt Addresses of the authors:\\

\smallskip Erik B\'edos, Department of Mathematics, University of
Oslo, \\
P.B. 1053 Blindern, N-0316 Oslo, Norway.\\ E-mail: bedos@math.uio.no. \\

\smallskip \noindent
Roberto Conti, Dipartimento SBAI,
Sapienza Universit\`a di Roma \\
Via A. Scarpa 16,
I-00161 Roma, Italy.
\\ E-mail: roberto.conti@sbai.uniroma1.it
\par}

\end{document}